\documentclass{amsart}
\usepackage{mathrsfs}
\usepackage{amsfonts}
\usepackage{mathrsfs}
\usepackage{amssymb}
\usepackage{amsmath,amssymb}

\def \nin{\notin}
\newtheorem{theorem}{Theorem}[section]
\newtheorem{lemma}[theorem]{Lemma}

\theoremstyle{definition}
\newtheorem{definition}[theorem]{Definition}

\theoremstyle{remark}
\newtheorem{remark}[theorem]{Remark}

\numberwithin{equation}{section}

\begin{document}

\title[Well-posedness and regularity of  Generalized  Navier-Stokes equations]
{Well-posedness and regularity of Generalized Navier-Stokes
equations in some Critical $Q-$spaces}

\author{Pengtao Li}
\address{School of Mathematics, Peking University, Beijing, 100871, China}
\curraddr{Department of Mathematics and Statistics, Memorial
University of Newfoundland, St. John's, NL A1C 5S7, Canada}
\email{li\_ptao@163.com}

\author{Zhichun Zhai}
\address{Department of Mathematics and Statistics, Memorial University of Newfoundland, St. John's, NL A1C 5S7, Canada}
\curraddr{}
 \email{a64zz@mun.ca}
\thanks{Project supported in part  by Natural Science and
Engineering Research Council of Canada.}


\date{}


\begin{abstract} We study the well-posedness and  regularity  of
the generalized Navier-Stokes equations with  initial data
  in a  new critical space
 $Q_{\alpha;\infty}^{\beta,-1}(\mathbb{R}^{n})=\nabla\cdot(Q_{\alpha}^{\beta}(\mathbb{R}^{n}))^{n}, \beta\in(\frac{1}{2},1)$
  which is larger than some known critical homogeneous Besov spaces.
 Here  $Q_{\alpha}^{\beta}(\mathbb{R}^{n})$ is a space
  defined as the set of all measurable functions with
 $$\sup(l(I))^{2(\alpha+\beta-1)-n}\int_{I}\int_{I}\frac{|f(x)-f(y)|^{2}}{|x-y|^{n+2(\alpha-\beta+1)}}dxdy<\infty$$
 where  the supremum is taken over all cubes $I$ with the edge length $l(I)$
 and the edges parallel to the coordinate axes in $\mathbb{R}^{n}.$
 In order to study the well-posedness and regularity,
  we give a Carleson measure  characterization of
$Q_{\alpha}^{\beta}(\mathbb{R}^{n})$  by  investigating a new type
of tent spaces and an atomic decomposition of the predual for
$Q_{\alpha}^{\beta}(\mathbb{R}^{n}).$ In addition, our regularity
 results   apply to the incompressible Navier-Stokes
equations with initial data in $Q_{\alpha;\infty}^{1,
-1}(\mathbb{R}^{n}).$

\end{abstract}

\maketitle

\tableofcontents \pagenumbering{arabic}

 \vspace{0.1in}

 \section{Introduction}
This paper considers  the well-posedness and regularity of the
generalized
 Navier-Stokes equations on the half-space
$\mathbb{R}^{1+n}_{+}=(0,\infty)\times\mathbb{R}^{n},$ $n\geq 2:$
\begin{equation}\label{eq1e}
\left\{\begin{array} {l@{\quad \quad}l}
\partial_{t}u+(-\triangle)^{\beta}u+(u\cdot \nabla)u-\nabla p=0,
& \hbox{in}\ \mathbb{R}^{1+n}_{+};\\
\nabla \cdot u=0, & \hbox{in}\ \mathbb{R}^{1+n}_{+};\\
u|_{t=0}=a, &\hbox{in}\ \mathbb{R}^{n} \end{array} \right.
\end{equation}
with $\beta\in(1/2,1).$ Here $(-\triangle)^{\beta}$ is the
fractional Laplacian with respect to $x$ defined by
$$
\widehat{(-\triangle)^{\beta}u}(t,\xi)
 =|\xi|^{2\beta}\widehat{u}(t,\xi).
 $$

Note that the following scaling
 \begin{equation}\label{eq2e}
 u_{\lambda}(t,x)=\lambda^{2\beta-1}u(\lambda^{2\beta}t,\lambda x),
\ \
 p_{\lambda}(t,x)= \lambda^{4\beta-2}p(\lambda^{2\beta}t,\lambda x),\ \
 a_{\lambda}(x)= \lambda^{2\beta-1} a(\lambda x)
\end{equation}
 is particularly significant for  equations (\ref{eq1e}).

The fractional Laplacian operator appears in a wide class of
physical systems and engineering problems, including L$\acute{e}$vy
flights, stochastic interfaces and anomalous diffusion problems. In
fluid mechanics, the fractional Laplacian is often applied to
describe many complicated phenomenons via partial differential
equations.

When $\beta=1,$ equations (\ref{eq1e}) become the classical
Navier-Stokes equations. A natural approach in studying the
solutions   is to iterate the corresponding operator
$v\longrightarrow
e^{t\triangle}u_{0}-\int_{0}^{t}e^{(t-s)\triangle}P(u\otimes v)ds$
and to find a fixed point. This solution is called mild solution.
For the classical Navier-Stokes equations, this approach was
pioneered by Kato and Fujita.  The existence of mild solutions and
their regularity  have been established locally in time and global
for small initial data in various functional spaces,
 for example, see Kato \cite{T Kato}, Cannone
\cite{M Cannone}, Giga-Miyakawa \cite{Y Giga T Miyakawa},
Koch-Tataru \cite{H. Koch D. Tataru}, Germain-Pavlovic-Staffilani
\cite{P. Germain N. Pavlovic G.Staffilani}
 and the references therein.

For equations (\ref{eq1e}),   J. L. Lions
 \cite{J L Lions}  proved the
global  existence of  the classical solutions when $\beta\geq
\frac{5}{4}$ in dimensional $3.$
  Similar result holds  for general  dimension $n$ if $\beta\geq
\frac{1}{2}+\frac{n}{4},$ see Wu \cite{J Wu 1}.  For the important
case $\beta< \frac{1}{2}+\frac{n}{4},$  Wu in \cite{J Wu 2}-\cite{J
Wu 3} established the global existence  for equations (\ref{eq1e})
in the homogeneous Besov spaces
$\dot{B}^{1+\frac{n}{p}-2\beta}_{p,q}$
 for $1\leq q\leq \infty$ and for either $1/2<\beta$ and $p=2$ or $1/2<\beta\leq 1$
and  $2< p<\infty$ and  in $\dot{B}_{2,\infty}^{r}$
 with $r>\max\{1,1+\frac{n}{p}-2\beta\}.$ For   the corresponding  regularity
 criteria, we refer the readers to  Wu  \cite{J Wu 4}.

Our work originates mainly from two observations in mathematics. At
first, it is worth pointing out that most of the function spaces
listed above are critical spaces. A space is called critical for
equations (\ref{eq1e}) if it is invariant under the scaling
\begin{equation}\label{scl 2}
f_{\lambda}(x)=\lambda^{2\beta-1}f(\lambda x).
\end{equation}
 For
example, the spaces $\dot{L}^{2}_{\frac{n}{2}-1}, L^{n},
\dot{B}^{-1+\frac{n}{p}}_{p|{p<\infty},\infty}$ and $ BMO^{-1} $
 are critical spaces for $\beta=1.$
 For general positive $\beta,$ $\dot{B}_{2,1}^{1+\frac{n}{2}-2\beta}$ and
 $\dot{B}_{2,\infty}^{1+\frac{n}{2}-2\beta}$  are critical
spaces. We see that $\dot{L}^{2}_{\frac{n}{2}-1}\hookrightarrow
L^{n}\hookrightarrow\dot{B}^{-1+\frac{n}{p}}_{p|{p<\infty},\infty}\hookrightarrow
BMO^{-1}$.  $BMO^{-1}$ is the largest critical space for $\beta=1$
among the spaces listed above where such existence results are
available. This fact inspires us to find a larger critical space for
general $\beta>0$ which includes
$\dot{B}^{1+\frac{n}{2}-2\beta}_{2,1}$ and has a structure similar
to $BMO^{-1}$. With the small initial value in the space, the
corresponding global existence result holds.

On the other hand, in \cite{J. Xiao 1}, Xiao introduced a new space
$Q_{\alpha;\infty}^{-1}(\mathbb{R}^{n})$
 to replace $BMO^{-1}$ in  \cite{H. Koch D.
Tataru} and   generalized Koch-Tataru's global existence result for
the classical Navier-Stokes equations. Here
$Q_{\alpha;\infty}^{-1}(\mathbb{R}^{n})$ is
$\nabla\cdot(Q_{\alpha}(\mathbb{R}^{n}))^{n}$
 and  $Q_{\alpha}(\mathbb{R}^{n})$  is the space of all measurable
complex-valued functions $f$ on $\mathbb{R}^{n}$ satisfying
\begin{equation}
\|f\|_{Q_{\alpha}(\mathbb{R}^{n})}=\sup_{I}\left(l(I))^{2\alpha-n}\int_{I}\int_{I}\frac{|f(x)-f(y)|^{2}}{|x-y|^{n+2\alpha}}dxdy\right)^{1/2}<\infty
\end{equation}
where  the supremum is taken over all cubes $I$ with the edge length
$l(I)$
 and the edges parallel to the coordinate axes in $\mathbb{R}^{n}.$
It is easy to see that $Q_{\alpha;\infty}^{-1}$ and $BMO^{-1}$  are
invariant under the scaling $f(x)\mapsto \lambda f(\lambda x)$ which
is corresponding to the classical  Navier-Stokes equations. If we
want to generalize the results of Koch-Tataru \cite{H. Koch D.
Tataru} and Xiao \cite{J. Xiao 1} to the fractional case, we should
find a class of spaces such that their derivative spaces are
invariant under the scaling (\ref{scl 2}).

The above two observations suggest us introducing the following
spaces.
\begin{definition}\label{Q space} For $\alpha\in (-\infty,\beta)$  and $\beta\in (1/2,1),$  we define
$Q_{\alpha}^{\beta}(\mathbb{R}^{n})$ be the set of all measurable
complex-valued functions $f$ on $\mathbb{R}^{n}$ satisfying
\begin{equation}\label{def Q a b}
\|f\|_{Q_{\alpha}^{\beta}(\mathbb{R}^{n})}=\sup_{I}\left(l(I))^{2(\alpha+\beta-1)-n}\int_{I}
\int_{I}\frac{|f(x)-f(y)|^{2}}{|x-y|^{n+2(\alpha-\beta+1)}}dxdy\right)^{1/2}<\infty
\end{equation}
where  the supremum is taken over all cubes $I$ with the edge length
$l(I)$
 and the edges parallel to the coordinate axes in $\mathbb{R}^{n}.$

\end{definition}
\begin{remark}
Obviously, when $\beta=1$ and $\alpha<0,$ $Q_{\alpha}^{\beta}=BMO.$
Thus our well-posedness result  generalizes that of Koch-Tataru
\cite{H. Koch D. Tataru} and  Xiao  \cite{J. Xiao 1} to the
fractional equations. It is easy to show that
$Q_{\alpha}^{\beta}(\mathbb{R}^{n})$ is invariant in the following
sense: for any $f\in Q_{\alpha}^{\beta}(\mathbb{R}^{n})$,
$$\|f(\cdot)\|_{Q_{\alpha}^{\beta}(\mathbb{R}^{n})}=\|\lambda^{2\beta-2}f(\lambda\cdot+x_{0})\|_{Q_{\alpha}^{\beta}(\mathbb{R}^{n})},\quad
 \lambda>0\ \hbox{and}\ x_{0}\in \mathbb{R}^{n}.$$
\end{remark}
Xiao in \cite{J. Xiao 1} characterized  $Q_{\alpha}(\mathbb{R}^{n})$
 equivalently as
\begin{equation}\label{semi}
\|f\|^{2}_{Q_{\alpha}(\mathbb{R}^{n})}=\sup_{x\in\mathbb{R}^{n},
r\in (0,\infty)}r^{2\alpha-n} \int_{0}^{r^{2}}\int_{|y-x|<r}|\nabla
e^{t\triangle} f(y)|^{2}t^{-\alpha}dtdy<\infty.
\end{equation}
 The advantage of this equivalent characterization is the occurrence of  $e^{t\triangle}$
    which generates the mild  solutions
 for the classical Navier-Stokes equations.
Note that  the mild  solutions for equations (\ref{eq1e}) can be
represented as
 $$u(t,x)=e^{-t(-\triangle)^{\beta}}a(x)-\int_{0}^{t}e^{-(t-s)(-\triangle)^{\beta}}P\nabla(u\otimes u)ds,$$
 where
$$e^{-t(-\triangle)^{\beta}}f(x)
:=K_{t}^\beta(x)\ast f(x)\ \text{ with }
 \widehat{K_{t}^{\beta}}(\xi)=e^{-t|\xi|^{2\beta}}$$
 and $P$ is the Helmboltz-Weyl projection:
 $$P=\{P_{j,k}\}_{j,k=1,\cdots,n}=\{\delta_{j,k}+R_{j}R_{k}\}_{j,k=1,\cdots,n}$$
 with $\delta_{j,k}$ being the Kronecker symbol and
 $R_{j}=\partial_{j}(-\triangle)^{-1/2}$ being the Riesz transform.
 Thus, we should characterize
  $Q_{\alpha}^{\beta}(\mathbb{R}^{n})$ by
  the semigroup $e^{-t(-\triangle)^{\beta}}.$
In fact,  we  prove  that $f\in Q_{\alpha}^{\beta}$ if and only if
 \begin{equation}\label{char 1}
 \sup_{x\in\mathbb{R}^{n},r\in(0,\infty)}r^{2\alpha-n+2\beta-2}
\int_{0}^{r^{2\beta}}\int_{|y-x|<r}|\nabla
e^{-t(-\triangle)^{\beta}}
f(y)|^{2}t^{-\frac{\alpha}{\beta}}dydt<\infty.
\end{equation}

It is easy to to check that for each $j=1,\cdots,n,$
$\partial_{j}K_{1}^{\beta}(x):=\phi_{j}(x)$ is a  $C^{\infty}$
real-valued function on $\mathbb{R}^{n}$  satisfying  the
properties:
\begin{equation}\label{prp psi}
\phi_{j}\in L^{1}(\mathbb{R}^{n}),\ |\phi_{j}(x)|\lesssim
(1+|x|)^{-(n+1)},\
 \int_{\mathbb{R}^{n}}\phi_{j}(x)dx=0\ \hbox{and}\
(\phi_{j})_{t}(x)=t^{-n}\phi_{j}\left(\frac{x}{t}\right).
\end{equation}
This observation leads us to characterize
$Q_{\alpha}^{\beta}(\mathbb{R}^{n})$ by a general $C^{\infty}$
real-valued function $\phi$ on $\mathbb{R}^{n}$ with the properties
(\ref{prp psi}) as
 \begin{equation}\label{eq cha q a b}
 f\in
Q_{\alpha}^{\beta}(\mathbb{R}^{n})\Longleftrightarrow \sup_{x\in
\mathbb{R}^{n},
r\in(0,\infty)}r^{2\alpha-n+2\beta-2}\int_{0}^{r}\int_{|y-x|<r}|f\ast
\phi_{t}(y)|^{2}t^{-(1+2(\alpha-\beta+1))}dtdy<\infty,
\end{equation}
 i.e.,  $d\mu_{f,\phi,\alpha,\beta}(t,x)=|f\ast \phi_{t}(x)|^{2}t^{-(1+2(\alpha-\beta+1))}dtdy$  is a
  $1-2(\alpha+\beta-1)/n-$ Carleson measure.

  In order to get  (\ref{eq cha q a b}), inspired by
 Coifman-Meyer-Stein \cite{Coifman Meyer Stein} and Dafni-Xiao \cite{G.
Dafni J. Xiao},  we  introduce new tent spaces
$T^{1}_{\alpha,\beta}$ and  $T^{\infty}_{\alpha,\beta},$  then
 define a space $HH^{1}_{-\alpha,\beta}(\mathbb{R}^{n})$ as
  a subspace of distributions in $\dot{L}^{2}_{-\frac{n}{2}+2(\beta-1)}.$
  Finally, we  identify $Q_{\alpha}^{\beta}(\mathbb{R}^{n})$  with
   the dual space of $HH^{1}_{-\alpha,\beta}(\mathbb{R}^{n}).$

In order to establish the equivalent norm (\ref{char 1}) of the
space $Q^{\beta}_{\alpha}$ in Definition \ref{Q space}, we need the
notation of Hausdorff capacity(see \cite{Adams} and \cite{D. Yang W.
Yuan}). By the definition of Hausdorff capacity, we should assume
$\alpha+\beta-1\geq0$ to guarantee
$\Lambda^{\infty}_{n-2(\alpha+\beta-1)}$ is meaningful. However if
we define the space $Q^{\beta}_{\alpha}$ by (\ref{char 1}) directly,
the proofs for the well-posedness and regularity still holds without
the restriction that $\alpha+\beta-1\geq0$.

 We now give the organization of this paper. In Section 2,   we
introduce some notation and  some  facts about homogeneous Besov
spaces,  Hausdorff capacity and Carleson measures. In Section 3, in
order to establish (\ref{eq cha q a b}), we introduce a new type of
tent spaces, their atomic decompositions and the predual space of
$Q_{\alpha}^{\beta}(\mathbb{R}^{n}).$ The proofs of the main
theorems in this section
  are similar to that of Dafni-Xiao \cite{G. Dafni J. Xiao}. For
the completeness, we provide the details.  In Section 4,  we
establish the well-posedness of the generalized Navier-Stokes
equations in  a new  critical space
$Q_{\alpha;\infty}^{\beta,-1}(\mathbb{R}^{n})$ which  is the
derivative spaces of $Q_{\alpha}^{\beta}(\mathbb{R}^{n})$ and
contains all known critical homogeneous Besov spaces for equations
(\ref{eq1e}). In Section 5, we establish the regularity of the
global solutions to equations (\ref{eq1e}) with the initial value in
$Q_{\alpha,\infty}^{\beta, -1}(\mathbb{R}^{n})$ for $\beta\in
(1/2,1].$

 In \cite{D. Yang W. Yuan}, D. Yang and W. Yuan introduced two new classes of
  spaces, i.e. $\dot{F}^{-s,\tau/q}_{p',q'}$ and
  $F\dot{H}^{s,\tau}_{p,q},$ and studied their dual relation. It's
  worth mentioning that when dealing with the duality relation
   our method is different from that in \cite{D. Yang W. Yuan}.
  Because when $p\neq q,$ the atomic decomposition of the tent space
  $F\dot{T}^{-s,\tau/q}_{p,q}$ is not completely known, the  authors do not
  invoke it in \cite{D. Yang W. Yuan}. For our case, because  $Q_{\alpha}^{\beta}=\dot{F}^{\alpha-\beta+1,\frac{1}{2}-\frac{\alpha+\beta-1}{n}}_{2,2},$
    we can apply the atomic
  decompositions of the tent space $T^{1}_{\alpha,\beta}$ and Hardy-Hausdorff space $HH^{1}_{-\alpha,\beta}$
to prove the dual relation between $Q_{\alpha}^{\beta}$ and
$HH^{1}_{-\alpha,\beta}.$ See also \cite{D. Yang W. Yuan}, Remark
5.2, Page 2080.

\vspace{0.1in} \noindent
 {\bf{Acknowledgements.}} We would like to thank our  supervisor Professor Jie Xiao
  for suggesting the problem and kind  encouragement.

\section{Notation and Preliminaries}
In this paper the symbols $\mathbb{C},\mathbb{Z} $  and $\mathbb{N}$
 denote the sets of all complex numbers, integers and natural
 numbers, respectively. For $n\in \mathbb{N},$ $\mathbb{R}^{n}$ is
 the $n-$dimensional Euclidean space, with Euclidean norm dented by
 $|x|$ and the Lebesgue measure by $dx.$ $\mathbb{R}^{1+n}_{+}$ is
 the  upper half-space $\{(t,x)\in \mathbb{R}^{1+n}_{+}: t>0, x\in \mathbb{R}^{n}\}$
 with Lebesgue measure denoted by $dtdx.$

A ball in $\mathbb{R}^{n}$ with center $x$ and radius $r$ will be
  denoted by $B=B(x,r),$  its Lebesgue measure by $|B|.$ 
   A cube in $\mathbb{R}^{n}$ will always mean a cube in $\mathbb{R}^{n}$ with side parallel to the coordinate axes.
   The sidelength of a cube $I$ will be denoted by $l(I).$
   Similarly,  its  volume
    will be  denoted by $|I|.$

  The symbol $U\lesssim V$ means that there exists a positive
  constant  $C$ such that $U\leq
  CV$. $U\approx V$ means $U\lesssim V$ and $V\lesssim U.$ For convenience, the positive constants $C$
  may change from one line to another and usually depend
  on the dimension $n,$ $\alpha,$ $\beta$
  and other fixed parameters.

   The characteristic function of a set $A$ will be denoted by
   $1_{A}.$ For  $\Omega\subset \mathbb{R}^{n},$ the space
   $C_{0}^{\infty}(\Omega)$ consists of all smooth functions with
   compact support in $\Omega.$ The Schwartz class of rapidly
   decreasing functions and its dual will be denoted by
   $\mathscr{S}(\mathbb{R}^{n})$ and $\mathscr{S}'(\mathbb{R}^{n}),$ respectively.
    For a
   function $f\in \mathscr{S}(\mathbb{R}^{n}),$  $\widehat{f}$
   means  the Fourier transform of $f.$

   We use $\mathscr{S}_{0}$ to denote the following subset of $\mathscr{S},$
   $$\mathscr{S}_{0}=\left\{\phi\in \mathscr{S}: \int_{\mathbb{R}^{n}}\psi(x)x^{\gamma}dx=0,|\gamma|=0,1,2,\cdots\right\},$$
where $x^{\gamma}=x_{1}^{\gamma^{1}}x_{2}^{\gamma^{2}} \cdots
x_{n}^{\gamma^{n}},|\gamma|=\gamma_{1}+\gamma_{2}+\cdots+\gamma_{n}.$
Its dual
 $\mathscr{S}'_{0}=\mathscr{S}'/\mathscr{S}^{\bot}_{0}=\mathscr{S}'/\mathcal {P},$
 where $\mathcal {P}$ is the space of multinomials.

We introduce a dyadic partition of $\mathbb{R}^{n}.$ For each
$j\in\mathbb{Z},$ we let
$$D_{j}=\{\xi\in \mathbb{R}^{n}: 2^{j-1}<|\xi|\leq 2^{j+1}\}.$$
We choose $\phi_{0}\in\mathscr{S}(\mathbb{R}^{n})$ such that
$\hbox{supp}(\phi_{0})=\{\xi:2^{-1}<|\xi|\leq 2\}\ \hbox{and}\
 \phi_{0}>0 \ \hbox{on}\ D_{0}.$
Let
$$\phi_{j}(\xi)=\phi_{0}(2^{-j}\xi)\quad \hbox{and}\quad \widehat{\Psi_{j}}(\xi)=\frac{\phi_{j}(\xi)}{\sum_{j}\phi_{j}(\xi)}.$$
Then $\Psi_{j}\in\mathscr{S}$ and
$\widehat{\Psi_{j}}(\xi)=\widehat{\Psi_{0}}(2^{-j}\xi), \
\hbox{supp}(\widehat{\Psi_{j}})\subset D_{j},
 \Psi_{j}(x)=2^{jn}\Psi_{0}(2^{j}x).$
 Moreover,
 \begin{equation}\label{eq2e}
\sum_{k=-\infty}^{\infty}\widehat{\Psi_{k}}(\xi)=\left\{\begin{array}
{l@{\quad \quad}l}
 1,\ \hbox{if}\ \xi\in \mathbb{R}^{n}\backslash \{0\}, \\
 0,\ \hbox{if}\ \xi=0.
 \end{array}
  \right.
\end{equation}
To define the homogeneous Besov spaces, we let
$$\Delta_{j}f=\Psi_{j}\ast f, j=0,\pm1,\pm2,\cdots.$$
For $s\in \mathbb{R}^{n}$ and $1\leq p,q\leq \infty,$ we define
homogeneous Besov spaces $\dot{B}^{s}_{p,q}$ be the set of all $f\in
\mathscr{S}'_{0}$ with
$$\|f\|_{\dot{B}^{s}_{p,q}}=\left(\sum_{j=-\infty}^{\infty}(2^{js}\|\Delta_{j}f\|_{L^{p}})^{q}\right)^{1/q}<\infty \quad \hbox{for}\ q<\infty,$$
$$\|f\|_{\dot{B}^{s}_{p,q}}=\sup_{-\infty<j<\infty}2^{js}\|\Delta_{j}f\|_{L^{p}}<\infty \quad \hbox{for}\ q=\infty.$$
When  $0<s<1,$ we have the following equivalent characterization. If
$1\leq p,q<\infty$, then $f\in\dot{B}^{s}_{p,q}$ is equivalent to
\begin{equation}\label{De Besov}
\int_{\mathbb{R}^{n}}\left(\int_{\mathbb{R}^{n}}|f(x+y)-f(x)|^{p}dx\right)^{q/p}\frac{dy}{|y|^{n+qs}}<\infty;
\end{equation}
if $0<s<1$ and $1\leq p<q=\infty,$ $f\in\dot{B}^{s}_{p,q}$ amounts
to
\begin{equation}\label{de 2 Besov}
\sup_{y\in\mathbb{R}^{n}}|y|^{-s}\left(\int_{\mathbb{R}^{n}}|f(x+y)-f(x)|^{p}\right)^{1/p}<\infty.
\end{equation}
We refer to Peetre \cite{J Peetre} and Triebel \cite{H. Triebel} for
more information. The usual homogeneous Sobolev space
$\dot{L}^{2}_{s}$ defined by
$\dot{L}^{2}_{s}=(-\triangle)^{-s/2}L^{2}$ is a special type of the
homogeneous Besov space. That is,
$\dot{B}^{s}_{2,2}=\dot{L}^{2}_{s}.$

The homogenous Besov spaces obey the following inclusion relations
 (see \cite{Berg Lofstrom}). 
\begin{theorem}\label{besov emdeing th}
 Let $s\in \mathbb{R}$ and $p,q\in[1,\infty].$\\
(i) If $1\leq q_{1}\leq q_{2}\leq \infty,$ then
$\dot{B}^{s}_{p,q_{1}}(\mathbb{R}^{n})\subseteq\dot{B}^{s}_{p,q_{2}}(\mathbb{R}^{n});$\\
(ii) If $1\leq p_{1}\leq p_{2}\leq \infty$ and
$s_{1}=s_{2}+n\left(\frac{1}{p_{1}}-\frac{1}{p_{2}}\right)$, then
$\dot{B}^{s_{1}}_{p_{1},q}\subseteq\dot{B}^{s_{2}}_{p_{2},q}.$
\end{theorem}
We recall the definition of fractional Carleson measures (see
Essen-Janson-Peng-Xiao \cite{M. Essen S. Janson L. Peng J. Xiao})
and their connection with  Hausdorff capacity established by
Dafni-Xiao in
 \cite{G. Dafni J. Xiao}.

\begin{definition}
For $p>0,$ we say that a Borel measure $\mu$ on
$\mathbb{R}^{1+n}_{+}$ is a $p-$Carleson measure provided that
\begin{equation}\label{norm carleson measure}
|\|\mu|\|_{p}=\sup\frac{\mu(S(I))}{(l(I))^{np}}<\infty
\end{equation}
 where the supremun is taken over all Carleson boxes
$S(I)=\{(t,x): x\in I, t\in (0,l(I))\}.$
\end{definition}

Obviously, the $1-$Carleson measures are the usual Carleson
measures. On the other hand, similar to the case $p=1,$ if we denote
by
$$T(E)=\{(t,x)\in \mathbb{R}^{1+n}_{+}: B(x,t)\subset E\}$$
 the tent based on the set $E\subset \mathbb{R}^{n},$  then a Borel measure $\mu$ on $\mathbb{R}^{1+n}_{+}$ is a
$p-$Carleson measre if and only if $|\mu|(T(B))\leq C|B|^{p}$ holds
for all balls $B\subset \mathbb{R}^{n}.$   That is to say
$p-$Carleson measures can be equivalently defined in terms of tents
over balls.

We recall some  definitions and properties  about   Hausdorff
capacity (see Adams \cite{Adams}, Dafni-Xiao \cite{G. Dafni J. Xiao}
and Yang-Yuan \cite{D. Yang W. Yuan}).
\begin{definition}
Let $d\in (0,n]$ and $E\subset \mathbb{R}^{n}.$\\
 (i) The
$d-$dimensional Hausdorff capacity of $E$ is defined by
\begin{equation}\label{HC}
\Lambda_{d}^{(\infty)}(E);=\inf\left\{\sum\limits_{j}r_{j}^{d}:
E\subset \cup_{j=1}^{\infty}B(x_{j},r_{j})\right\},
\end{equation}
where the infimum is taken over all covers of $E$ by countable
families of open (closed) balls with radii $r_{j}.$\\

(ii) The capacity  $\widetilde{\Lambda}_{d}^{(\infty)}(E)$ in the
sense of Choquet  is defined by
$$\widetilde{\Lambda}_{d}^{(\infty)}(E):=\inf\left\{\sum_{j}l(I_{j})^{d}: E\subset \left(\cup_{j=1}^{\infty}I_{j}\right)^{o}\right\},$$
where $B^{o}$ denotes the interior of $B$ and the infimum ranges only over covers of $E$ by dyadic cubes.\\
(iii) For a function $f:\mathbb{R}^{n}\longrightarrow [0,\infty],$
 we define
$$\int_{\mathbb{R}^{n}}fd\Lambda_{d}^{(\infty)}:=\int_{0}^{\infty}\Lambda_{d}^{(\infty)}(\{x\in \mathbb{R}^{n}: f(x)>\lambda\})d\lambda.$$
\end{definition}

\begin{remark}
(i) $\Lambda_{d}^{(\infty)}$ is not a capacity in the sense of
Choquent. But, its dyadic counterpart
$\widetilde{\Lambda}_{d}^{(\infty)}$ is a capacity since it is
monotone, vanishes on the empty set, and satisfies the strong
subadditivity condition
$$\widetilde{\Lambda}_{d}^{(\infty)}(E_{1}\cup E_{2})+\widetilde{\Lambda}_{d}^{(\infty)}(E_{1}\cap E_{2})
\leq
\widetilde{\Lambda}_{d}^{(\infty)}(E_{1})+\widetilde{\Lambda}_{d}^{(\infty)}(E_{2}),$$
as well as the  continuity conditions  (see  Admas \cite{Adams}):
$$\widetilde{\Lambda}_{d}^{(\infty)}\left(\cap_{i}
K_{i}\right)=\lim\limits_{i\longrightarrow
\infty}\widetilde{\Lambda}_{d}^{(\infty)}(K_{i}),\   \{K_{i}\}\
\hbox{a decreasing sequence of compact sets},$$
$$\widetilde{\Lambda}_{d}^{(\infty)}\left(\cup_{i}
K_{i}\right)=\lim\limits_{i\longrightarrow
\infty}\widetilde{\Lambda}_{d}^{(\infty)}(K_{i}),\  \{K_{i}\}\
\hbox{an increasing sequence of sets}.$$
  (ii) There exist positive
 constants $C_{1}(n,d)$  and $C_{2}(n,d)$ such that
\begin{equation}\label{capar Capaci}
C_{1}(n,d)\Lambda_{d}^{(\infty)}(E)\leq
\widetilde{\Lambda}_{d}^{(\infty)}(E)\leq
C_{2}(n,d)\Lambda_{d}^{(\infty)}(E)\ \hbox{for all }\ E\subset
\mathbb{R}^{n}.
\end{equation}
(iii) The integral with respect to
$\widetilde{\Lambda}_{d}^{(\infty)}(E)$ satisfies  Fatou's lemma

\begin{equation}\label{Fatou lemma}
\int_{\mathbb{R}^{n}}\liminf
f_{n}d\widetilde{\Lambda}_{d}^{(\infty)}\leq \liminf
f_{n}\int_{\mathbb{R}^{n}}d\widetilde{\Lambda}_{d}^{(\infty)}.
\end{equation}
\end{remark}

For $x\in\mathbb{R}^{n}$, let
$\Gamma(x)=\{(y,t)\in\mathbb{R}^{n+1}_{+}: |y-x|<t\}$ be the cone at
$x$. Define the nontangential maximal function $N(f)$ of a
measurable function on $\mathbb{R}^{n+1}_{+}$ by
$$N(f)(x):=\sup_{(y,t)\in\Gamma(x)}|f(y,t)|.$$
In  \cite{G. Dafni J. Xiao}, Dafni-Xiao characterized the fractional
Carleson measures as follows.
\begin{theorem}(\cite[Theorem 4.2]{G. Dafni J. Xiao})
Let $d\in (0,n]$ and $\mu$ be a Borel measure on
$\mathbb{R}^{1+n}_{+}.$ Then $\mu$ is a $d/n-$Carleson measure if
and only if the inequality
\begin{equation}\label{eq46}
\int_{\mathbb{R}^{1+n}_{+}}|f(t,y)|d|\mu|\leq
A\int_{\mathbb{R}^{n}}N(f)d\Lambda_{d}^{(\infty)}
\end{equation}
holds for all Borel measurable functions $f$ on
$\mathbb{R}^{1+n}_{+}.$ If this is the case then in (\ref{eq46}) the
constant $A\approx|\|\mu|\|_{d/n}$ which is defined by (\ref{norm
carleson measure}).
\end{theorem}


 \section{Carleson Measure Characterization of $Q_{\alpha}^{\beta}(\mathbb{R}^{n})$ } \label{sec 2}
In this section,  we  establish   the equivalent characterization
(\ref{eq cha q a b}). We first give some basic properties of
$Q_{\alpha}^{\beta}(\mathbb{R}^{n}).$ Then inspired by
 Coifman-Meyer-Stein \cite{Coifman Meyer Stein} and Dafni-Xiao \cite{G.
Dafni J. Xiao},    we introduce new tent spaces
$T^{1}_{\alpha,\beta}$ and $T^{\infty}_{\alpha,\beta}.$ Finally, we
obtain the predual space of $Q_{\alpha}^{\beta}(\mathbb{R}^{n}).$

\subsection{Basic Properties of $Q_{\alpha}^{\beta}(\mathbb{R}^{n})$}

\begin{lemma}\label{eq def}
Let $-\infty<\alpha$ and $\max\{\alpha,1/2\}<\beta<1.$ Then $f\in
Q_{\alpha}^{\beta}(\mathbb{R}^{n})$ if and only if
\begin{equation}\label{eqieq def}
\sup_{I}(l(I))^{-n+2(\alpha+\beta-1)}\int_{|y|<l(I)}\int_{I}\left|f(x+y)-f(x)\right|^{2}\frac{dxdy}{|y|^{n+2(\alpha-\beta+1)}}<\infty.
\end{equation}
\end{lemma}
\begin{proof}
If the double integrals (\ref{def Q a b}) and (\ref{eqieq def}) are
denoted by $U_{1}(I)$ and $U_{2}(I),$ respectively, then by the
change of  variable $y\longrightarrow x+y$ and simple geometry one
obtains $U_{1}(I)\leq U_{2}(\sqrt{n}I)$ and $U_{2}(I)\leq
U_{1}(3I).$
\end{proof}

\begin{theorem}\label{eqvali norm} Let $-\infty<\alpha$ and $\max\{\alpha,1/2\}<\beta<1.$ Then\\
 (i) $Q_{\alpha}^{\beta}(\mathbb{R}^{n})$ is
decreasing in $\alpha$ for a fixed $\beta,$ i.e.
$$Q_{\alpha_{1}}^{\beta}(\mathbb{R}^{n})\subseteq Q_{\alpha_{2}}^{\beta}(\mathbb{R}^{n}),
 \hbox{if}\ \alpha_{2}\leq \alpha_{1};$$
(ii) If $\alpha\in(-\infty,\beta-1),$ then
$$Q_{\alpha}^{\beta}(\mathbb{R}^{n})
=Q^{\beta}_{-\frac{n}{2}+\beta-1}(\mathbb{R}^{n}):=BMO^{\beta}(\mathbb{R}^{n}).$$
\end{theorem}

\begin{proof}
(i) Suppose $\alpha_{1}\leq\alpha_{2}$. If $f\in
Q^{\beta}_{\alpha_{2}}(\mathbb{R}^{n})$, then for any cube $I$ we
have
\begin{eqnarray*}
\int_{I}\int_{I}\frac{|f(x)-f(y)|^{2}}{|x-y|^{n+2\alpha_{1}-2\beta+2}}dxdy
&\lesssim&
[l(I)]^{2(\alpha_{2}-\alpha_{1})}\int_{I}\int_{I}\frac{|f(x)-f(y)|^{2}}{|x-y|^{n+2\alpha_{2}-2\beta+2}}dxdy\\
&\lesssim&
[l(I)]^{n-2\alpha_{1}-2\beta+2}\|f\|_{Q^{\beta}_{\alpha_{2}}}^{2}.
\end{eqnarray*}

(ii) We divide the discussion into two cases.

{\it Case 1:} $-\frac{n}{2}+\beta-1\leq\alpha<\beta-1.$ By (i), we
have $Q^{\beta}_{\alpha}(\mathbb{R}^{n})\subseteq
Q^{\beta}_{-\frac{n}{2}+\beta-1}(\mathbb{R}^{n})=BMO^{\beta}(\mathbb{R}^{n})$.
On the other hand, if $f\in BMO^{\beta}(\mathbb{R}^{n})$ and $I$ is
a cube, then for every $y\in\mathbb{R}^{n}$ with $|y|<l(I)$,
\begin{eqnarray*}
\int_{I}|f(x+y)-f(x)|^{2}dx
&\lesssim& \int_{I}\left(|f(x+y)-f(2I)|^{2}+|f(x)-f(2I)|^{2}\right)dx\\
&\lesssim&\int_{2I}|f(x)-f(2I)|^{2}dx\lesssim|I|^{1-\frac{4(\beta-1)}{n}}\|f\|^{2}_{BMO^{\beta}(\mathbb{R}^{n})},
\end{eqnarray*}
since we can get easily
\begin{eqnarray*}
\|f\|^{2}_{Q^{\beta}_{-\frac{n}{2}+\beta-1}}
&\approx&\sup_{I}|I|^{-1+\frac{4(\beta-1)}{n}}\int_{I}|f(x)-f(I)|^{2}dx
\end{eqnarray*}
with $f(I)=|I|^{-1}\int_{I}f(x)dx$ being the mean value of $f$ over
the cube $I.$

Hence
\begin{eqnarray*}
\int_{|y|<l(I)}\int_{I}\frac{|f(x+y)-f(x)|^{2}}{|y|^{n+2\alpha-\beta+2}}dxdy
&\leq&
[l(I)]^{n+4-4\beta}\|f\|_{BMO^{\beta}}^{2}\int_{|y|<l(I)}\frac{dy}{|y|^{n+2\alpha-2\beta+2}}\\
&\leq& [l(I)]^{n-2\alpha-2\beta+2}\|f\|_{BMO^{\beta}}^{2}.
\end{eqnarray*}
This tells us $BMO^{\beta}(\mathbb{R}^{n})\subseteq
Q_{\alpha}^{\beta}(\mathbb{R}^{n}).$

 {\it Case II:}
$\alpha\in(-\infty,-\frac{n}{2}+\beta-1]$. In this case,
$BMO^{\beta}(\mathbb{R}^{n})\subseteq
Q^{\beta}_{\alpha}(\mathbb{R}^{n}).$ If $f\in
Q^{\beta}_{\alpha}(\mathbb{R}^{n})$, let $I$ be a cube. If $x,y\in
I$, then the set $\{z\in I: \min(|x-z|,|y-z|)>\frac{1}{8}l(I)\}$ has
measure at least $\frac{1}{2}|I|$ and thus for
$-2\alpha-n+2\beta-2>0$,
\begin{eqnarray*}
\int_{I}\min\left\{|x-z|^{-2\alpha-n+2\beta-2},|y-z|^{-2\alpha-n+2\beta-2}\right\}dz
\geq C[l(I)]^{-2\alpha-n+2\beta-2}[l(I)]^{n}\geq
C[l(I)]^{-2\alpha+2\beta-2}.
\end{eqnarray*}
 Hence we can get
\begin{eqnarray*}
&&[l(I)]^{-2n+4\beta-4}\int_{I}\int_{I}|f(x)-f(y)|^{2}dxdy\\
&\leq&
[l(I)]^{2\alpha-2n+2\beta-2}\int_{I}\int_{I}\int_{I}|f(x)-f(y)|^{2}
\min\{|x-z|^{-2\alpha-n+2\beta-2},|y-z|^{-2\alpha-n+2\beta-2}\}dxdydz\\
&\lesssim&(l(I))^{-n+2(\alpha+\beta-1)}\int_{I}\int_{I}\frac{|f(x)-f(y)|^{2}}{|x-y|^{n+2(\alpha-\beta+1)}}dxdy\lesssim\|f\|^{2}_{Q_{\alpha}^{\beta}(\mathbb{R}^{n})}.
\end{eqnarray*}
Thus $Q_{\alpha}^{\beta}(\mathbb{R}^{n})\subseteq
BMO^{\beta}(\mathbb{R}^{n}).$
 This
completes the proof of Lemma \ref{eqvali norm}.
\end{proof}

In the following, we establish the connection between
$Q_{\alpha}^{\beta}(\mathbb{R}^{n})$ and homogeneous Besov spaces.

\begin{theorem}\label{connection Q B}
Let $n\geq 2$ and   $\max\{1/2,\alpha\}<\beta<1.$\\
(i) If $1\leq q\leq 2$ and $\alpha+\beta-1>0$, then
$\dot{B}^{\alpha-\beta+1}_{\frac{n}{\alpha+\beta-1},q}
(\mathbb{R}^{n})\subseteq
Q_{\alpha}^{\beta}(\mathbb{R}^{n}).$\\
(ii) Let  $1\leq q\leq \infty,$ $\gamma_{1}>(\alpha-\beta+1)$ and
$\gamma_{2}>0.$
  If
$\gamma_{1}-\gamma_{2}=2-2\beta,$ then
$\dot{B}^{\gamma_{1}}_{n/\gamma_{2},q}\subseteq
Q_{\alpha}^{\beta}(\mathbb{R}^{n}).$
\end{theorem}
\begin{remark} Similar results  hold for $\beta=1,$ see Essen-Janson-Peng-Xiao \cite[Theorem 2.7]{M. Essen S. Janson L. Peng J.
Xiao}.
\end{remark}
\begin{proof}
(i) It follows form  (i) of Theorem  \ref{besov emdeing th} that
$\dot{B}^{\alpha-\beta+1}_{\frac{n}{\alpha+\beta-1},q}(\mathbb{R}^{n})\subseteq
\dot{B}^{\alpha-\beta+1}_{\frac{n}{\alpha+\beta-1},2}(\mathbb{R}^{n}),$
so we can assume $q=2.$ For $f\in
\dot{B}^{\alpha-\beta+1}_{\frac{n}{\alpha+\beta-1},2}(\mathbb{R}^{n}).$
H\"{o}lder's inequality implies that for any cube $I$ in
$\mathbb{R}^{n},$
\begin{eqnarray*}
&&\int_{|y|<l(I)}\int_{I}|f(x+y)-f(x)|^{2}\frac{dy}{|y|^{2+2(\alpha-\beta+1)}}\\
&\lesssim&|I|^{(n-2(\alpha+\beta-1))/n}\int_{\mathbb{R}^{n}}
\left(\int_{I}|f(x+y)-f(x)|^{\frac{n}{(\alpha+\beta-1)}}dx\right)^{2(\alpha+\beta-1)/n}\frac{dy}{|y|^{n+2(\alpha-\beta+1)}}.
\end{eqnarray*}
This estimate, Lemma \ref{eq def} and (\ref{De Besov}) imply that
$f\in Q_{\alpha}^{\beta}(\mathbb{R}^{n}).$\\
 (ii) According to (ii) of Theorem  \ref{besov emdeing th}, we have
 $\dot{B}^{\gamma_{1}}_{n/\gamma_{2},q}\subset\dot{B}^{\theta_{1}}_{n/\theta_{2},q}$ for
  $\gamma_{1}>\theta_{1}$, $\gamma_{2}>\theta_{2}$ and $\gamma_{1}-\gamma_{2}=\theta_{1}-\theta_{2}$. Thus we can
  suppose that $\gamma_{1}<1$ and $\gamma_{2}<1.$
  Assume that $f\in
\dot{B}^{\gamma_{1}}_{n/\gamma_{2},q}\subseteq\dot{B}^{\gamma_{1}}_{n/\gamma_{2},\infty}.$
For any cube $I$ in $\mathbb{R}^{n},$     H\"{o}lder's inequality
and (\ref{de 2 Besov}) tell us
\begin{eqnarray*}
&&\int_{|y|<l(I)}\int_{I}|f(x+y)-f(x)|^{2}dx\frac{dy}{|y|^{n+2(\alpha-\beta+1)}}\\
&\lesssim&|I|^{(n-2\gamma_{2})/n}\int_{|y|<l(I)}\left(\int_{I}|f(x+y)-f(x)|^{n/\gamma_{2}}\right)^{2\gamma_{2}/n}\frac{dy}{|y|^{n+2(\alpha-\beta+1)}}\\
&\lesssim&(l(I))^{n-2\gamma_{2}}\int_{|y|<l(I)}|y|^{2\gamma_{1}-n-2(\alpha-\beta+1)}dy\|f\|^{2}_{\dot{B}^{\gamma_{1}}_{n/\gamma_{2},\infty}(\mathbb{R}^{n})}
\\
&\lesssim&
(l(I))^{n-2(\alpha+\beta-1)}\|f\|^{2}_{\dot{B}^{\gamma_{1}}_{n/\gamma_{2},\infty}(\mathbb{R}^{n})}.
\end{eqnarray*}
Thus $f\in Q_{\alpha}^{\beta}(\mathbb{R}^{n}).$
\end{proof}

\subsection{New Tent Spaces}
We introduce new tent spaces motivated by similar arguments in
Dafni-Xiao   \cite{G. Dafni J. Xiao}.

\begin{definition}
For $\alpha>0$ and  $\max\{1/2,\alpha\}<\beta<1$ with
$\alpha+\beta-1\geq 0$, we define $T_{\alpha,\beta}^{\infty}$ be the
class of all Lebesgue measurable functions $f$ on
$\mathbb{R}^{1+n}_{+}$ with
$$\|f\|_{T_{\alpha,\beta}^{\infty}}=\sup_{ B\subset \mathbb{R}^{n}}
\left(\frac{1}{|B|^{1-{2(\alpha+\beta-1)}/{n}}}\int_{T(B)}|f(t,y)|^{2}\frac{dtdy}{t^{1+2(\alpha-\beta+1)}}\right)^{1/2}<\infty,$$
where $B$ runs over all balls in $\mathbb{R}^{n}.$
\end{definition}
\begin{definition}
For $\alpha>0$ and  $\max\{1/2,\alpha\}<\beta<1$ with
$\alpha+\beta-1\geq 0$,  a function $a$ on $\mathbb{R}^{1+n}_{+}$ is
said to be a $T^{1}_{\alpha,\beta}-$atom provided there exists a
ball $B\subset \mathbb{R}^{n}$ such that $a$ is supported in the
tent $T(B)$ and satisfies
$$\int_{T(B)}|a(t,y)|^{2}\frac{dtdy}{t^{1-2(\alpha-\beta+1)}}\leq
\frac{1}{|B|^{1-{2(\alpha+\beta-1)}/{n}}}.$$\
\end{definition}

\begin{definition} For $\alpha>0$ and  $\max\{1/2,\alpha\}<\beta<1$ with $\alpha+\beta-1\geq0$ the space $T_{\alpha,\beta}^{1}$ consists of all
measurable functions $f$ on $\mathbb{R}^{1+n}_{+}$ with
$$\|f\|_{T^{1}_{\alpha,\beta}}=
\inf_{\omega}\left(\int_{\mathbb{R}^{1+n}_{+}}|f(t,x)|^{2}\omega^{-1}(t,x)\frac{dtdx}{t^{1-2(\alpha-\beta+1)}}\right)^{1/2}<\infty,$$
where the infimum is taken over all nonnegative Borel measurable
functions $\omega$ on $\mathbb{R}^{1+n}_{+}$ with
$$\int_{\mathbb{R}^{n}}N\omega d\Lambda^{\infty}_{n-2(\alpha+\beta-1)}\leq 1$$
and with the restriction that $\omega$ is allowed to vanish only
where $f$ vanishes.
\end{definition}

\begin{lemma} \label{equiva metric}
If $\sum_{j}\|g_{j}\|_{T_{\alpha,\beta}^{1}}<\infty,$ then
$g=\sum_{j}g_{j}\in T_{\alpha,\beta}^{1}$ with
$$\|g\|_{T_{\alpha,\beta}^{1}}\leq \sqrt{C_{1}^{-1}(n,d)C_{2}(n,d)}\sum_{j}\|g_{j}\|_{T_{\alpha,\beta}^{1}},$$
where $C_{1}(n,d), C_{2}(n,d)$ are the constants in (\ref{capar
Capaci}).
\end{lemma}

\begin{proof} The proof of this lemma is similar to that of Dafni-Xiao \cite[Lemma 5.3]{G. Dafni J.
Xiao}.
\end{proof}
\begin{theorem} \label{atomde T}
Let $\alpha>0$ and  $\max\{1/2,\alpha\}<\beta<1$ with $\alpha+\beta-1\geq 0$,  then\\
(i) $f\in T^{1}_{\alpha,\beta}$ if and only if there is a sequence
of $T^{1}_{\alpha,\beta}-$atoms $a_{j}$ and  an $l^{1}-$sequence
$\{\lambda_{j}\}$ such that $f=\sum_{j}\lambda_{j}a_{j}.$ Moreover
$$\|f\|_{T_{\alpha,\beta}^{1}}\approx\inf\left\{\sum_{j}|\lambda_{j}|: f=\sum_{j}\lambda_{j}a_{j}\right\}$$
where the infimum is taken over all possible atomic decompositions
of $f\in T^{1}_{\alpha,\beta}.$ The right-hand side thus defines a
norm on $T^{1}_{\alpha,\beta}$ which makes it into a Banach space.\\
(ii) The inequality
\begin{equation}\label{atom t ineq}
\int_{\mathbb{R}^{1+n}_{+}}|f(t,y)g(t,y)|\frac{dtdy}{t}\leq
C\|f\|_{T^{1}_{\alpha,\beta}}\|g\|_{T_{\alpha,\beta}^{\infty}}
\end{equation} holds for all $f\in T^{1}_{\alpha,\beta}$ and $g\in
T^{\infty}_{\alpha,\beta}.$\\
(iii) The Banach space  dual of $T_{\alpha,\beta}^{1}$ can be
identified with $T_{\alpha,\beta}^{\infty}$ under the following
pairing
$$\langle f,g\rangle=\int_{\mathbb{R}^{1+n}_{+}}f(t,y)g(t,y)\frac{dtdy}{t}.$$
\end{theorem}
\begin{proof}
(i) Let $a$ be a $T^{1}_{\alpha,\beta}$ atom. Then we can find a
ball $B=B(x_{B},r)\subset \mathbb{R}^{n}$ such that
$\hbox{supp}(a)\subset T(B)$ and
$$\int_{T(B)}|a(t,y)|^{2}\frac{dtdy}{t^{1-2(\alpha-\beta+1)}}\leq \frac{1}{|B|^{1-{2(\alpha+\beta-1)}/{n}}}.$$
Fix $\varepsilon >0$ and define
 $$\omega(t,x)=\kappa r^{-n+2(\alpha+\beta-1)} \min\left\{1,(\frac{r}{\sqrt{|x-x_{B}|^{2}+t^{2}}})^{n-2(\alpha+\beta-1)+\varepsilon}\right\},$$
 where $\sqrt{|x-x_{B}|^{2}+t^{2}}$ is the distance between $(t,x)$
 and $(0,x_{B}).$ 
For $x\in \mathbb{R}^{n},$ the distance in $\mathbb{R}^{1+n}_{+}$
from the cone $\Gamma(x)$ to $(0,x_{B})$ is
$\frac{|x-x_{B}|}{\sqrt{2}}.$ So
\begin{eqnarray*}
N\omega(x)&=&\sup_{(t,y)\in \Gamma(x)}\left|\kappa
r^{-n+2(\alpha+\beta-1)}
\min\left\{1,\left(\frac{r}{\sqrt{|x-x_{B}|^{2}+t^{2}}}\right)^{n-2(\alpha+\beta-1)+\varepsilon}\right\}\right|\\
&\leq&\kappa
r^{-n+2(\alpha+\beta-1)}\min\left\{1,(\frac{\sqrt{2}r}{{|x-x_{B}|}})^{n-2(\alpha+\beta-1)+\varepsilon}\right\}.
\end{eqnarray*}
Thus
\begin{eqnarray*}
\kappa^{-1}\int_{\mathbb{R}^{n}}N\omega
d\Lambda^{\infty}_{n-2(\alpha+\beta-1)}\leq
\int_{0}^{\infty}\Lambda^{\infty}_{n-2(\alpha+\beta-1)}(\{x:
N\omega(x)>\lambda\})d\lambda.
\end{eqnarray*}
If $\lambda<N\omega(x),$ then
$ |x-x_{B}|\leq
\sqrt{2}\left(\frac{r^{\varepsilon}}{\lambda}\right)^{\frac{1}{n-2(\alpha+\beta-1)+\varepsilon}}.$
 Meanwhile, $\lambda<N\omega(x)\leq \kappa
 r^{-n+2(\alpha+\beta-1)},$ so we obtain
\begin{eqnarray*}
\kappa^{-1}\int_{\mathbb{R}^{n}}N\omega
d\Lambda^{(\infty)}_{n-2(\alpha+\beta-1)}
\leq\int_{0}^{r^{-n+2(\alpha+\beta-1)}}\left(\frac{r^{\varepsilon}}{\lambda}\right)^{\frac{n-2(\alpha+\beta-1)}{n-2(\alpha+\beta-1)+\varepsilon}}d\lambda
\lesssim1.
\end{eqnarray*}
Morover, on $T(B)$ we have
$\omega^{-1}(t,x)=r^{n-2(\alpha+\beta-1)}.$
 By the definition of $T^{1}_{\alpha,\beta}-$atom, we get
 $$\int_{T(B)}|a(t,y)|^{2}\omega^{-1}(t,x)\frac{dtdx}{t^{1-2(\alpha-\beta+1)}}\lesssim 1.$$
 Thus $a\in T^{1}_{\alpha.\beta}$ with $\|a\|_{T_{\alpha,\beta}^{1}}\lesssim
 1.$ For any sum $\sum_{j}\lambda_{j}a_{j}$ with  $\|\{\lambda_{j}\}\|_{l_{1}}=\sum|\lambda_{j}|<\infty$
 and $T^{1}_{\alpha,\beta}-$atoms $a_{j}$,  Lemma \ref{equiva
 metric} implies that
 the sum converges in the quasi-norm to $f\in T_{\alpha,\beta}^{1}$
 with
 $\|f\|_{T^{1}_{\alpha,\beta}}\lesssim\sum_{j}|\lambda_{j}|.$

 Conversely, suppose that $f\in T^{1}_{\alpha,\beta}.$ There exists
 a Borel measurable function $\omega\geq 0$ on
 $\mathbb{R}^{1+n}_{+}$ such that
 $$\int_{\mathbb{R}^{1+n}_{+}}|f(t,x)|^{2}\omega^{-1}(t,x)\frac{dtdx}{t^{1-2(\alpha-\beta+1)}}\leq 2\|f\|^{2}_{T_{\alpha,\beta}^{1}}.$$
 For each $k\in \mathbb{Z},$ let $E_{k}=\{x\in \mathbb{R}^{n}:
 N\omega(x)>2^{k}\}.$  According to    Dafni-Xiao \cite[Lemma 4.1]{G. Dafni J. Xiao}, there exists
 a sequence of dyadic cubes $\{I_{j,k}\}$ with disjoint interiors
 such that
 $$\sum_{j}l(I_{j,k})^{n-2(\alpha+\beta-1)}\leq 2 \widetilde{\Lambda}^{(\infty)}_{n-2(\alpha+\beta-1)}(E_{k})\
 \text{   and    }
\ T(E_{k})\subset\cup_{j}S^{*}(I_{j,k}).$$
 Here we have used a
 Carleson box:
 $S^{*}(I_{j,k})=\{(t,y)\in \mathbb{R}^{1+n}_{+}: y\in I_{j,k}, t<2\hbox{diam}(I_{j,k})\}$
 to replace the tent $T(I^{*}_{j,k})$ over the dilated
 cube $I_{j,k}^{*}=5\sqrt{n}I_{j,k}.$
   Consequently, if  we define
 $T_{j,k}=S^{*}(I_{j,k})\backslash\cup_{m>k}\cup_{l}S^{*}(I_{l,m}),$
 these will have disjoint interiors for different values of $j$ or
 $k.$
    Now
 $$\cup_{k=-K}^{K}\cup_{j}T_{j,k}=\cup_{j}S^{*}(I_{j,-K})\backslash \cup_{m>K}\cup_{l}S^{*}(I_{l.m})\supseteq
 T(E_{-K})\backslash\cup_{m>K}\cup_{l}S^{*}(I_{l.m}).$$
 Similar to the discussion in the proof of Dafni-Xiao \cite[Theorem 5.4]{G. Dafni J. Xiao},   we have
 $$\cup_{k}\cup_{j}T_{j,k}\supseteq
 \cup_{k}T(E_{k})\backslash\cap_{k}\cup_{m>k}\cup_{l}S^{*}(I_{l,m})=\{(t,x)\in \mathbb{R}^{1+n}_{+}: \omega(t,x)>0\}\backslash T_{\infty}$$
with
$\Lambda_{n-2(\alpha+\beta-1)}^{(\infty)}(T_{\infty})=|T_{\infty}|=0.$
 Since $\omega$ is allowed
to vanish only where $f$ vanishes,  $f=\sum f1_{T_{j,k}}$ a.e. on
$\mathbb{R}^{1+n}_{+}.$
 Defining
 $a_{j,k}=f1_{T_{j,k}}(\lambda_{j,k})^{-1}$ and
$$\lambda_{j,k}=\left((l(I_{j,k}^{*}))^{n-2(\alpha+\beta-1)}\int_{T_{j,k}}|f(t,x)|^{2}\frac{dtdx}{t^{1-2(\alpha-\beta+1)}}\right)^{1/2},$$
we get $f=\sum_{j,k}\lambda_{j,k}a_{j,k}$ almost everywhere. Since
$S^{*}(I_{j,k})\subset T(B_{j,k})$ where $B_{j,k}$ is the ball with
the same center as $I_{j,k}$ and radius $l(I^{\ast}_{j,k})/2$.
$a_{j,k}$ is supported in $T(B_{j,k})$ and
\begin{eqnarray*}
&&\int_{T(B_{j,k})}|a_{j,k}(t,y)|^{2}\frac{dtdy}{t^{1-2(\alpha-\beta+1)}}\\
&\leq&(l(I_{j,k}^{*}))^{-n+2(\alpha+\beta-1)}
\left(\int_{T_{j,k}}|f(t,x)|^{2}\frac{dtdx}{t^{1-2(\alpha-\beta+1)}}\right)^{-1}
\left(\int_{T(B_{j,k})}|f(t,x)|^{2}\frac{dtdx}{t^{1-2(\alpha-\beta+1)}}\right)\\
&\leq&(l(I_{j,k}^{*}))^{-n+2(\alpha+\beta-1)}\leq
{|B_{j,k}|^{-1+{2(\alpha+\beta-1)}/{n}}}.
 \end{eqnarray*}
 Thus each $a_{j,k}$ is a $T^{1}_{\alpha,\beta}-$atom.

Next, we prove that $\{\lambda_{j,k}\}$ is $l^{1}-$summable. Noting
that $\omega\leq 2^{k+1}$ on $T_{j,k}\subset
(\cup_{l}S^{*}(I_{l,k+1}))^{c}\subset (T(E_{k+1}))^{c}$ and applying
the Cauchy-Schwarz inequality,  we obtain
\begin{eqnarray*}
\sum_{j,k}|\lambda_{j,k}|&\leq&\sum_{j,k}(l(I_{j,k}^{*}))^{\frac{n}{2}-(\alpha+\beta-1)}
\left(\int_{T_{j,k}}|f(t,x)|^{2}\frac{dtdx}{t^{1-2(\alpha-\beta+1)}}\right)^{1/2}\\
&\leq&\sum\limits_{j,k}\sup_{T_{j,k}}\omega^{1/2}(l(I^{*}_{j,k}))^{\frac{n}{2}-(\alpha+\beta-1)}\left(\int_{T_{j,k}}|f(t,x)|^{2}\omega^{-1}(t,x)
\frac{dtdx}{t^{1-2(\alpha-\beta+1)}}\right)^{1/2}\\
&\leq&\left(\sum\limits_{j,k}2^{(k+1)}(l(I^{*}_{j,k}))^{n-2(\alpha+\beta-1)}\right)^{1/2}\left(\sum_{j,k}\int_{T_{j,k}}|f(t,x)|^{2}\omega^{-1}(t,x)
\frac{dtdx}{t^{1-2(\alpha-\beta+1)}}\right)^{1/2}\\
&\lesssim&\|f\|_{T^{1}_{\alpha,\beta}}\left(\sum_{k}2^{k}\sum_{j}(l(I_{j,k}))^{n-2(\alpha+\beta-1)}\right)^{1/2}\\
&\lesssim&\|f\|_{T^{1}_{\alpha,\beta}}\left(\sum_{k}2^{k}\Lambda^{(\infty)}_{n-2(\alpha+\beta-1)}(E_{k})\right)^{1/2}\\
&\lesssim&\|f\|_{T^{1}_{\alpha,\beta}}\left(\int_{\mathbb{R}^{n}}N\omega
d
\Lambda^{(\infty)}_{n-2(\alpha+\beta-1)}\right)^{1/2}\lesssim\|f\|_{T^{1}_{\alpha,\beta}}.
\end{eqnarray*}

 Thus  $T^{1}_{\alpha,\beta}$ is  a Banach space since it is
 complete in the quasi-norm (Lemma \ref{equiva metric}) and
$$\|f\|_{T^{1}_{\alpha,\beta}}\approx |\|f|\|_{T^{1}_{\alpha,\beta}}=\inf\left\{\sum_{j}|\lambda_{j}|: f=\sum_{j}\lambda_{j}a_{j}\right\}$$
where the infimum is taken over all possible atomic decompositions
of $f\in T_{\alpha,\beta}^{1}$ and
$|\|\cdot|\|_{T^{1}_{\alpha,\beta }}$ is a norm.\\
(ii) Let  $\omega$ be a nonnegative Borel measurable function on
$\mathbb{R}^{1+n}_{+}$ satisfying $\int_{\mathbb{R}^{n}}N\omega
d\Lambda^{(\infty)}_{\alpha,\beta}\leq 1.$
 For  $g\in T^{\infty}_{\alpha,\beta},$
 $d\mu_{g,n-2(\alpha+\beta-1)}(t,x)=|g(t,x)|^{2}t^{-1-2(\alpha-\beta+1)}dtdx$
 is a $1-{2(\alpha+\beta-1)}/{n}-$Carleson measure.
Then  (\ref{eq46}) tells us, with
 $A\approx|\|\mu_{g,n-2(\alpha+\beta-1)}\||_{n-{2(\alpha+\beta-1)}/{n}}\approx\|g\|^{2}_{T^{\infty}_{\alpha,\beta}},$
 \begin{eqnarray*}\int_{\mathbb{R}^{1+n}_{+}}\omega(t,x)|g(t,x)|^{2}\frac{dtdx}{t^{1+2(\alpha-\beta+1)}}
 \lesssim
 \|g\|_{T_{\alpha,\beta}^{\infty}}^{2}\int_{\mathbb{R}^{n}}N\omega
 d\Lambda_{n-2(\alpha+\beta-1)}^{(\infty)}
 \lesssim\|g\|_{T_{\alpha,\beta}^{\infty}}^{2}.
 \end{eqnarray*}
 Thus  if $f\in T^{1}_{\alpha,\beta},$ then
 \begin{eqnarray*}
\int_{\mathbb{R}^{1+n}_{+}}|f(t,x)g(t,x)|\frac{dtdx}{t}
&\leq&\left(\int_{\mathbb{R}^{1+n}_{+}}|f(t,x)|^{2}\omega^{-1}(t,x)\frac{dtdx}{t^{1-2(\alpha-\beta+1)}}\right)^{1/2}
\|g\|_{T^{\infty}_{\alpha,\beta}}.\end{eqnarray*}
 Hence we  finish the proof of (ii) by  taking the
infimum on the right over all admissible $\omega.$
 \\
 (iii) Form (ii), we know that for every $g\in
 T^{\infty}_{\alpha,\beta},$ the pairing
 $$\langle f,g\rangle=\int_{\mathbb{R}^{1+n}_{+}}f(t,y)g(t,y)\frac{dtdy}{t}$$
 defines a bounded linear functional on $T^{1}_{\alpha,\beta}.$ Now
 we prove the converse. Let $L$ be a bounded linear functional
 on $T^{1}_{\alpha,\beta}.$ Fix a ball $B=B(x_{B},r)\subset
 \mathbb{R}^{n}.$ If $f$ is supported on $T(B)$ with
 $f\in L^{2}(T(B), t^{-1}dtdx)$ then
 \begin{eqnarray*}
\int_{T(B)}|f(t,x)|^{2}\frac{dtdx}{t^{1-2(\alpha-\beta+1)}}&\leq&
r^{2(\alpha-\beta+1)}\int_{T(B)}|f(t,x)|^{2}\frac{dtdx}{t}\\
&\lesssim&\frac{1}{|B|^{1-2(\alpha+\beta-1)/n}}r^{n-2(\alpha+\beta-1)+2(\alpha-\beta+1)}\int_{T(B)}|f(t,x)|^{2}\frac{dtdx}{t}\\
&\lesssim&
\frac{1}{|B|^{1-2(\alpha+\beta-1)/n}}r^{n-4\beta+4}\|f\|^{2}_{L^{2}(T(B),
t^{-1}dtdx)}.
 \end{eqnarray*}
This tells us that  $f(t,x)$ is a multiple of a
$T^{1}_{\alpha,\beta}-$atom and $L$ is a bounded linear functional
on $L^{2}(T(B),t^{-1}dtdx)$ which can be represented by the
inner-product with some function $g_{B}\in L^{2}(T(B), t^{-1}dtdx).$
Taking $B_{j}=B(0,j),$ $j\in\mathbb{N},$ then
$g_{B_{j}}=g_{B_{j+1}}$ on $T(B_{j})$. So we get a single function
$g$ on $\mathbb{R}^{1+n}_{+}$ that is locally in $L^{2}(t^{-1}dtdx)$
such that
$$L(f)=\int_{\mathbb{R}^{1+n}_{+}}f(t,x)g(t,x)\frac{dtdx}{t}$$\
whenever $f\in T^{1}_{\alpha,\beta}$ is  supported  in some tent
$T(B).$
 By the  atomic decomposition, the subset of such $f$ is dense in
$T^{1}_{\alpha,\beta}.$ We only need to prove $g\in
T^{\infty}_{\alpha,\beta}$ with
$\|g\|_{T^{\infty}_{\alpha,\beta}}\lesssim \|L\|.$

For a ball $B\subset \mathbb{R}^{n}$ and  every $\varepsilon>0,$ we
set
$$f_{\varepsilon}(t,x)=t^{-2(\alpha-\beta+1)}\overline{g(t,x)}1_{T^{\varepsilon}(B)}(t,x)$$
where $T^{\varepsilon}(B)$ is the truncated tent
$T(B)\bigcap\{(t,x): t>\varepsilon\}.$ Since  $g\in L^{2}(T(B)),$ we
have
\begin{eqnarray*}
\int_{T(B)}|f_{\varepsilon}(t,x)|^{2}\frac{dtdx}{t^{1-2(\alpha-\beta+1)}}
=\int_{T^{\varepsilon}(B)}|g(t,x)|^{2}\frac{dtdx}{t^{1+2(\alpha-\beta+1)}}\lesssim
\infty.
\end{eqnarray*}
Hence we can obtain  that  $f_{\varepsilon}$ is a multiple of a
$T^{1}_{\alpha,\beta}-$atom with
$$\|f_{\varepsilon}\|_{T^{1}_{\alpha,\beta}}^{2}\lesssim r^{n-2(\alpha+\beta-1)}\int_{T^{\varepsilon}(B)}|g(t,x)|^{2}\frac{dtdx}{t^{1+2(\alpha-\beta+1)}}.$$
 According to  the
representation above, we also get
\begin{eqnarray*}
\int_{T^{\varepsilon}(B)}|g(t,x)|^{2}\frac{dtdx}{t^{1+2(\alpha-\beta+1)}}
&\lesssim&
\|L\|\left(r^{n-2(\alpha+\beta-1)}\int_{T^{\varepsilon}(B)}|g(t,x)|^{2}\frac{dtdx}{t^{1+2(\alpha-\beta+1)}}\right)^{1/2}.
\end{eqnarray*}
This gives us
$$\left(r^{-n+2(\alpha+\beta-1)}\int_{T^{\varepsilon}(B)}|g(t,x)|^{2}\frac{dtdx}{t^{1+2(\alpha-\beta+1)}}\right)^{1/2}\lesssim\|L\|,$$
that is,
 $g\in T^{\infty}_{\alpha,\beta}$ with
$\|g\|_{T^{\infty}_{\alpha,\beta}}\lesssim \|L\|.$ This completes
the proof of Theorem \ref{atomde T}.
\end{proof}

\subsection{The
Preduality of $Q_{\alpha}^{\beta}(\mathbb{R}^{n})$}

In this subsection, we introduce a new space which can be viewed as
the predual space of $Q_{\alpha}^{\beta}(\mathbb{R}^{n}).$ Then, we
give an atomic decomposition for this space. For this purpose we
need the following lemma which is Lemma 1.1 in \cite{Frazier Jawerth
Weiss}.

\begin{lemma} \label{Little} Fix $N\in \mathbb{N}.$ Then there exists a function $\phi:\mathbb{R}^{n}\longrightarrow
\mathbb{R}^{n}$ such that\\
(1) $\text{ supp }(\phi)\subset \{x\in \mathbb{R}^{n}: |x|\leq 1\};$\\
(2) $\phi$ is radial;\\
 (3) $\phi\in C^{\infty}(\mathbb{R}^{n});$\\
 (4)$\int_{\mathbb{R}^{n}}x^{\gamma}\phi(x)dx=0$ if $\gamma\in
 \mathbb{N}^{n},$ $x^{\gamma}=x_{1}^{\gamma^{1}}x_{2}^{\gamma^{2}}\cdots
 x_{n}^{\gamma^{n}},$
 $|\gamma|=\gamma_{1}+\gamma_{2}+\cdots+\gamma_{n};$\\
  (5) $\int_{0}^{\infty}(\widehat{\phi}(t\xi))^{2}\frac{dt}{t}=1$ if
  $\xi\in \mathbb{R}^{n}\backslash\{0\}.$
\end{lemma}

For $\phi$ satisfying the conditions of Lemma \ref{Little} and any
$f\in \mathscr{S}'(\mathbb{R}^{n}),$ we have the well known
 Calder\'{o}n reproducing formula
\begin{equation}\label{calder}
f=\int_{0}^{\infty}f\ast\phi_{t}\ast\phi_{t}\frac{dt}{t}=\lim_{\varepsilon\longrightarrow
0,
N\longrightarrow\infty}\int_{\varepsilon}^{N}f\ast\phi_{t}\ast\phi_{t}\frac{dt}{t}.
\end{equation}

We introduce the notation of $HH^{1}_{-\alpha,\beta}$ in the sense
of distributions.

\begin{definition}
For $\phi$ as in above lemma,  $\alpha>0$ and
$\max\{1/2,\alpha\}<\beta<1$ with $\alpha+\beta-1\geq0$, we define
the Hardy-Hausdorff space $HH^{1}_{-\alpha,\beta}(\mathbb{R}^{n})$
to be the class of all distributions $f\in
\dot{L}^{2}_{-\frac{n}{2}+2(\beta-1)}(\mathbb{R}^{n})$ with
$$\|f\|_{HH^{1}_{-\alpha,\beta}(\mathbb{R}^{n})}:=\|f\ast \phi_{t}(\cdot)\|_{T^{1}_{\alpha,\beta}}<\infty.$$
\end{definition}

\begin{theorem}
$\|\cdot\|_{HH^{1}_{-\alpha,\beta}(\mathbb{R}^{n})}$ is a
quasi-norm. Furthermore, $HH^{1}_{-\alpha,\beta}(\mathbb{R}^{n})$ is
complete under this quasi-norm.
\end{theorem}

\begin{proof}
Obviously, $\|\cdot\|_{HH^{1}_{-\alpha,\beta}(\mathbb{R}^{n})}$ is a
quasi-norm according to the linearity  of  $\rho_{\phi}(t,x)=f\ast
\phi_{t}(x)$  and the corresponding property of
$\|\cdot\|_{T^{1}_{\alpha,\beta}}.$ Suppose that $\{f_{j}\}$ is a
Cauchy sequence. By the Calder\'{o}n reproducing formula and Theorem
\ref{connection Q B}, we get
$\dot{L}^{2}_{\frac{n}{2}-2(\beta-1)}\hookrightarrow
Q_{\alpha}^{\beta}$ and for every $\psi\in
\mathscr{S}(\mathbb{R}^{n})$
\begin{eqnarray*}
 |\langle f_{j}-f_{k},
\psi\rangle|
&\lesssim&\|\rho_{\phi}(f_{j}-f_{k})\|_{T^{1}_{\alpha,\beta}}\|\phi_{t}\ast\psi\|_{T^{\infty}_{\alpha,\beta}}\\
&\lesssim&\|\rho_{\phi}(f_{j}-f_{k})\|_{T^{1}_{\alpha,\beta}}\|\psi\|_{Q_{\alpha}^{\beta}}\\
&\lesssim&\|\rho_{\phi}(f_{j}-f_{k})\|_{T^{1}_{\alpha,\beta}}\|\psi\|_{\dot{L}^{2}_{\frac{n}{2}-2(\beta-1)}}.
\end{eqnarray*}
This deduces that $\{f_{j}\}$ is a Cauchy sequence in
$\dot{L}^{2}_{-\frac{n}{2}+2(\beta-1)}.$ By completeness, $f=\lim
f_{n}$ exists in $\dot{L}^{2}_{-\frac{n}{2}+2(\beta-1)}.$ Thus there
exists a subsequence such that $f=f_{1}+\sum_{j\geq
1}(f_{j+1}-f_{j})$ in $\mathscr{S}'(\mathbb{R}^{n})$
 with
 $\sum\|f_{j+1}-f_{j}\|_{HH^{1}_{-\alpha,\beta}(\mathbb{R}^{n})}<\infty.$
 Then we have
 $$\|\rho_{\phi}(f)\|_{T^{1}_{\alpha,\beta}}\lesssim(\|\rho_{\phi}(f_{1})\|_{T^{1}_{\alpha,\beta}}
 +\sum\|\rho_{\phi}(f_{j+1}-f_{j})\|_{T^{1}_{\alpha,\beta}})<\infty$$
 and so $f\in HH^{1}_{-\alpha,\beta}(\mathbb{R}^{n})$. Similarly we
 can prove
 $f_{j}\longrightarrow f $ in $HH^{1}_{-\alpha,\beta}(\mathbb{R}^{n}).$
\end{proof}

\begin{definition}\label{Def HH}
Let $\alpha>0$ and  $\max\{1/2,\alpha\}<\beta<1$ with
$\alpha+\beta-1\geq 0.$ A tempered distribution $a$ is called an
$HH^{1}_{-\alpha,\beta}(\mathbb{R}^{n})$ atom if $a$ is supported in
a cube $I$ and satisfies the following two
conditions: \\
(i) a {\it{local Sobolev$-(\alpha-\beta+1)$}} condition: for all
$\psi\in \mathscr{S}$
$$|\langle a, \psi\rangle|\leq \hbox{diam}(I)^{-\frac{n}{2}+\alpha+\beta-1}
\left(\int_{I}\int_{I}\frac{|\psi(x)-\psi(y)|^{2}}{|x-y|^{2(\alpha-\beta+1)}}dxdy\right)^{1/2};$$
(ii) a cancelation condition: $\langle a, \psi\rangle=0$ for any
$\psi\in \mathscr{S}$ which coincides with a polynomial of degree
$\leq \frac{n}{2}+1$ in a neighborhood of $I$.
\end{definition}
In \cite{G. Dafni J. Xiao},  Dafni-Xiao  established  the following
factional Poincar\'{e} inequality which will help us to understand
the previous definition.
\begin{lemma} \label{fractional Poincare}
Let $\psi\in C^{\infty}(\mathbb{R}^{n})$ and $I$ be a cube. Denote
by $\psi(I)$ the average of $\psi$ over $I.$ If $0\leq
\alpha_{1},\alpha_{2}<\beta$ for a fixed $\beta\in (1/2,1),$ then
\begin{eqnarray*}
\|\psi-\psi(I)\|_{L^{2}}&\leq&n^{n/4}\hbox{diam}(I)^{\alpha_{1}-\beta+1}
\left(\int_{I}\int_{I}\frac{|\psi(x)-\psi(y)|^{2}}{|x-y|^{n+2(\alpha_{1}-\beta+1)}}dxdy\right)^{1/2}\\
&\leq&n^{n/4}\hbox{diam}(I)^{\alpha_{2}-\beta+1}
\left(\int_{I}\int_{I}\frac{|\psi(x)-\psi(y)|^{2}}{|x-y|^{n+2(\alpha_{2}-\beta+1)}}dxdy\right)^{1/2}\\
&\leq&C\hbox{diam}(I)\|\nabla\psi\|_{L^{2}(I)}
\end{eqnarray*}
with $C$ depending only on the dimension and $\alpha_{2}.$ If in
addition $\int_{I}\frac{\partial\psi}{\partial x_{k}}dx=0$ for all
$k=1,\cdots,n,$ then the quantities above are also bounded by
$$C\hbox{diam}(I)\|\nabla\psi-(\nabla\psi)_{I}\|_{L^{2}(I)}
\leq C n^{n/4}\hbox{diam}(I)^{\alpha_{1}-\beta+2}
\left(\int_{I}\int_{I}\frac{|\nabla\psi(x)-\nabla\psi(y)|^{2}}{|x-y|^{n+2(\alpha_{1}-\beta+1)}}dxdy\right)^{1/2}.$$
Here $(\nabla\psi)_{I}$ denotes the vector whose coordinates are the
means $(\frac{\partial\psi}{\partial x_{k}})(I),$ $k=1,\cdots,n.$
\end{lemma}

\begin{remark}\label{remark 314} Similar to Remark (2) after Lemma 6.2 of Dafni-Xiao \cite{G. Dafni J.
Xiao},
 we can prove that  an
$HH^{1}_{-\alpha,\beta}-$atom $a$ belongs to the homogeneous Sobolev
spaces $\dot{L}^{2}_{-s}$ with $\alpha+\beta-1\leq s\leq
\frac{n}{2}+1.$ Particularly, we have
\begin{eqnarray*}
|\langle a,\psi\rangle|
\lesssim(\hbox{diam}(I))^{-\frac{n}{2}+\alpha+\beta-1}\|\psi\|_{\dot{L}^{2}_{\alpha-\beta+1}}.
\end{eqnarray*}
This deduces $\|a\|_{\dot{L}^{2}_{-(\alpha-\beta+1)}}\lesssim
(\hbox{diam}(I))^{-\frac{n}{2}+\alpha+\beta-1}.$ Meanwhile,
$|\langle
a,\psi\rangle|\lesssim(\hbox{diam}(I))\|\psi\|_{\dot{L}^{2}_{\frac{n}{2}-2\beta+3}}$
 and so
$\|a\|_{\dot{L}^{2}_{-(\frac{n}{2}-2\beta+3)}}\lesssim
\hbox{diam}(I).$
\end{remark}

We can obtain the atomic decomposition of $HH^{1}_{-\alpha,\beta}$
as follows.
\begin{theorem}\label{atom dec H}
Let $\alpha>0$ and  $\max\{1/2,\alpha\}<\beta<1$ with
$\alpha+\beta-1\geq0.$ A tempered distribution $f$ on
$\mathbb{R}^{n}$ belongs to $HH^{1}_{-\alpha,\beta}$ if and only if
there exist $HH^{1}_{-\alpha,\beta}-$atoms $\{a_{j}\}$ and an
$l^{1}-$summable sequence $\{\lambda_{j}\}$ such that
$f=\sum_{j}\lambda_{j}a_{j}$ in the sense of distributions.
Moreover,
$$\|f\|_{HH^{1}_{-\alpha,\beta}}\approx\inf\left\{\sum_{j}|\lambda_{j}|: f=\sum_{j}\lambda_{j}a_{j}\right\}.$$
\end{theorem}

\begin{proof}
{\it{Part 1.}} ``$\Longleftarrow$" By  the completeness of
$HH^{1}_{-\alpha,\beta}(\mathbb{R}^{n}),$   we only need to prove
that if $a$ is an $HH^{1}_{-\alpha,\beta}-$atom then $a$ is in
$HH^{1}_{-\alpha,\beta}(\mathbb{R}^{n})$ with the quasinorm bounded
by a constant. Since $a$ is an $HH^{1}_{-\alpha,\beta}-$atom and
$\alpha+\beta-1\leq \frac{n}{2}-2(\beta-1)\leq \frac{n}{2}+1,$
Remark \ref{remark 314}
 implies  that $a\in \dot{L}^{2}_{-\frac{n}{2}+2(\beta-1)}$ with norm bounded
by a constant. On the other hand, assume that $I$ is the support of
$a$ and $x_{I}$ represents its center. For $\varepsilon\in(0,2),$
let
$$
\omega(t,x)=\kappa(l(I))^{-n+2(\alpha+\beta-1)}\min\left\{1,\left(\frac{l(I)}
{\sqrt{(x-x_{I})^{2}+t^{2}}}\right)^{n-2(\alpha+\beta-1)+\varepsilon}\right\}
$$
where  $\kappa$ is a  constant to be chosen later. Similar to the
proof of Theorem \ref{atomde T}, we have
$$N\omega (x)\leq \kappa(l(I))^{-n+2(\alpha+\beta-1)}
\min\left\{1,
\left(\frac{\sqrt{2}l(I)}{|x-x_{I}|}\right)^{n-2(\alpha+\beta-1)+\varepsilon}\right\}$$
and so $\int_{\mathbb{R}^{n}}N\omega
d\Lambda_{n-2(\alpha+\beta-1)}^{(\infty)}\lesssim \kappa\leq 1$ by
choosing $\kappa$ small enough.

Now, let  $B_{I}=B(x_{I}, \hbox{diam}(I)),
E_{I}=(0,\hbox{diam}(I))\times B_{I}$ and $
E_{I}^{c}=\mathbb{R}^{1+n}_{+}\backslash E_{I}.$ Suppose $S_{a}$ is
the support of $a\ast \phi_{t}(x)$ in $\mathbb{R}^{1+n}_{+}.$ We
have
$$\int_{\mathbb{R}^{1+n}_{+}}|a\ast\phi_{t}(x)|^{2}\omega^{-1}(t,x)\frac{dtdx}{t^{1-2(\alpha-\beta+1)}}
=\left(\int_{E_{I}}+\int_{E^{c}_{I}\cap
S_{a}}\right)|a\ast\phi_{t}(x)|^{2}\omega^{-1}(t,x)\frac{dtdx}{t^{1-2(\alpha-\beta+1)}}.$$
By the definition of  the cylinder $E_{I}$ in
$\mathbb{R}^{1+n}_{+},$ we can find a half-ball centered at $(0,
x_{I})$ to cover $E_{I}.$ Thus we have
$\omega^{-1}\lesssim(l(I))^{n-2(\alpha+\beta-1)}$ on $E_{I}.$ This
fact  implies that
\begin{eqnarray*}
\int_{E_{I}}|a\ast\phi_{t}(x)|^{2}\omega^{-1}(t,x)\frac{dtdx}{t^{1-2(\alpha-\beta+1)}}
\!\!
&\leq&\!\!(l(I))^{n-2(\alpha+\beta-1)}\int_{0}^{\infty}\int_{\mathbb{R}^{n}}|\widehat{a}(\xi)|^{2}|\widehat{\phi}(t\xi)|^{2}d\xi\frac{dt}{t^{1-2(\alpha-\beta+1)}}\\
&\leq&\!\!(l(I))^{n-2(\alpha+\beta-1)}\!\!\!\int_{\mathbb{R}^{n}}|\widehat{a}(\xi)|^{2}|\xi|^{-2(\alpha-\beta+1)}d\xi
\!\!\!\int_{0}^{\infty}
|\widehat{\phi}(t)|^{2}\frac{dt}{t^{1-2(\alpha-\beta+1)}}\\
&\leq&\!\!(l(I))^{n-2(\alpha+\beta-1)}\|a\|_{\dot{L}^{2}_{-(\alpha-\beta+1)}}
\lesssim1.
\end{eqnarray*}
For the integral on $E^{c}_{I}\bigcap S_{a}.$  If
 $z\in I,$ $x\nin B_{I}$ and $t\leq |x-x_{I}|/2,$
 then  $$|x-z|\geq|x-x_{I}|-\hbox{diam(I)/2}\geq |x-x_{I}|/2 \geq
 t,$$
  and $a\ast \phi_{t}(x)=\int a(z)\phi_{t}(x-z)dz=0.$
 Otherwise, we have
\begin{eqnarray*}
|a\ast\phi_{t}(x)|
&\leq&
\|a\|_{\dot{L}^{2}_{-\frac{n}{2}+2\beta-3}}\|\phi_{t}^{x}\|_{\dot{L}^{2}_{\frac{n}{2}-2\beta+3}}
\leq\hbox{diam}(I)t^{-(n-2\beta+3)}\left(\int_{\mathbb{R}^{n}}|\widehat{\phi}(\xi)|^{2}|\xi|^{n-4\beta+6}d\xi\right)^{1/2}\\
&\leq&\hbox{diam}(I)t^{-(n-2\beta+3)}.
\end{eqnarray*}
It is easy to check
$t\approx\sqrt{|x-x_{I}|^{2}+t^{2}}:=r(t,x)>$diam$I$. This implies
that
\begin{eqnarray*}
\omega^{-1}(t,x)
&\approx&\kappa^{-1}(l(I))^{n-2(\alpha+\beta-1)}\frac{t^{n-2(\alpha+\beta-1)+\varepsilon}}{(l(I))^{n-2(\alpha+\beta)+\varepsilon}}
\lesssim(l(I))^{-\varepsilon}{t^{n-2(\alpha+\beta-1)+\varepsilon}}.
\end{eqnarray*}
Then we can get
\begin{eqnarray*}
\int_{E^{c}_{I}\bigcap S_{a}}|a\ast
\phi_{t}(x)|^{2}\omega^{-1}(t,x)\frac{dtdx}{t^{1-2(\alpha-\beta+1)}}&\lesssim&(l(I))^{2-\varepsilon}\int_{E^{c}_{I}\bigcap
S_{a}}t^{\varepsilon-n-3}dtdx\\
&\lesssim&(l(I))^{2-\varepsilon}\int_{r(t,x)\geq \hbox{diam}(I)}r(t,x)^{\varepsilon-n-3}dt\\
&\lesssim&(l(I))^{2-\varepsilon+\varepsilon-2}\lesssim 1.
\end{eqnarray*}

{\it{Part 2.}} ``$\Longrightarrow$" Suppose $f\in
HH^{1}_{-\alpha,\beta}(\mathbb{R}^{n}).$ Note that the  Calder\'{o}n
reproducing formula (\ref{calder}) holds in the sense of
distributions. Since the support of $\phi$ is the unit ball, we can
denote
$$f^{\varepsilon,N}(x)=\int_{S^{\varepsilon,N}}F(t,y)\phi_{t}(x-y)\frac{dtdy}{t}$$
where $F(t,y)=f\ast\phi_{t}(y)$ and $S^{\varepsilon, N}$ is the
strip $\{(t,x)\in \mathbb{R}^{1+n}_{+}: \varepsilon\leq t\leq N\}.$
Similar to the proof of Theorem \ref{atomde T}, there exists an
$\omega\geq 0 $ on $\mathbb{R}^{1+n}_{+}$ such that
$\int_{\mathbb{R}^{n}}N\omega
d\Lambda_{n-2(\alpha+\beta-1)}^{(\infty)}\leq 1$ and
$$\int_{\mathbb{R}^{1+n}_{+}}|F(t,x)|^{2}\omega^{-1}(t,x)\frac{dtdx}{t^{1-2(\alpha-\beta+1)}}\leq 2\|F\|_{T^{1}_{\alpha,\beta}}.$$
Let  $T_{j,k}$ be the corresponding  structures  over the set
$E_{k}=\{N\omega > 2^{k}\}$ as those in Theorem \ref{atomde T} (i).
Noting that $T_{j,k}$ have mutually disjoint interiors and $F=\sum
F\chi_{T_{j,k}}$ a.e. on $\mathbb{R}^{1+n}_{+},$ we let
$$g^{\varepsilon,N}_{j,k}(x)=\int_{S^{\varepsilon,N}\bigcap T_{j,k}}F(t,y)\phi_{t}(x-y)\frac{dtdy}{t}.$$
Since $T_{j,k}\subset T(I^{*}_{j,k})$,   these smooth functions in
$x$ is supported in $\{x: \Gamma(x)\cap T_{j,k}\neq
\emptyset\}\subset I^{*}_{j,k}$ and have the same number moments as
$\phi$.  We want to verify that there are distributions $g_{j,k}$
such that $g^{\varepsilon,N}_{j,k}\longrightarrow g_{j,k}$ as
$\varepsilon\longrightarrow 0$ and $N\longrightarrow \infty$  with
$f=\sum_{j,k}g_{j,k}$ in $\mathscr{S}'(\mathbb{R}^{n}).$ To see
this, noting that $\omega\leq 2^{k+1}$ on $T_{j,k},$ we have
\begin{eqnarray*}\label{es 64}
&&|\langle g^{\varepsilon,N}_{j,k},\psi\rangle|
=\left|\int_{\mathbb{R}^{n}}\left(\int_{S^{\varepsilon,N}\bigcap
T_{j,k}}F(t,y)\phi_{t}(x-y)\frac{dtdy}{t}\right)\psi(x) dx\right|\\
&\leq&2^{(k+1)/2}\left(\int_{S^{\varepsilon,N}\bigcap
T_{j,k}}|F(t,y)|^{2}\omega^{-1}(t,y)\frac{dtdy}{t^{1-2(\alpha-\beta+1)}}\right)^{1/2}
\left(\int_{S^{\varepsilon,N}\bigcap
T_{j,k}}|\psi\ast\phi_{t}(y)|^{2}\frac{dtdy}{t^{1+2(\alpha-\beta+1)}}\right)^{1/2}
\\
&\leq &2^{(k+1)/2}\left(\int_{S^{\varepsilon,N}\bigcap
T_{j,k}}|F(t,y)|^{2}\omega^{-1}(t,y)\frac{dtdy}{t^{1-2(\alpha-\beta+1)}}\right)^{1/2}
\left(\int_{3I^{*}_{j,k}}\int_{3I^{*}_{j,k}}\frac{|\psi(x)-\psi(y)|^{2}}{|x-y|^{n+2(\alpha-\beta+1)}}dtdy\right)^{1/2}.
\end{eqnarray*}
Similarly, we obtain that  for $\varepsilon_{1}<\varepsilon_{2}$ and
$N_{1}>N_{2},$
$$|\langle g^{\varepsilon_{1}, N_{1}}_{j,k}- g^{\varepsilon_{2}, N_{2}}_{j,k},\psi\rangle|
\leq C_{k}\left(\int_{(S^{\varepsilon_{1},N_{1}}\backslash
S^{\varepsilon_{2},N_{2}})\bigcap
T_{j,k}}|F(t,y)|^{2}\omega^{-1}(t,y)\frac{dtdy}{t^{1-2(\alpha-\beta+1)}}\right)^{1/2}\|\psi\|_{\dot{L}^{2}_{\alpha-\beta+1}}.$$
This gives us that $\|g^{\varepsilon_{1},
N_{1}}_{j,k}-g^{\varepsilon_{2},
N_{2}}_{j,k}\|_{\dot{L}^{2}_{-(\alpha-\beta+1)}}\longrightarrow0.$
as $\varepsilon_{1},$ $\varepsilon_{2}\longrightarrow 0$ and
$N_{1},$ $N_{2} \longrightarrow \infty$. Thus, $g^{\varepsilon,
N}_{j,k}\longrightarrow g_{j,k}\in \dot{L}^{2}_{-(\alpha-\beta+1)}$
in the sense of distributions and $g_{j,k}$ is supported in
$I^{*}_{j,k}$ with
$$\|g_{j,k}\|_{\dot{L}^{2}_{-(\alpha-\beta+1)}(3I^{*}_{j,k})}\lesssim2^{(k+1)/2}\left(\int_{T_{j,k}}
|F(t,y)|^{2}\omega^{-1}(t,y)\frac{dtdy}{t^{1-2(\alpha-\beta+1)}}\right)^{1/2}.$$
Let
$a_{j,k}=g_{j,k}\|g_{j,k}\|^{-1}_{\dot{L}^{2}_{-(\alpha-\beta+1)}(3I^{*}_{j,k})}(l(3I^{*}_{j,k}))^{(\alpha+\beta-1)-\frac{n}{2}}
\ \hbox{and}\
\lambda_{j,k}=\|g_{j,k}\|_{\dot{L}^{2}_{-(\alpha-\beta+1)}(3I^{*}_{j,k})}(l(3I^{*}_{j,k}))^{\frac{n}{2}-(\alpha+\beta-1)}.$
Then
\begin{eqnarray*}
|\langle a_{j,k},
\psi\rangle|
&\lesssim&\frac{1}{\|g_{j,k}\|_{\dot{L}^{2}_{-(\alpha-\beta+1)}(3I^{*}_{j,k})}}(l(3I^{*}_{j,k}))^{(\alpha+\beta-1)-\frac{n}{2}}\\
&\times&\left(\int_{S^{\varepsilon,N}\bigcap
T_{j,k}}|F(t,y)|^{2}w(t,y)^{-1}\frac{dtdy}{t^{1-2(\alpha-\beta+1)}}\right)^{\frac{1}{2}}
\left(\int_{3I^{*}_{j,k}}\int_{3I^{*}_{j,k}}\frac{|\psi(x)-\psi(y)|^{2}}{|x-y|^{n+2(\alpha-\beta+1)}}dxdy\right)^{\frac{1}{2}}.
\end{eqnarray*}
This means that $a_{j,k}$ are $HH^{1}_{-\alpha,\beta}-$atoms. On the
other hand, the Cauchy-Schwarz inequality implies that
\begin{eqnarray*}
\sum_{j,k}|\lambda_{j,k}|
&\lesssim&\left(\sum_{j,k}2^{k+1}(l(3I^{*}_{j,k}))^{n-2(\alpha+\beta-1)}\right)^{1/2}
\left(\sum_{j,k}\int_{T_{j,k}}|F(t,y)|^{2}\omega^{-1}(t,y)\frac{dtdy}{t^{1-2(\alpha-\beta+1)}}\right)^{1/2}\\
&\lesssim&\left(\sum_{j,k}2^{k+1}\Lambda^{(\infty)}_{n-2(\alpha+\beta-1)}(3I^{*}_{j,k})\right)^{1/2}
\left(\sum_{j,k}\int_{\mathbb{R}^{1+n}_{+}}|F(t,y)|^{2}\omega^{-1}(t,y)\frac{dtdy}{t^{1-2(\alpha-\beta+1)}}\right)^{1/2}\\
&\lesssim&\left(\sum_{k}\int_{E_{K}}2^{k+1}d\Lambda_{n-2(\alpha+\beta-1)}^{(\infty)}(E_{k})\right)^{1/2}\|f\|_{HH^{1}_{-\alpha,\beta}}\\
&\lesssim&
\left(\sum_{k}\int_{E_{k}}N\omega(x)d\Lambda^{(\infty)}_{n-2(\alpha+\beta-1)}\right)^{1/2}\|f\|_{HH^{1}_{-\alpha,\beta}}\lesssim\|f\|_{HH^{1}_{-\alpha,\beta}}.
\end{eqnarray*}
The above estimates tell us that $\sum
g_{j,k}=\sum\lambda_{j,k}a_{j,k}$ converges to a distribution $g$ in
$HH^{1}_{-\alpha,\beta}(\mathbb{R}^{n})$.
 We need to verify that $g=f.$
Since for a fix $\psi\in \mathscr{S}(\mathbb{R}^{n}),$ every
$0<\varepsilon<N,$
\begin{eqnarray*}
&&|\langle g^{\varepsilon,N}_{j,k},\psi\rangle|\\
&\lesssim&2^{(k+1)/2}\left(\int_{
T_{j,k}}|F(t,y)|^{2}w(t,y)^{-1}\frac{dtdy}{t^{1-2(\alpha-\beta+1)}}\right)^{1/2}
\left(\int_{3I^{*}_{j,k}}\int_{3I^{*}_{j,k}}\frac{|\psi(x)-\psi(y)|^{2}}{|x-y|^{n+2(\alpha-\beta+1)}}dtdy\right)^{1/2}\\
&\lesssim&2^{(k+1)/2}
(l(3I^{*}_{j,k}))^{\frac{n}{2}-(\alpha+\beta-1)}\left(\int_{
T_{j,k}}|F(t,y)|^{2}w(t,y)^{-1}\frac{dtdy}{t^{1-2(\alpha-\beta+1)}}\right)^{1/2}\|\psi\|_{Q_{\alpha,\beta}(\mathbb{R}^{n})}\\
&\lesssim&\|f\|_{HH^{1}_{-\alpha,\beta}}\|\psi\|_{Q_{\alpha,\beta}(\mathbb{R}^{n})}.
\end{eqnarray*}
Then, $\lim\limits_{\varepsilon\longrightarrow 0,
N\longrightarrow\infty}\sum_{j,k}g^{\varepsilon,N}=\sum_{j,k}g_{j,k}=g.$
Meanwhile, we can also obtain that
\begin{eqnarray*}
\sum\limits_{j,k}\int_{\mathbb{R}^{1+n}_{+}}1_{S^{\varepsilon,N}\bigcap
T_{j,k}}(t,y)F(t,y)\phi_{t}\ast\psi(y)\frac{dtdy}{t}
=\int_{S^{\varepsilon,N}}F(t,y)\phi_{t}\ast\psi(y)\frac{dtdy}{t}=\langle
f^{\varepsilon,N}, \psi\rangle.
\end{eqnarray*}
 This tells us $\sum_{j,k}g_{j,k}^{\varepsilon,N}=f^{\varepsilon,N}\longrightarrow
 f$ in $\mathscr{S}'(\mathbb{R}^{n}).$ Therefore $f=g$ in $\mathscr{S}'(\mathbb{R}^{n}).$
   \end{proof}

\begin{lemma}\label{atom dec H lemma}
(i) If $a$ is an $HH^{1}_{-\alpha,\beta}-$atom, then there exists a
nonnegative function $\omega$ on $\mathbb{R}^{1+n}_{+}$ with
$\int_{\mathbb{R}^{n}}N\omega
d\Lambda_{n-2(\alpha+\beta-1)}^{(\infty)}\leq 1$ and
$$\sigma_{\delta}(a,\omega)=\sup_{|y|\leq \delta}\left(\int_{\mathbb{R}^{1+n}_{+}}
|a\ast\psi_{t}(x-y)-a\ast\psi_{t}(x)|^{2}\omega(t,x)^{-1}\frac{dtdx}{t^{1-2(\alpha-\beta+1)}}\right)^{1/2}\longrightarrow
0.$$
 (ii) $HH^{1}_{-\alpha,\beta}\cap C^{\infty}_{0}(\mathbb{R}^{n})$ is
 dense in $HH^{1}_{-\alpha,\beta}.$
\end{lemma}

\begin{proof} (i) For a fixed $\varepsilon\in (0,2),$  the same  $\omega$
defined in the proof of Theorem \ref{atom dec H},  $y\in
B(0,\delta)$ and $x\in \mathbb{R}^{n},$ we have
$a\ast\phi_{t}(x-y)-a\ast\phi_{t}(x)=\langle a,
\phi_{t}^{x-y}-\phi_{t}^{x}\rangle$ and
\begin{equation}
|(\widehat{\phi}_{t}^{x-y}-\widehat{\phi}_{t}^{x})(\xi)|=|1-e^{2\pi
iy\cdot\xi}||\widehat{\phi}_{t}(\xi)|\leq
C\min\{2,\delta|\xi|\}|\widehat{\phi}_{t}(\xi)|.
\end{equation}
Note that
\begin{eqnarray*}
&&\sup_{|y|<\delta}\left(\int_{\mathbb{R}^{1+n}_{+}}|a\ast\phi_{t}(x-y)-a\ast\phi_{t}(x)|^{2}
\omega^{-1}(t,x)\frac{dtdx}{t^{1-2(\alpha-\beta+1)}}\right)^{1/2}\\
&\lesssim&\left((\sup_{|y|<\delta}\int_{E_{I}}+\sup_{|y|<\delta}\int_{E_{I}^{c}\bigcap
S_{a,\delta}})|a\ast\phi_{t}(x-y)-a\ast\phi_{t}(x)|^{2}
\omega^{-1}(t,x)\frac{dtdx}{t^{1-2(\alpha-\beta+1)}}\right)^{1/2},
\end{eqnarray*}
where $B_{I}$ is the ball $B(x_{I},2\hbox{diam}(I)),$ and
$E_{I}=(0,2\hbox{diam}(I))\times B_{I}.$ By Fourier transforms, we
can estimate the first term as
\begin{eqnarray*}
&&\sup_{|y|<\delta}\int_{E_{I}}|a\ast\phi_{t}(x-y)-a\ast\phi_{t}(x)|^{2}
\omega^{-1}(t,x)\frac{dtdx}{t^{1-2(\alpha-\beta+1)}}\\
&\lesssim&(l(I))^{n-2(\alpha+\beta-1)}\sup_{|y|<\delta}\int_{0}^{\infty}\int_{\mathbb{R}^{n}}|a\ast(\phi_{t}^{y}-\phi_{t})(x)|^{2}
\frac{dtdx}{t^{1-2(\alpha-\beta+1)}}\\
&\lesssim&(l(I))^{n-2(\alpha+\beta-1)}\sup_{|y|<\delta}
\int_{\mathbb{R}^{n}}|\widehat{
a}(\xi)|^{2}\min\{2,\delta|\xi|\}^{2}
\int_{0}^{\infty}|\widehat{\phi}(t|\xi|)|^{2}
\frac{dt}{t^{1-2(\alpha-\beta+1)}}d\xi
\\
&\lesssim&(l(I))^{n-2(\alpha+\beta-1)}\sup_{|y|<\delta}
\int_{\mathbb{R}^{n}}|\widehat{
a}(\xi)|^{2}\delta^{2}|\xi|^{2}|\xi|^{-2(\alpha-\beta+1)}
\int_{0}^{\infty}|\widehat{\psi}(t)|^{2}
\frac{dt}{t^{1-2(\alpha-\beta+1)}}d\xi\rightarrow0
\end{eqnarray*}
 as $\delta\longrightarrow 0$ according to the dominated convergence theorem.

 For the second term. 
 Since $\hbox{supp}(a)=I,$
 when $x\nin
B_{I}$ and $t\leq |x-x_{I}|/4,$ we obtain
$|y|<\hbox{diam}(I)<\frac{1}{2}|x-x_{I}|$ for $y\in B(0,\delta)$
with $\delta<\hbox{diam}(I).$ Therefore
$$|x-y-z|>|x-x_{I}|-|z-x_{I}|-|y|\geq \frac{3}{4}|x-x_{I}|\geq t.$$
On the other hand $|x-z|\geq\frac{3}{4}|x-x_{I}|>t$. These estimates
imply that $a\ast[\phi_{t}(x-y)-\phi_{t}(x)]=0.$ Otherwise, we have
\begin{eqnarray*}
|a\ast\phi_{t}(x-y)-a\ast\phi_{t}(x)|
&\lesssim&\|a\|_{\dot{L}^{2}_{-(\frac{n}{2}-2\beta+3)}}\|\phi_{t}^{x-y}-\phi_{t}^{x}\|_{\dot{L}^{2}_{\frac{n}{2}-2\beta+3}}\\
&\lesssim&\hbox{diam}(I)\left(\int_{\mathbb{R}^{n}}|\widehat{\phi_{t}^{x-y}}
(\xi)-\widehat{\phi_{t}^{x}}(\xi)|^{2}|\xi|^{n-4\beta+6}d\xi\right)^{1/2}\\
&\lesssim&\hbox{diam}(I)\left(\int_{\mathbb{R}^{n}}\min\{2,\delta|\xi|\}^{2}|\widehat{\phi_{t}}(\xi)|^{2}|\xi|^{n-4\beta+6}d\xi\right)^{1/2}\\
&\lesssim&\hbox{diam}(I)\delta\left(\int_{\mathbb{R}^{n}}|\widehat{\phi_{t}}(\xi)|^{2}|\xi|^{n-4\beta+8}d\xi\right)^{1/2}\\
&\lesssim&\hbox{diam}(I)\delta t^{2\beta-4-n}.
\end{eqnarray*}
 Using  the above estimates and the fact $\omega^{-1}\lesssim
 t^{n-2(\alpha+\beta-1)+\varepsilon},$ we have
  \begin{eqnarray*}
&&\int_{E^{c}_{I}\bigcap
S_{a,\delta}}|a\ast\phi_{t}(x-y)-a\ast\phi_{t}(x)|^{2}\omega^{-1}(t,x)\frac{dtdx}{t^{1-2(\alpha-\beta+1)}}\\
&\lesssim&\delta^{2}(l(I))^{2-\varepsilon}\int_{E^{c}_{I}\bigcap
S_{a,\delta}}t^{-n-5+\varepsilon}dtdx\\
&\lesssim&\delta^{2}(l(I))^{2-\varepsilon}\int_{l(I)}^{\infty}\lambda^{\varepsilon-5}d\lambda
\longrightarrow 0\end{eqnarray*} as $\delta\longrightarrow0.$ Thus
$\sigma_{\delta}(a,\omega)\longrightarrow0$ as
$\delta\longrightarrow0.$

 (ii) For an $HH^{1}_{\alpha,\beta}-$atom $a,$ take $\eta\in
 C^{\infty}(\mathbb{R}^{n})$ with support in $B(0,1)$ and $\int
 \eta=1.$ Then   $a\ast\eta_{j}\in C_{0}^{\infty}(\mathbb{R}^{n})$ and $\eta_{j}=j^{n}\eta(jx)$ form an approximate
 identity,  $a\ast\eta_{j}\longrightarrow a$ in $\mathscr{S}'(\mathbb{R}^{n})$
 as $j\longrightarrow\infty.$  For any
 nonnegative function $\omega$ on $\mathbb{R}^{1+n}_{+}$ with
 $\int_{\mathbb{R}^{n}}N\omega d\Lambda_{n-2(\alpha+\beta-1)}^{(\infty)}\leq 1,$
 we have
 \begin{eqnarray*}
&&\left(\int_{\mathbb{R}^{1+n}_{+}}|a\ast\eta_{j}\ast\phi_{t}(x)-a\ast\phi_{t}(x)|^{2}
\omega^{-1}(t,x)\frac{dtdx}{t^{1-2(\alpha-\beta+1)}}\right)^{1/2}\\
&\lesssim&\int_{\mathbb{R}^{n}}|\eta_{j}(y)|\left(\int_{\mathbb{R}^{1+n}_{+}}|a\ast\phi_{t}(x-y)-a\ast\phi_{t}(x)|^{2}
\omega^{-1}(t,x)\frac{dtdx}{t^{1-2(\alpha-\beta+1)}}\right)^{1/2}\\
&\lesssim& \sigma_{\frac{1}{j}}(a,\omega).
 \end{eqnarray*}
From (i), we know that for every $\varepsilon>0$ there exists an
 $\omega$ such that $\sigma_{\frac{1}{j}}(a,\omega)<\varepsilon$
 with $j$ large enough. Taking the infimum over all $\omega$ induces
 $$\|a\ast\eta_{j}-a\|_{HH^{1}_{-\alpha,\beta}}<\varepsilon \ \hbox{for large}\ j,$$
 that is, $a\ast\eta_{j}\longrightarrow a$ in
 $HH^{1}_{-\alpha,\beta}.$
Hence,  we can get the desired
 density from  the fact that   every $f\in
 HH^{1}_{-\alpha,\beta}$ can be approximated by finite sums of atoms.
\end{proof}

\begin{lemma}\label{Carleson m 1}
For $\alpha>0,$  $\max\{\alpha,1/2\}<\beta<1$ with
$\alpha+\beta-1\geq0$, $f\in L_{loc}^{2}(\mathbb{R}^{n})$ and
$\phi\in\mathscr{S}(\mathbb{R}^{n})$ with
$\int_{\mathbb{R}^{n}}\phi(x)dx=0,$ let
$$d\mu_{f,\phi,\alpha,\beta}(t,x)=|(f\ast \phi_{t})(y)|^{2}t^{-1-2(\alpha-\beta+1)}dtdy.$$
Then there is a constant $C$ such that for any cubes $I$ and $J$ in
$\mathbb{R}^{n}$ with center $x_{0}$
and $l(J)\geq 3l(I),$\\
(i)
$$\mu_{f,\phi,\alpha,\beta}(S(I))\leq\int_{J}\int_{J}\frac{|f(x)-f(y)|^{2}}{|x-y|^{n+2(\alpha-\beta+1)}}dxdy+[l(I)]^{n-2(\alpha-\beta)}
\left(\int_{\mathbb{R}^{n}\backslash\frac{2}{3}J}\frac{|f(x)-f(y)|}{|x-x_{0}|^{n+1}}dx\right)^{2}.
$$
(ii) If in addition  $\hbox{supp}(\phi)\subset \{x\in
\mathbb{R}^{n}: |x|\leq 1\}$ then
$$\mu_{f,\phi,\alpha,\beta}(S(I))\leq C\int_{J}\int_{J}\frac{|f(x)-f(y)|^{2}}{|x-y|^{n+2(\alpha-\beta+1)}}dxdy.$$
\end{lemma}
\begin{proof} This lemma is  a special case of Dafini-Xiao \cite[Lemma 3.2]{G. Dafni J.
Xiao}.
\end{proof}

\begin{theorem}\label{Carleson m 2}
Let $\phi$ be a function as in Lemma \ref{Little}, $\alpha>0$ and
$\max\{\alpha,1/2\}<\beta<1$ with $\alpha+\beta-1\geq0.$ If $f\in
Q_{\alpha}^{\beta}(\mathbb{R}^{n})$ then
$d\mu_{f,\phi,\alpha,\beta}(t,x)=|(f\ast\phi_{t})(x)|^{2}t^{-1-2(\alpha-\beta+1)}dtdx$
is a $1-2(\alpha+\beta-1)/n-$Carleson measure.
\end{theorem}
\begin{proof}
The proof follows from (ii) Lemma \ref{Carleson m 1} by taking
$J=3I.$ 
\end{proof}
To establish the equivalent (\ref{eq cha q a b})
 we need another  theorem which contains  the converse of Theorem \ref{Carleson m
 2}.

\begin{theorem}\label{equ carle 1}
Consider the operator $\pi_{\phi}$ defined by
\begin{equation}\label{eq pi}
\pi_{\phi}(F)=\int_{0}^{\infty}F(t,\cdot)\ast\phi_{t}\frac{dt}{t}.
\end{equation}
(i) The operator $\pi_{\phi}$ is a bounded and surjective operator
form $T^{\infty}_{\alpha,\beta}$ to $Q_{\alpha}^{\beta}.$ More
precisely, if $F\in T^{\infty}_{\alpha,\beta}$ then the right-hand
side of the above integral converges to a function $f\in
Q_{\alpha}^{\beta}$ and
$$\|f\|_{Q_{\alpha}^{\beta}}\lesssim \|F\|_{T^{\infty}_{\alpha,\beta}}$$
and any $f\in Q_{\alpha,\beta}$ can be thus represented.\\
(ii) The operator $\pi_{\psi}$ initially defined on $F\in
T_{\alpha,\beta}^{1}$ with compact support in $\mathbb{R}^{1+n}_{+}$
extends to a  bounded and surjective operator form
$T^{1}_{\alpha,\beta}$ to $HH^{1}_{-\alpha,\beta}.$
\end{theorem}

\begin{proof}
(i) Taking $f=\pi_{\phi}(F)$, we only need to prove
$\sup_{I}D_{f,\alpha,\beta}(I)<\infty$ where
$$D_{f,\alpha,\beta}(I)=[l(I)]^{2\alpha-n+2\beta-2}\int_{|y|<l(I)}\int_{I}|f(x+y)-f(y)|^{2}\frac{dxdy}{|y|^{n+2(\alpha-\beta+1)}}.$$
Denote the function $x\rightarrow f(x+y)$ by $f_{y}$ and note that
the integral in (\ref{eq pi}) is valid in $\mathscr{S}'(R^{n})$
modulo constants, that is,  when it acts on test functions of
integration zero, we obtain
$$f_{y}-f=\int_{0}^{\infty}[(F(t,\cdot)\ast\phi_{t})_{y}-(F(t,\cdot)\ast\phi_{t})]\frac{dt}{t}\text{ in } \mathscr{S}'(R^{n}).$$
Fix a cube $I$ and $y\in B(0,l(I))$. For any $g\in
C_{0}^{\infty}(I),$  we write
\begin{eqnarray*}
|\langle f_{y}-f, g\rangle|
&\leq&\int_{0}^{|y|}\int_{\mathbb{R}^{n}}|F(t,x)||\phi_{t}\ast(g_{-y}-g)(x)|\frac{dt dx}{t}\\
&&+\int^{l(I)}_{|y|}\int_{\mathbb{R}^{n}}|(F(t,\cdot)\ast\phi_{t})(x+y)-(F(t,\cdot)\ast\phi_{t})(x)||g(x)|\frac{dt dx}{t}\\
&&+\int_{l(I)}^{\infty}\int_{\mathbb{R}^{n}}|F(t,x)||\phi_{t}\ast(g_{-y}-g)(x)|\frac{dt dx}{t}\\
&:=&A_{1}(g,y)+A_{2}(g,y)+A_{3}(g,y).
\end{eqnarray*}
For $A_{1}(g,y)$,  $|y|<l(I)$ verifies that $g_{-y}-g$ is supported
in the dilated cube $3I$. Also if $t\leq |y|$ we have that
$\phi_{t}\ast(g_{y}-g)$ is supported in the large cube $J=5I$. Then
we can get
\begin{eqnarray*}
A_{1}(g,y)
&\leq&\int_{0}^{|y|}\left(\int_{J}|F(t,x)|^{2}dx\right)^{1/2}\|\phi_{t}\ast(g_{-y}-g)\|_{L^{2}}\frac{dt}{t}\\
&\lesssim&\|\phi\|_{L^{1}}\|g\|_{L^{2}(I)}\int_{0}^{|y|}\left(\int_{J}|F(t,x)|^{2}dx\right)^{1/2}\frac{dt}{t}.
\end{eqnarray*}
For $A_{2},$ if $|y|<t$,  by changing  variable $z-y=z$, we get
\begin{eqnarray*}
|(F(t,\cdot)\ast\phi_{t})(x+y)-(F(t,\cdot)\ast\phi_{t})(x)|
&\lesssim&\int_{\mathbb{R}^{n}}|\phi(t^{-1}y+z)-\phi(z)||F(t,x-t z)|dz\\
&\lesssim &t^{-1}|y|\sup_{|\xi|\leq 1}|\nabla\phi(\xi)|\int_{|z|\leq
2}|F(t,x-t z)|dz\\
&\lesssim& C_{\phi}t^{-1}|y|\int_{|z|\leq 2}|F(t,x-t z)|dz
\end{eqnarray*}
with $C_{\phi}=\sup|\nabla\phi|<\infty.$ Fubini's theorem and the
fact that $g$ is supported in $I$ imply that
\begin{eqnarray*}
A_{2}(g,y)&\leq& C_{\phi}|y|\int_{|y|}^{l(I)}\int_{|z|\leq
2}\int_{I}|F(t,x-tz)||g(x)|dxdz\frac{dt}{t^{2}}\\
&\lesssim&CC_{\phi}\|g\|_{L^{2}}|y|\int_{|y|}^{l(I)}\int_{|z|\leq
2}\left(\int_{I}|F(t,x-tz_{t})|^{2}dx\right)^{1/2}dz\frac{dt}{t^{2}}
\end{eqnarray*}
where $|z_{t}|\leq 2$ and $C=Vol(B(0,2))$.

For $A_{3},$ let
$G_{y}(t,x)=\phi_{t}\ast(g_{-y}-g)(x)1_{\{(t,x):t\geq|y|\}}.$ Then
 the inequality  (\ref{atom t ineq}) implies that
$$A_{3}=\int_{\mathbb{R}^{1+n}_{+}}|F(t,x)G_{y}(t,x)|\frac{dtdx}{t}\lesssim \|F\|_{T^{\infty}_{\alpha,\beta}}\|G_{y}\|_{T^{1}_{\alpha,\beta}}$$
if  we claim that $G_{y}\in T^{1}_{\alpha,\beta}.$ To prove
$G_{y}\in T^{1}_{\alpha,\beta},$   we follow the proof of Lemma
\ref{atom dec H lemma} (i) and choose $\omega$ be the same function
as that in Theorem \ref{atom dec H} with
$0<2(\alpha+\beta-1)<\varepsilon<2-4+4\beta$. Note that if
$S_{y}:=\text{supp } (G_{y})$, then we obtain
 $\omega^{-1}(x)\simeq
l(I)^{-\varepsilon}t^{n-2(\alpha+\beta-1)+\varepsilon}$. Hence
\begin{eqnarray*}
&&\int_{\mathbb{R}^{n+1}_{+}}|G_{y}(t,x)|^{2}\omega^{-1}(t,x)\frac{dx
dt}{t^{1-2(\alpha-\beta+1)}}\\
&\leq&
l(I)^{-\varepsilon}\int_{l(I)}^{\infty}\int_{\mathbb{R}^{n}}|\phi_{t}\ast(g_{-y}-g)(x)|^{2}dx
t^{n-2(\alpha+\beta-1)+\varepsilon}\frac{
dt}{t^{1-2(\alpha-\beta+1)}}\\
&\leq&
l(I)^{-\varepsilon}\|g\|_{L^{1}}^{2}\int_{l(I)}^{\infty}\|\phi_{t}^{y}-\phi_{t}\|_{L^{2}}^{2}t^{n-4\beta+3+\varepsilon}
dt\\
&\leq&
l(I)^{n-\varepsilon}\|g\|_{L^{2}}^{2}\int_{l(I)}^{\infty}\int_{\mathbb{R}^{n}}|(\widehat{\phi_{t}^{y}}-\widehat{\phi_{t}})(\xi)|^{2}d\xi
t^{n-4\beta+3+\varepsilon}
dt\\
&\leq&
l(I)^{n-\varepsilon}\|g\|_{L^{2}}^{2}\int_{l(I)}^{\infty}\int_{\mathbb{R}^{n}}|1-e^{2\pi
iy\cdot\xi}|^{2}|\widehat{\phi(t|\xi|)}|^{2}d\xi
t^{n-4\beta+4+\varepsilon}\frac{dt}{t}\\
&\leq&
l(I)^{n-\varepsilon}\|g\|_{L^{2}}^{2}\int_{\mathbb{R}^{n}}\frac{|1-e^{2\pi
iy\cdot\xi}|^{2}}{|\xi|^{n-4\beta+4+\varepsilon}}d\xi\int_{0}^{\infty}|\widehat{\phi(t)}|^{2}t^{n-4\beta+4+\varepsilon}\frac{dt}{t}\\
&\leq&
C_{\phi}l(I)^{n-\varepsilon}\|g\|_{L^{2}}^{2}|y|^{-4\beta+4+\varepsilon}.
\end{eqnarray*}
In  the last inequality we have used the fact:
$$\int_{\mathbb{R}^{n}}\frac{|1-e^{2\pi
iy\xi}|^{2}}{|\xi|^{n-4\beta+4+\varepsilon}}d\xi
\lesssim|y|^{\varepsilon-4\beta+4}.$$
 In fact, we can write
\begin{eqnarray*}
\int_{\mathbb{R}^{n}}\frac{|1-e^{2\pi
iy\xi}|^{2}}{|\xi|^{n-4\beta+4+\varepsilon}}d\xi
&\lesssim&\int_{R^{n}}\frac{|1-e^{2\pi
iy\xi}|^{2}}{|y\xi|^{n-4\beta+4+\varepsilon}}|y|^{n-4\beta+4+\varepsilon}\frac{d(y\xi)}{|y|^{n}}\\
&\lesssim&|y|^{\varepsilon-4\beta+4}\left(\int_{|z|\leq
1}+\int_{|z>1|}\right)\frac{|1-e^{2\pi i
z}|^{2}}{|z|^{n-4\beta+4+\varepsilon}}dz\\
&:=&|y|^{\varepsilon-4\beta+4}(I_{1}+I_{2}).
\end{eqnarray*}
It is easy to see that
\begin{eqnarray*} I_{2}=\int_{|z|\geq 1}\frac{|1-e^{2\pi
i z}|^{2}}{|z|^{n+\varepsilon-4\beta+4}}dz\lesssim\int_{|z|\geq
1}\frac{|z|^{n-1}}{|z|^{n+\varepsilon-4\beta+4}}d|z|\lesssim 1,
\end{eqnarray*}
\begin{eqnarray*} I_{1}&=&\int_{|z|< 1}\frac{|1-e^{2\pi
i
z}|^{2}}{|z|^{n+\varepsilon-4\beta+4}}dz\lesssim\int_{|z|<1}\left|\sum_{k=1}^{\infty}\frac{(2\pi
i
z)^{k-1}}{k!}\right|^{2}\frac{|z|^{2}}{|z|^{n+\varepsilon-4\beta+4}}dz\\
&\lesssim&\int_{|z|<1}|z|^{1-\varepsilon+4\beta-4}d|z| \lesssim 1.
\end{eqnarray*}
Then
$\|G_{y}\|_{T^{1}_{\alpha,\beta}}\leq\|g\|_{L^{2}}\sqrt{l(I)^{n-\varepsilon}|y|^{\varepsilon-4\beta+4}}.$
Thus we get
\begin{eqnarray*}
\|f_{y}-f\|_{L^{2}(I)}&\leq&\sup_{g\in C^{\infty}_{0}(I),\|g\|_{2}\leq 1}|\langle f_{y}-f,g\rangle|\\
&\lesssim&
\int_{0}^{|y|}\left(\int_{J}|F(t,x)|^{2}dx\right)^{1/2}\frac{dt}{t}+|y|\int_{|y|}^{l(I)}\left(\int_{I}|F(t,x-tz_{t})|^{2}dx\right)^{1/2}\frac{dt}{t^{2}}\\
&&+\|F\|_{T^{\infty}_{\alpha,\beta}}l(I)^{(n-\varepsilon)/2}|y|^{\varepsilon/2-2\beta+2}.
\end{eqnarray*}
Then, by Hardy's inequality(see Stein \cite{Stein 1}),   we have
\begin{eqnarray*}
&&\int_{|y|<l(I)}\int_{I}|f(x+y)-f(x)|^{2}\frac{dxdy}{|y|^{n+2(\alpha-\beta+1)}}\\
&\lesssim&\int_{0}^{l(I)}\left(\int_{0}^{s}\left(\int_{J}|F(t,x)|^{2}dx\right)^{1/2}\frac{dt}{t}\right)\frac{ds}{s^{1+2(\alpha-\beta+1)}}\\
&&+\int_{0}^{l(I)}\left(\int_{s}^{l(I)}\left(\int_{I}|F(t,x-tz_{t})|^{2}dx\right)^{1/2}\frac{dt}{t^{2}}\right)^{2}\frac{ds}{s^{2(\alpha-\beta+1)-1}}\\
&&+\|F\|^{2}_{T^{\infty}_{\alpha,\beta}}l(I)^{n-\varepsilon}\int_{0}^{l(I)}\frac{s^{n-1}s^{\varepsilon-4\beta+4}}{s^{n+2(\alpha-\beta+1)}}ds\\
&\lesssim&
\int_{0}^{l(I)}\int_{J}|F(t,x)|^{2}t^{-1-2(\alpha-\beta+1)}dxdt\\
&&+\int_{0}^{l(I)}\int_{I}|F(t,x-tz_{t})|^{2}t^{-1-2(\alpha-\beta+1)}dtdx\\
&&+\|F\|^{2}_{T^{\infty}_{\alpha,\beta}}l(I)^{n-\varepsilon}l(I)^{\varepsilon-2(\alpha+\beta-1)}\\
&\lesssim&
l(I)^{n-2(\alpha+\beta-1)}\|F\|^{2}_{T^{\infty}_{\alpha,\beta}}
\end{eqnarray*}
since for each $t\leq l(I),$ $|z_{t}|\leq 2$ implies
$I-tz_{t}\subset J=5I.$  Then we get
$\sup_{I}D_{f,\alpha,\beta}(I)\lesssim
\|F\|^{2}_{T^{\infty}_{\alpha,\beta}}<\infty,$ that is $f\in
Q^{\beta}_{\alpha}$ and
$\|f\|_{Q_{\alpha}^{\beta}}\lesssim\|F\|_{T^{\infty}_{\alpha,\beta}}.$

(ii) Firstly, we verify  that for a $T^{1}_{\alpha,\beta}-$ atom
$a$, the integral in (\ref{eq pi})  converges in
$\dot{L}^{2}_{-n/2-2+2\beta}$ to a distribution which is a multiple
of an $HH^{1}_{-\alpha,\beta}-$atom. Assume  $a(x,t)$ is supported
in $T(B)$ for some $B.$ For  $\varepsilon>0$, let
$$\pi_{\phi}^{\varepsilon}(a)=\int_{\varepsilon}^{\infty}a(t,\cdot)\ast\phi_{t}(x)\frac{dxdt}{t}$$
and   $T^{\varepsilon}(B)$ be the truncated tent $T(B)\cap\{(t,x):
t>\varepsilon\}.$  The Cauchy-Schwarz inequality and (ii) of Lemma
\ref{Carleson m 1}  imply that
\begin{eqnarray*}
|\langle\pi_{\phi}^{\varepsilon}(a),\psi\rangle|
&\leq&\left(\int_{T^{\varepsilon}(B)}|a(t,x)|^{2}\frac{dxdt}{t^{1-2(\alpha-\beta+1)}}\right)^{1/2}
\left(\int_{T^{\varepsilon}(B)}|\psi\ast\phi_{t}(x)|^{2}\frac{dxdt}{t^{1+2(\alpha-\beta+1)}}\right)^{1/2}\\
&\lesssim&\left(l(\widetilde{B})^{2\alpha-n+2\beta-2}\int_{\widetilde{B}}\int_{\widetilde{B}}\frac{|\psi(x)-\psi(y)|^{2}}{t^{n+2(\alpha-\beta+1)}}dxdy\right)^{1/2}
\end{eqnarray*}
hold  for any $\psi\in \mathscr{S}(\mathbb{R}^{n}),$ where
$\widetilde{B}$ is some fixed dilate of the ball $B$. Since the
right-hand side is dominated by
$\|\psi\|_{Q^{\beta}_{\alpha}}\leq\|\psi\|_{\dot{L}^{2}_{n/2+2-2\beta}},$
the same argument also gives, for
$0<\varepsilon_{1}<\varepsilon_{2}$,
$$|\langle\pi_{\phi}^{\varepsilon_{1}}(a)-\pi_{\phi}^{\varepsilon_{2}}(a),\psi\rangle|\leq\left(\int_{T^{\varepsilon_{1}}(B)\setminus T^{\varepsilon_{2}}(B)}
|a(t,x)|^{2}\frac{dxdt}{t^{1-2(\alpha-\beta+1)}}\right)^{1/2}\|\psi\|_{\dot{L}^{2}_{n/2+2-2\beta}}.
$$
Thus $\pi_{\phi}(a)=\lim_{\varepsilon\longrightarrow
0}\pi^{\varepsilon}_{\phi}(a)$ exists in
$\dot{L}^{2}_{-n/2-2+2\beta}$. This distribution is supported in
$\widetilde{B}$ and satisfies condition (i) of Definition \ref{Def
HH} since $\phi$ satisfies the same condition. Therefore
$\pi_{\phi}(a)$ is a multiple of an $HH^{1}_{-\alpha,\beta}$ atom.
For a function $F=\sum_{j}^{}\lambda_{j}a_{j}$ in
$T^{1}_{\alpha,\beta}$ and a test function $\psi\in
\mathscr{S}(\mathbb{R}^{n})$, by Theorem \ref{atomde T}, we have
\begin{eqnarray*}
\int_{\mathbb{R}^{n+1}_{+}}(F(t,\cdot)\ast\phi_{t})(x)\psi(x)\frac{dxdt}{t}
=\sum_{j}\lambda_{j}\langle\pi_{\phi}a_{j},\psi\rangle
=\left\langle\sum_{j}\lambda_{j}\pi_{\phi}a_{j},\psi\right\rangle,
\end{eqnarray*}
since $\rho_{\phi}(\psi)(t,x)=(\phi_{t}\ast\psi)(x)$ is a function
in $T^{\infty}_{\alpha,\beta}.$ So
$\pi_{\phi}(F)=\sum_{j}\lambda_{j}\pi_{\phi}a_{j}\in
\mathscr{S}'(\mathbb{R}^{n})$
 and
  $$\|\pi_{\phi}(F)\|_{HH^{-1}_{-\alpha},\beta}\leq\inf\sum_{j}|\lambda_{j}|\approx\|F\|_{T^{1}_{\alpha,\beta}}$$
  the infimun being taken over all possible atomic decompositions of
  $F$ in $T^{1}_{\alpha,\beta}.$ This finishes the proof of Theorem
  \ref{equ carle 1}.
\end{proof}

By Theorem \ref{Carleson m 2}, Lemma \ref{atom dec H lemma} and
Theorem \ref{equ carle 1}, using a similar argument of   Dafni-Xiao
\cite[Theorem 7.1]{G. Dafni J. Xiao}, we can prove the following
duality theorem.
\begin{theorem}
The duality of $HH^{1}_{-\alpha,\beta}(\mathbb{R}^{n})$ is
$Q^{\beta}_{\alpha}(\mathbb{R}^{n})$ in the following sense: if
$g\in Q^{\beta}_{\alpha}$ then the linear functional
$$L(f)=\int_{\mathbb{R}^{n}}f(x)g(x)dx,$$
defined initially for $f\in
HH^{1}_{-\alpha,\beta}(\mathbb{R}^{n})\cap
C^{\infty}_{0}(\mathbb{R}^{n})$, has a bounded extension to all
elements of $HH^{1}_{-\alpha,\beta}(\mathbb{R}^{n})$ with $\|L\|\leq
C\|g\|_{Q^{\beta}_{\alpha}(\mathbb{R}^{n})}.$ Conversely, if $L$ is
a bounded linear functional on
$HH^{1}_{-\alpha,\beta}(\mathbb{R}^{n})$ then there is a function
$g\in Q^{\beta}_{\alpha}(\mathbb{R}^{n})$ so that
$\|g\|_{Q^{\beta}_{\alpha}(\mathbb{R}^{n})}\leq C\|L\|$ and $L$ can
be written in the above form for every $f\in
HH^{1}_{-\alpha,\beta}(\mathbb{R}^{n})\cap
C^{\infty}_{0}(\mathbb{R}^{n}).$
\end{theorem}

\section{Well-Posedness of Generalized Navier-Stokes Equations }
In this section, we deal with the well-posedness for the
 generalized  Navier-Stokes system in the setting of $Q_{\alpha}^{\beta}(\mathbb{R}^{n}).$
  Before stating our main result,
we first introduce a new critical spaces, i.e. the derivative spaces
of $Q_{\alpha}^{\beta}(\mathbb{R}^{n}).$ Then we   establish  some
theorems and lemmas which will be used in the proof of the
well-possedness.

\subsection{Some properties of $Q_{\alpha;\infty}^{\beta,-1}$}
\begin{definition}\label{de 1}
 For $\alpha>0$ and $\max\{\alpha,1/2\}<\beta<1$ with $\alpha+\beta-1\geq0$,  we say that a tempered distribution
 $f\in Q_{\alpha;\infty}^{\beta,-1}$ if and only if
 $$\sup_{x\in\mathbb{R}^{n},r\in(0,\infty)}r^{2\alpha-n+2\beta-2}
\int_{0}^{r^{2\beta}}\int_{|y-x|<r}| K_{t}^\beta\ast
f(y)|^{2}t^{-\frac{\alpha}{\beta}}dydt<\infty.
$$
\end{definition}
\begin{remark}
In Definition \ref{de 1}, if we take $\beta=1$, the space
$Q^{1,-1}_{\alpha,\infty}$ becomes the space
$Q^{-1}_{\alpha,\infty}$ introduced by Xiao in \cite{J. Xiao 1}.
\end{remark}
In the next theorem, we prove an useful characterization of
$Q_{\alpha,\infty}^{\beta,-1}$. For this purpose, we need the
following lemma.
\begin{lemma} \label{le 1}
For $\alpha>0$ and $\max\{\alpha,\frac{1}{2}\}<\beta<1$ with
$\alpha+\beta-1\geq0$, let
$f_{j,k}=\partial_{j}\partial_{k}(-\triangle)^{-1}f(j,k=1,2,\cdots,n)$.
If $f\in Q_{\alpha,\infty}^{\beta,-1}$ for
$\beta\in(\frac{1}{2},1)$, then $f_{j,k}\in
Q_{\alpha,\infty}^{\beta, -1}.$
 \end{lemma}
\begin{proof}
Take $\phi\in C^{\infty}_{0}(\mathbb{R}^{n})$ with $\hbox{supp
}(\phi)\subset B(0,1)=\{x\in \mathbb{R}^{n}: |x|<1\} $ and
$\int_{\mathbb{R}^{n}}\phi(x)dx=1.$ Write
$\phi_{r}(x)=r^{-n}\phi(\frac{x}{r})$ and define
$g_{r}(t,x)=\phi_{r}\ast\partial_{j}\partial_{k}(-\triangle)^{-1}
e^{-t(-\triangle)^\beta}f(x).$ Then
$$e^{-t(-\triangle)^\beta}f_{j,k}(x)=
\partial_{j}\partial_{k}(-\triangle)^{-1}e^{-t(-\triangle)^\beta}f(x)=f_{r}(t,x)+g_{r}(t,x).$$
Since $\dot{B}^{2\beta-1}_{1,1}$ is the predual of the homogeneous
Besov space $\dot{B}^{1-2\beta}_{\infty,\infty}$ and
$Q^{\beta,-1}_{\alpha,\infty}\hookrightarrow
\dot{B}^{1-2\beta}_{\infty,\infty}$(see Remark \ref{R1} and Theorem
\ref{Em X B} below), we have
$$\|g_{r}(t,\cdot)\|_{L^{\infty}}\leq\|\phi\|_{\dot{B}^{2\beta-1}_{1,1}}
\left\|\partial_{j}\partial_{k}(-\triangle)^{-1}e^{-t(-\triangle)^\beta}f\right\|_{\dot{B}^{1-2\beta}_{\infty,\infty}}\lesssim
Cr^{1-2\beta}\|f\|_{\dot{B}^{1-2\beta}_{\infty,\infty}}.$$ Therefore
$$\int_{0}^{r^{2\beta}}\int_{|y-x|<r}|g_{r}(t,y)|^{2}t^{-\alpha/\beta}dydt\lesssim
r^{n-2\alpha-2\beta+2}\|f\|^{2}_{\dot{B}^{1-2\beta}_{\infty,\infty}}
\lesssim
r^{n-2\alpha-2\beta+2}\|f\|^{2}_{Q^{\beta,-1}_{\alpha,\infty}}.$$
  To estimate $f_{r}$ we take $\varphi\in
  C^{\infty}_{0}(\mathbb{R}^{n})$ with $\varphi=1$ on
 $B(0,10)=\{x\in \mathbb{R}^{n}: |x|<10\}$ and define $\varphi_{r,x}=\varphi(\frac{y-x}{r}).$
 Then $f_{r}=F_{r,x}+G_{r,x}$ with
 $$G_{r,x}=\partial_{j}\partial_{k}(-\triangle)^{-1}\varphi_{r,x}e^{-t(-\triangle)^\beta}f-
 \phi_{r}\ast\partial_{j}\partial_{k}(-\triangle)^{-1}\varphi_{r,x}e^{-t(-\triangle)^\beta}f.$$
Using Plancherel's identity, we have
\begin{eqnarray*}
&&\int_{0}^{r^{2\beta}}\|\partial_{j}\partial_{k}(-\triangle)^{-1}\varphi_{r,x}
e^{-t(-\triangle)^\beta}f\|_{L^{2}}^{2}\frac{dt}{t^{\alpha/\beta}}\\
&\lesssim&\int_{0}^{r^{2\beta}}\left(\int_{\mathbb{R}^{n}}\left|\xi_{j}\xi_{k}|\xi|^{-2}
(\varphi_{r,x}e^{-t(-\triangle)^{\beta}}f)\hat(\xi)\right|^{2}d\xi\right)\frac{dt}{t^{\alpha/\beta}}\\
&\lesssim& \int_{0}^{r^{2\beta}}\|(\varphi_{r,x}
e^{-t(-\triangle)^\beta }f)\hat{}\
\|_{L^{2}}^{2}\frac{dt}{t^{\alpha\beta}}\\
&\lesssim&\int_{0}^{r^{2\beta}}\|\varphi_{r,x}e^{-t(-\triangle)^{\beta}}f\|^{2}_{L^{2}}\frac{dt}{t^{\alpha/\beta}}.
\end{eqnarray*}
 Similarly we can prove
 \begin{eqnarray*}
\int_{0}^{r^{2\beta}}\|\phi_{r}\ast
\partial_{j}\partial_{k}(-\triangle)^{-1}\varphi_{r,x}e^{-t(-\triangle)^{\beta}}\|^{2}_{L^{2}}\frac{dt}{t^{\alpha/\beta}}
 \lesssim\int_{0}^{r^{2\beta}}\|\varphi_{r,x}e^{-t(-\triangle)^{\beta}}f\|^{2}_{L^{2}}\frac{dt}{t^{\alpha/\beta}}.
 \end{eqnarray*}
Thus, we obtain
$$\int_{0}^{r^{2\beta}}\|G_{r,\cdot}(t,\cdot)\|^{2}_{L^{2}}\frac{dt}{t^{\alpha/\beta}}
\lesssim\int_{0}^{r^{2\beta}}\|\varphi_{r,x}e^{-t(-\triangle)^{\beta}}f\|^{2}_{L^{2}}\frac{dt}{t^{\alpha/\beta}}.
$$
To bound $F_{r,x},$ noting  that
$$\int_{|y-x|<r}|F_{r,x}(t,y)|^{2}dy\lesssim r^{n+1}
\int_{|w-x|\geq
r}|e^{-t(-\triangle)^{\beta}}f(w)|^{2}|x-w|^{-(n+1)}dw,$$
 we establish
 \begin{eqnarray*}
&&\int_{0}^{r^{2\beta}}\left(\int_{|y-x|<r}|F_{r,x}(t,y)|^{2}dy\right)\frac{dt}{t^{\alpha/\beta}}\\
&\lesssim& r^{n+1}\int_{|w-x|\geq
r}|x-w|^{-(n+1)}\left(\int_{0}^{r^{2\beta}}|e^{-t(-\triangle)^{\beta}}f(w)|^{2}\frac{dt}{t^{\alpha/\beta}}\right)dw\\
&\lesssim&\sum_{k=1}^{\infty}2^{-k(n+1)}\int_{|w-x|\leq
2^{k+1}r}\left(\int_{0}^{r^{2\beta}}|e^{-t(-\triangle)^{\beta}}f(w)|^{2}\frac{dt}{t^{\alpha/\beta}}\right)dw\\
&\lesssim&\sum_{k=1}^{\infty}2^{-k(n+1)}\left(\int_{0}^{(2^{k+1}r)^{2\beta}}\int_{|w-x|\leq
2^{k+1}r}|e^{-t(-\triangle)^{\beta}}f(w)|^{2}\frac{dt}{t^{\alpha/\beta}}\right)dw\\
&\lesssim&r^{n-2\alpha-2\beta+2}\|f\|_{Q_{\alpha;\infty}^{\beta,-1}}\sum_{k=1}^{\infty}2^{-k(2\alpha+2\beta-1)}
\lesssim r^{n-2\alpha-2\beta+2}\|f\|_{Q_{\alpha;\infty}^{\beta,-1}}.
\end{eqnarray*}
Now we have proved that
$$\int_{0}^{r^{2\beta}}\int_{|y-x|<r}|f_{r}(t,y)|^{2}\frac{dydt}{t^{\alpha/\beta}}\lesssim r^{n-2\alpha-2\beta+2}\|f\|^{2}_{Q_{\alpha;\infty}^{\beta,-1}},$$
that is, $f_{j,k}\in Q_{\alpha;\infty}^{\beta,-1}.$
\end{proof}

Using Lemma 4.2, we can prove the following theorem. By this
theorem, we can regard $Q^{\beta,-1}_{\alpha,\infty}$ as derivatives
of $Q^{\beta}_{\alpha}.$
\begin{theorem}  \label{th 2}
 For $\alpha>0$ and $\max\{\alpha,1/2\}<\beta<1$ with $\alpha+\beta-1\geq0$. $Q_{\alpha;\infty}^{\beta, -1}=\nabla\cdot (Q_{\alpha}^{\beta})^{n},$ where a
 tempered distribution $f\in \mathbb{R}^{n}$ belongs to
 $\nabla\cdot (Q_{\alpha}^{\beta})^{n}$ if and only if there are $f_{j}\in
 Q_{\alpha}^{\beta}$ such that $f=\sum_{j=1}^{n}\partial_{j}f_{j}.$
\end{theorem}
\begin{proof} For any $f\in \nabla\cdot(Q_{\alpha}^{\beta})^{n},$
there exist $f_{1}, f_{2},\cdots, f_{n}\in Q_{\alpha}^{\beta}$ such
that $f=\sum_{j=1}^{n}\partial_{j}f_{j}.$ We have
$$\|f\|_{Q_{\alpha;\infty}^{\beta,-1}}\leq \sum_{j=1}^{n}\|\partial_{j}f_{j}\|_{Q_{\alpha;\infty}^{\beta,-1}}
\lesssim\sum_{j=1}^{n}\|f_{j}\|_{Q_{\alpha}^{\beta}}.$$ On the other
hand, if $f\in Q_{\alpha;\infty}^{\beta,-1}$ and
$f_{j,k}=\partial_{j}\partial_{k}(-\triangle)^{-1}f,$ then
$f_{j,k}\in Q_{\alpha;\infty}^{\beta,-1}$ according to Lemma \ref{le
1}. Thus  we have $f_{k}=-\partial_{k}(-\triangle)^{-1}f\in
Q_{\alpha}^{\beta}$ and
$$\widehat{(\sum_{k=1}^{n}\partial_{k}f_{k})}(\xi)=-\sum_{k=1}^{n}i\xi_{k}\widehat{f_{k}}(\xi)=-\sum_{k=1}^{n}i\xi_{k}
\times i\xi_{k}|\xi|^{-2}\widehat{f}(\xi)=\widehat{f}(\xi).$$
\end{proof}
\begin{remark}\label{R1}
$Q_{\alpha;\infty}^{\beta,-1}(\mathbb{R}^{n})$ is critical for
equations (\ref{eq1e}) since
$Q_{\alpha;\infty}^{\beta,-1}(\mathbb{R}^{n})$ is the derivative
space of $Q_{\alpha}^{\beta}(\mathbb{R}^{n})$ and
$Q_{\alpha}^{\beta}(\mathbb{R}^{n})$ is invariant under the scaling
$f(x)\longrightarrow \lambda^{2\beta-2}f(\lambda x).$
\end{remark}

In the following theorem we  apply the arguments in the proof of the
``minimality of $\dot{B}^{0}_{1,1}$" used by Frazier-Jaweth-Weiss in
\cite{Frazier Jawerth Weiss} to prove that
$\dot{B}^{1-2\beta}_{\infty,\infty}$  contains all  critical spaces
for  equations (\ref{eq1e}). The special case $\beta=1$ of this
theorem was proved by Cannone in \cite{Cannone}.

\begin{theorem}\label{Em X B}
If a translation invariant Banach space of tempered distributions
$X$ is a critical space of the generalized Navier-Stokes equations
(\ref{eq1e}). Then $X$ is continuously embedded in the Besov space
$\dot{B}^{1-2\beta}_{\infty,\infty}.$
\end{theorem}
\begin{proof}
It follows from the assumption that $X\hookrightarrow \mathscr{S}'$
and for any $f\in X$
\begin{equation}\label{X em B2}
\|f(\cdot)\|_{X}=\|\lambda^{2\beta-1}f(\lambda\cdot-x_{0})\|_{X},
\lambda>0, x_{0}\in\mathbb{R}^{n}.
\end{equation}
 $X\hookrightarrow \mathscr{S}'$ implies that there exists
a constant $C$ such that
$$|\langle K_{1}^{2\beta},f\rangle|\leq C\|f\|_{X}, \forall f\in X.$$
According to the transformation invariant of $X,$ we have
$$\|e^{-(-\triangle)^{\beta}} f\|_{L^{\infty}}=\|K_{1}^{2\beta}\ast f\|_{L^{\infty}}\leq C\|f\|_{X} \quad\text{for }\forall f\in X.$$
Using the fact  $\widehat{f(\lambda
x)}(\xi)=\lambda^{-n}\widehat{f}(\xi/\lambda),$  the definition of
$e^{-(-\triangle)^{\beta}}f(x)$ and the scaling property (\ref{X em
B2}), we obtain that
$$\lambda^{2\beta-1}\|e^{-\lambda^{2\beta}(-\triangle)^{\beta}}f\|_{L^{\infty}}\leq C\|f\|_{X}.$$
It follows from Miao-Yuan-Zhang \cite[Prorposition 2.1]{C. Miao B.
Yuan B. Zhang} that for $s<0,$ $f\in \dot{B}_{\infty,\infty}^{s}$ if
and only if
 $$\sup_{r>0}r^{-s}\|e^{-r^{2\beta}(-\triangle)^{\beta}}f\|_{L^{\infty}}<\infty.$$
Thus $X\hookrightarrow \dot{B}^{1-2\beta}_{\infty,\infty}.$
\end{proof}

\begin{theorem}\label{Q Contain Bes}
Let $\alpha>0$ and $\max\{\alpha,\frac{1}{2}\}<\beta<1$ with
$\alpha+\beta-1\geq0$. If $1\leq q\leq \infty,$ $2<p<\infty$ and
$\alpha+\beta<1+\frac{n}{p}<2\beta$, then
$\dot{B}_{p,q}^{1+\frac{n}{p}-2\beta}$ and
$\dot{B}_{2,q}^{1+\frac{n}{2}-2\beta}$ are continuously embedded in
$Q_{\alpha;\infty}^{\beta,-1}.$
\end{theorem}
\begin{proof} We first prove  $\dot{B}_{p,q}^{1+\frac{n}{p}-2\beta}\hookrightarrow Q_{\alpha;\infty}^{\beta,-1}.$
Since $\dot{B}_{p,q}^{1+\frac{n}{p}-2\beta}\subset
\dot{B}_{p,\infty}^{1+\frac{n}{p}-2\beta}$. Assume that $q=\infty,$
 it follows
form $1+\frac{n}{p}-2\beta<0$ and Proposition 2.1 of \cite{C. Miao
B. Yuan B. Zhang} that for any
$f\in\dot{B}_{p,\infty}^{1+\frac{n}{p}-2\beta}$,
 $$\sup_{r>0}r^{-(1+\frac{n}{p}-2\beta)/2\beta}\|e^{-r(-\triangle)^{\beta}}f\|_{L^{p}}<\infty.$$
Then we have
\begin{eqnarray*}
&&\int_{0}^{r^{2\beta}}\int_{|y-x|<r}|e^{-t(-\triangle)^{\beta}}f(y)|^{2}t^{-\alpha/\beta}dydt\\
&\lesssim&r^{n(p-2)/p}\int_{0}^{r^{2\beta}}\|e^{-t(-\triangle)^{\beta}}f\|_{L^{p}}^{2}t^{-\alpha/\beta}dt\\
&\lesssim&r^{n(p-2)/p}\int_{0}^{r^{2\beta}}
\left(\sup_{t>0}t^{-(1+\frac{n}{p}-2\beta)/2\beta}\|e^{-t(-\triangle)^{\beta}}f\|_{L^{p}}\right)^{2}
t^{(1+\frac{n}{p}-2\beta)/\beta}t^{-\alpha/\beta}dt\\
&\lesssim&r^{n(p-2)/p}\int_{0}^{r^{2\beta}}
t^{(1+\frac{n}{p}-2\beta)/\beta}t^{-\alpha/\beta}dt\\
&\lesssim&r^{n-2(\alpha+\beta-1)}.
\end{eqnarray*}
Thus $f\in Q_{\alpha;\infty}^{\beta,-1}.$ Now we prove
$\dot{B}_{2,q}^{1+\frac{n}{2}-2\beta}\hookrightarrow
Q_{\alpha;\infty}^{\beta,-1}.$  Since $0< \alpha<\beta$ and
$1/2<\beta<1,$ we can find
 $p\in(2,\infty)$ large enough such that $\alpha+\beta<1+\frac{n}{p}<2\beta$ and
 $1+\frac{n}{2}-2\beta
 =1+\frac{n}{p}-2\beta+n\left(\frac{1}{2}-\frac{1}{p}\right).$
Then (ii) of  Theorem \ref{besov emdeing th} implies
$\dot{B}_{2,q}^{1+\frac{n}{2}-2\beta}\hookrightarrow
\dot{B}_{p,q}^{1+\frac{n}{p}-2\beta}\hookrightarrow
Q_{\alpha;\infty}^{\beta,-1}.$
\end{proof}

\subsection{Several Technical Lemmas}
We prove several technical  lemmas used in the proof of our
well-posedness result.
\begin{lemma} \label{le 2} Given $\alpha\in(0,1).$ For a fixed $T\in (0,\infty]$
and a function $f(\cdot, \cdot)$ on $\mathbb{R}^{1+n}_{+},$ let
$A(t)=\int_{0}^{t}e^{-(t-s)(-\triangle)^{\beta}}(-\triangle)^{\beta}f(s,x)ds.$
Then
\begin{equation}\label{eq1c}
\int_{0}^{T}\|A(t,\cdot)\|^{2}_{L^{2}}\frac{dt}{t^{\alpha/\beta}}\lesssim\int_{0}^{T}\|f(t,\cdot)\|^{2}_{L^{2}}\frac{dt}{t^{\alpha/\beta}}.
\end{equation}
\end{lemma}

\begin{proof}
According to the definition of $e^{-(t-s)(-\triangle)^{\beta}}$, by
Fubini's and Plancehrel's theorem, we have
\begin{eqnarray*}
I_{A}&=&\int_{0}^{\infty}\|A(t,\cdot)\|^{2}_{L^{2}}\frac{dt}{t^{\alpha/\beta}}\\
&=&\int_{0}^{\infty}\|\int_{0}^{t}|\xi|^{2\beta}e^{-(t-s)|\xi|^{2\beta}}\widehat{f(s,\xi)}d\xi\|^{2}_{L^{2}}\frac{dt}{t^{\alpha/\beta}}\\
&\lesssim&
\int_{0}^{\infty}\left(\int_{\mathbb{R}^{n}}
\left(\int_{0}^{t}\frac{|\xi|^{2\beta}}{\exp{(t-s)|\xi|^{2\beta}}}|\widehat{f(s,\xi)}|ds\right)^{2}d\xi\right)\frac{dt}{t^{\alpha/\beta}}\\
&\lesssim&
\int_{\mathbb{R}^{n}}\left(\int_{0}^{\infty}\left(\int_{0}^{\infty}1_{\{0\leq
s\leq
t\}}\frac{|\xi|^{2\beta}}{\exp{(t-s)|\xi|^{2\beta}}}|\widehat{f(s,\xi)}|ds\right)^{2}\frac{dt}{t^{\alpha/\beta}}\right)d\xi\\
&\lesssim&\!\!\!\!\int_{\mathbb{R}^{n}}\!\!\!\left(\!\!\int_{0}^{\infty}\!\!\left(\int_{0}^{\infty}\!\!1_{\{0\leq
s\leq
t\}}\frac{|\xi|^{2\beta}}{e^{{(t-s)|\xi|^{2\beta}}}}ds\right)\left(\!\int_{0}^{t}\frac{|\xi|^{2\beta}}
{e^{(t-s)|\xi|^{2\beta}}}|\widehat{f(s,\xi)}|^{2}ds
\right)ds\frac{dt}{t^{\alpha/\beta}}\right)d\xi.
\end{eqnarray*}
Since $\int_{0}^{t}|\xi|^{2\beta}e^{-(t-s)|\xi|^{2\beta}}ds\leq
e^{-t|\xi|^{2\beta}}(e^{t|\xi|^{2\beta}}-1)\leq 1,$ we have
\begin{eqnarray*}
I_{A}&\lesssim&\int_{\mathbb{R}^{n}}\left(\int_{0}^{\infty}\left(\int_{0}^{t}\frac{|\xi|^{2\beta}}{\exp{(t-s)|\xi|^{2\beta}}}|\widehat{f(s,\xi)}|^{2}ds
\right)ds\frac{dt}{t^{\alpha/\beta}}\right)d\xi\\
&\lesssim&
\int_{\mathbb{R}^{n}}\left(\int_{0}^{\infty}|\widehat{f(s,\xi)}|^{2}e^{s|\xi|^{2\beta}}\left(\int_{s}^{\infty}\frac{|\xi|^{2\beta}}{\exp{t|\xi|^{2\beta}}}
\right)dt\frac{ds}{s^{\alpha/\beta}}\right)d\xi\\
&\lesssim&
\int_{\mathbb{R}^{n}}\left(\int_{0}^{\infty}|\widehat{f(s,\xi)}|^{2}e^{s|\xi|^{2\beta}}(-e^{-t|\xi|^{2\beta}}|_{s}^{\infty})
\frac{ds}{s^{\alpha/\beta}}\right)d\xi\\
&\lesssim& \int_{0}^{\infty}||{f(t,\cdot)}\|_{L^{2}}^{2}
\frac{dt}{s^{\alpha/\beta}}.
\end{eqnarray*}
\end{proof}
\begin{lemma}\label{le5}
For $\beta\in(1/2,1)$ and $N(t,x)$ defined on $(0,1)\times
\mathbb{R}^{n},$ let $A(N)$ be the quantity
$$A(\alpha,\beta, N)=\sup_{x\in \mathbb{R}^{n},r\in (0,1)}r^{2\alpha-n+2\beta-2}\int_{0}^{r^{2\beta}}\int_{|y-x|<r}|f(t,x)|\frac{dxdt}{t^{\alpha/\beta}}.$$
Then  for each $k\in \mathbb{N}_{0}:=\mathbb{N}\cup \{0\}$ there
exists a constant $b(k)$ such that the following inequality holds:
\begin{equation}
\int^{1}_{0}\left\|t^{\frac{k}{2}}(-\triangle)^{\frac{k\beta+1}{2}}e^{-\frac{t}{2}(-\triangle)^{\beta}}\int_{0}^{t}N(s,\cdot)ds\right\|_{L^{2}}^{2}
\frac{dt}{t^{\alpha/\beta}}\leq b(k)A(\alpha,\beta,N)
\int^{1}_{0}\int_{\mathbb{R}^{n}}|N(s,x)|\frac{dxds}{s^{\alpha/\beta}}.
\end{equation}
\end{lemma}
\begin{proof}
Using the inner-product $\langle\cdot,\cdot\rangle$ in $L^{2}$ with
respect to the spatial variable $x\in \mathbb{R}^{n},$ we obtain
\begin{eqnarray*}
&&I=\int^{1}_{0}\left\|t^{\frac{k}{2}}(-\triangle)^{\frac{k\beta+1}{2}}e^{-\frac{t}{2}(-\triangle)^{\beta}}\int_{0}^{t}N(s,\cdot)ds\right\|_{L^{2}}^{2}
\frac{dt}{t^{\alpha/\beta}}\\
&&=\int_{0}^{1}\langle\int_{0}^{t}t^{\frac{k}{2}}(-\triangle)^{\frac{k\beta+1}{2}}e^{-\frac{t}{2}(-\triangle)^{\beta}}N(s,\cdot)d
s,\int_{0}^{t}t^{\frac{k}{2}}(-\triangle)^{\frac{k\beta+1}{2}}e^{-\frac{t}{2}(-\triangle)^{\beta}}N(h,\cdot)d
h\rangle_{L^{2}}\frac{dt}{t^{\alpha/\beta}}\\
&&=2\mathcal{R}e\left( \int\int_{0<h<s<1}\langle
N(s,\cdot),\int_{s}^{1}t^{k}(-\triangle)^{k\beta+1}e^{-t(-\triangle)^{\beta}}N(h,\cdot)dt\rangle_{L^{2}}\frac{dhds}{s^{\alpha/\beta}}\right)\\
&&=2\mathcal{R}e\left(\int\int_{0<h<s<1}\langle
N(s,\cdot),(-\triangle)^{1-\beta}\int_{s}^{1}(t(-\triangle)^{\beta})^{k}
e^{-t(-\triangle)^{\beta}}N(h,\cdot)d(t(-\triangle^{\beta}))\rangle_{L^{2}}\frac{dhds}{s^{\alpha/\beta}}\right)\\
&&\leq\int_{0}^{1}\left|\langle
N(s,\cdot),(-\triangle)^{1-\beta}\int_{0}^{s}(L_{k}(1)-L_{k}(s))N(h,\cdot)dh\rangle_{L^{2}}\right|\frac{ds}{s^{\alpha/\beta}}
\end{eqnarray*}
where
$L_{k}(t)=\sum_{m=0}^{k}b_{m}(k)t^{m}(-\triangle)^{m\beta}e^{-t(-\triangle)^{\beta}}.$

We consider the $\nu-$th derivative of the kernel $K^{\beta}_{t}(x)$
and let
$$(K_{1}^{\beta})^{\nu}(x)=(-\triangle)^{\nu/2}K_{1}^{\beta}(x)\quad \text{ and }
\quad
(K_{t}^{\beta})^{\nu}(x)=(-\triangle)^{\nu/2}K_{t}^{\beta}(x).$$
Using the estimates
$$(K_{1}^{\beta})^{\nu}(x)\lesssim\frac{1}{(1+|x|)^{n+\nu}}\quad\text{
and }\quad
(K_{t}^{\beta})^{\nu}(x)=t^{-\frac{\nu}{2\beta}}t^{-\frac{n}{2\beta}}(K^{\beta}_{1})^{\nu}\left(\frac{x}{t^{1/2\beta}}\right)$$
(see Miao-Yuan-Zhang\cite[Lemma 2.2 and Remark 2.1]{C. Miao B. Yuan
B. Zhang}), we get the kernel of the above operator satisfies the
estimate:
\begin{eqnarray*}
(-\triangle)^{1-\beta}L_{k}(t)(x,y)&\lesssim&\sum_{m=0}^{k}t^{m-\frac{2m\beta+n+2-2\beta}{2\beta}}
\frac{b_{m}(k)}{\left(1+t^{-1/2\beta}|x-y|\right)^{n+2m\beta+2-2\beta}}\\
&\lesssim&
t^{-\frac{n+2-2\beta}{2\beta}}\sum_{m=0}^{k}\frac{b_{m}(k)}{\left(1+t^{-1/2\beta}|x-y|\right)^{n+2m\beta+2-2\beta}},
\end{eqnarray*}
 we have
\begin{eqnarray*}
&&\left|\int_{0}^{s}(-\triangle)^{1-\beta}L_{k}(s)N(h,x)dh\right|\\
&\lesssim&
s^{-\frac{n+2-2\beta}{2\beta}}\int^{s}_{0}\int_{\mathbb{R}^{n}}\sum_{m=0}^{k}b_{m}(k)
\frac{|N(h,y)|dydh}{\left(1+s^{-1/2\beta}|x-y|\right)^{n+2m\beta+2-2\beta}}\\
&\lesssim&s^{-\frac{n+2-2\beta}{2\beta}}\sum_{m=0}^{k}b_{m}(k)\sum_{k\in\mathbb{Z}^{n}}\int_{0}^{s}\int_{\frac{x-y}{t^{1/2\beta}}\in
k+[0,1]^{n}}\frac{|N(h,y)|dydh}{\left(1+s^{-1/2\beta}|x-y|\right)^{n+2m\beta+2-2\beta}}\\
&\lesssim&b(k)\sup_{x\in
R^{n}}\sup_{0<t<1}t^{-\frac{n+2-2\beta}{2\beta}}\int_{0}^{t}\int_{|x-y|<t^{1/2\beta}}|N(h,y)|dydh\\
&\lesssim&b(k)\sup_{x\in
R^{n}}\sup_{0<\rho<1}\rho^{2\alpha-n+2\beta-2}\int_{0}^{\rho^{2\beta}}\int_{|x-y|<\rho}|N(h,y)|\frac{dydh}{h^{\alpha/\beta}}.
\end{eqnarray*}
Hence we can get
$$I\lesssim b(k)\left(\int_{0}^{1}\int_{\mathbb{R}^{n}}|N(s,x)|\frac{dsdx}{s^{\alpha/\beta}}\right)A(\alpha,\beta,N).$$ This completes
the proof.
\end{proof}
\begin{remark}
Similarly when $k=0$, we can prove the following inequality:
\begin{equation}
\int^{1}_{0}\left\|(-\triangle)^{\frac{1}{2}}e^{-t(-\triangle)^{\beta}}\int_{0}^{t}N(s,\cdot)ds\right\|_{L^{2}}^{2}
\frac{dt}{t^{\alpha/\beta}}\lesssim A(\alpha,\beta,N)
\int^{1}_{0}\int_{\mathbb{R}^{n}}|N(s,x)|\frac{dxds}{s^{\alpha/\beta}}.
\end{equation}

\end{remark}
\begin{lemma}{\label{le4}} For $1\leq j,k\leq n$ and $t>0,$ the operator $Q_{j,k,t}^{\beta}=\frac{1}{\triangle}\partial_{j}\partial_{k}e^{-t(-\triangle)^{\beta}}$
is a convolution operator with the kernel
$K_{j,k,t}^{\beta}(x)=\frac{1}{t^{{n}/{2\beta}}}K_{j,k}^{\beta}(\frac{x}{t^{{1}/{2\beta}}})$
for a smooth function $K_{j,k}^{\beta}$ such that for all $\alpha\in
\mathbb{N}^{n}$
$$(1+|x|)^{n+|\alpha|}\partial^{\alpha}K_{j,k}^{\beta}\in L^{\infty}(\mathbb{R}^{n}).$$
\end{lemma}
\begin{proof}
Since
$\widehat{K_{j,k}^{\beta}}(\xi)=\frac{\xi_{j}\xi_{k}}{|\xi|^{2}}e^{-|\xi|^{2\beta}},$
we have
$\widehat{\partial^{\alpha}K_{j,k}^{\beta}}(\xi)\lesssim|\xi|^{|\alpha|}\frac{\xi_{j}\xi_{k}}{|\xi|^{2}}e^{-|\xi|^{2\beta}}$
and
$\int\widehat{\partial^{\alpha}K_{j,k}^{\beta}}(\xi)d\xi<\infty.$
Thus $\partial^{\alpha}K_{j,k}^{\beta}(x)\in
L^{\infty}(\mathbb{R}^{n}).$

For $|x|\leq1$, we have
$$|(1+|x|)^{n+|\alpha|}\partial^{\alpha}K_{j,k}^{\beta}(x)|\lesssim |\partial^{\alpha}K_{j,k}(x)|\lesssim1.$$
For $|x|>1$, we write
$K_{j,k}^{\beta}=(I-S_{0})K_{j,k}^{\beta}+\sum_{l<0}\Delta_{l}K_{j,k}^{\beta}$
where $(I-S_{0})K_{j,k}^{\beta}\in \mathcal {S}$ and
$\Delta_{l}K_{j,k}^{\beta}=2^{ln}\omega_{j,k,l}^{\beta}(2^{l}x)$
with
$\widehat{\omega_{j,k,l}^{\beta}}=\psi(\xi)\frac{\xi_{j}\xi_{k}}{|\xi|^{2}}e^{-|2^{l}\xi|^{2\beta}}\in
L^{1}.$
 Then  the set $\{\omega_{j,k,l}^{\beta}:l<0\}$ is bounded in $\mathcal
 {S}$ and there exists an uniform constant $C_{N}$ such that
$$(1+2^{l}|x|)^{N}2^{l(n+|\alpha|)}|\partial^{\alpha}\triangle_{l}K_{j,k}^{\beta}(x)|\leq C_{N}.$$
 Thus
$$|\partial^{\alpha}S_{0}K_{j,k}(x)|\lesssim\sum_{2^{l}|x|\leq1}2^{l(n+|\alpha|)}+\sum_{2^{l}|x|>1}2^{l(n+|\alpha|-N)}|x|^{-N}\lesssim|x|^{-n-|\alpha|} .$$
\end{proof}

\subsection{Well-Posedness}
In this subsection, we establish the well-posedness result for the
solutions to the equations (\ref{eq1e}). Throughout this subsection,
we always assume $\beta\in(\frac{1}{2},1)$. In fact our results can
be also applied to the case $\beta=1$, that is, the classical
Naiver-Stokes equations. Hence our results can be regarded as a
generalization of the result of Koch-Tataru  \cite{H. Koch D.
Tataru} when $\alpha=0,\beta=1$ and that of Xiao \cite{J. Xiao 1}
when $\alpha\in(0,1),\beta=1$.
\begin{definition}  \label{X space} Let $\alpha>0$ and
 $\max\{1/2,\alpha\}<\beta<1$ with $\alpha+\beta-1\geq 0$.\\
(i)  A tempered distribution $f$ on $R^{n}$ belongs to
$Q_{\alpha;T}^{\beta,-1}(\mathbb{R}^{n})$ provided
 $$\|f\|_{Q_{\alpha;T}^{\beta,-1}(\mathbb{R}^{n})}=\sup_{x\in\mathbb{R}^{n},r\in(0,T)}\left(r^{2\alpha-n+2\beta-2}
\int_{0}^{r^{2\beta}}\int_{|y-x|<r}| K_{t}^\beta\ast
f(y)|^{2}t^{-\frac{\alpha}{\beta}}dydt\right)^{1/2}<\infty;
$$
 (ii) A tempered distribution $f$ on $R^{n}$ belongs to
$\overline{VQ^{\beta,-1}_{\alpha}}(\mathbb{R}^{n})$ provided
$\lim\limits_{T\longrightarrow
0}\|f\|_{Q^{\beta,-1}_{\alpha;T}(\mathbb{R}^{n})}=0;$\\
 (iii) A function $g$ on $R^{1+n}_{+}$ belongs to the space
$X^{\beta}_{\alpha;T}(\mathbb{R}^{n})$ provided
\begin{eqnarray*}
\|g\|_{X^{\beta}_{\alpha;T}(\mathbb{R}^{n})}&=&\sup_{t\in(0,T)}t^{1-\frac{1}{2\beta}}\|g(t,\cdot)\|_{L^{\infty}(\mathbb{R}^{n})}\\
&+&\sup_{x\in
R^{n},r^{2\beta}\in(0,T)}\left(r^{2\alpha-n+2\beta-2}\int_{0}^{r^{2\beta}}\int_{|y-x|<r}|g(t,y)|^{2}t^{-\alpha/\beta}dydt\right)^{1/2}<\infty.
\end{eqnarray*}
\end{definition}
\begin{theorem} \label{th 3}
Let $n\geq 2,$ $\alpha>0$ and $\max\{\alpha,1/2\}<\beta<1$ with $\alpha+\beta-1\geq0$. Then\\
(i) The fractional Navier-Stokes system (\ref{eq1e}) has a unique
small global mild solution in $(X^{\beta}_{\alpha;\infty})^{n}$ for
all initial data $a$ with $\nabla\cdot a=0$ and
$\|a\|_{(Q_{\alpha;\infty}^{\beta,-1})^{n}}$ being small.\\
(ii) For any $T\in(0,\infty)$ there is an $\varepsilon>0$ such that
the fractional Navier-Stokes system (\ref{eq1e}) has a unique small
mild solution in $(X_{\alpha,T}^{\beta})^{n}$ on $(0,T)\times
\mathbb{R}^{n}$ when the initial data $a$ satisfies $\nabla\cdot
a=0$ and $\|a\|_{(Q_{\alpha;T}^{\beta,-1})^{n}}\leq \varepsilon.$ In
particular for all $a\in (\overline{VQ_{\alpha}^{\beta,-1}})^{n}$
with $\nabla\cdot a=0$ there exists a unique small local mild
solution in $(X_{\alpha;T}^{\beta})^{n}$ on $(0,T)\times
\mathbb{R}^{n}.$
 \end{theorem}

 \begin{proof}
By Picard's contraction principle, it sufficient to verify the
bilinear operator
 $$B(u,v)=\int^{t}_{0}e^{-(t-s)(-\triangle)^{\beta}}P\nabla\cdot(u\otimes v)ds$$
is bounded from
$(X_{\alpha;T}^{\beta})^{n}\times(X_{\alpha;T}^{\beta})^{n}$ to
$(X_{\alpha;T}^{\beta})^{n}.$\\
 {\it{Part 1. $L^{2}-$bound.}} We want to establish  that
 if $x\in \mathbb{R}^{n}$ and $r^{2\beta}\in (0,T)$ then
 \begin{equation}
r^{2\alpha-n+2\beta-2}\int_{0}^{r^{2\beta}}\int_{|y-x|<r}|B(u,v)|^{2}dy\frac{ds}{s^{\alpha/\beta}}\lesssim
\|u\|^{2}_{(X_{\alpha;T}^{\beta})^{n}}\|v\|^{2}_{(X_{\alpha;T}^{\beta})^{n}}.
 \end{equation}
To this aim, define $1_{r,x}(y)=1_{|y-x|<10r}(y),$ i.e., the
indicate function on the ball $\{y\in \mathbb{R}^{n}: |y-x|<10r\}.$
We divide $B(u,v)$ into three parts:
$B(u,v)=B_{1}(u,v)+B_{2}(u,v)+B_{3}(u,v),$ where
\begin{eqnarray*}
&&B_{1}(u,v)=\int_{0}^{s}e^{-(s-h)(-\triangle)^{\beta}}P\nabla\cdot
((1-1_{r,x})u\otimes v)dh,\\
&&B_{2}(u,v)=(-\triangle)^{-1/2}P\nabla\cdot\int_{0}^{s}e^{-(s-h)(-\triangle)^{\beta}}(-\triangle)
((-\triangle)^{-1/2}(I-e^{-h(-\triangle)^{\beta}})(1_{r,x})u\otimes
v)dh\\
&&B_{3}(u,v)=(-\triangle)^{-1/2}P\nabla\cdot(-\triangle)^{1/2}e^{-s(-\triangle)^{\beta}}\int_{0}^{s}((1_{r,x})u\otimes
 v)dh.
 \end{eqnarray*}
At first, we estimate $B_{2}(u,v)$ as
 \begin{eqnarray*}
 I&=&\int_{0}^{r^{2\beta}}\|B_{2}(u,v)\|_{L^{2}}^{2}\frac{dt}{t^{\alpha/\beta}}\\
 &\lesssim&\int_{0}^{r^{2\beta}}\|\int_{0}^{s}e^{-(s-h)(-\triangle)^{\beta}}(-\triangle)
((-\triangle)^{-1/2}(I-e^{-h(-\triangle)^{\beta}})(1_{r,x})u\otimes
v)dh\|_{L^{2}}^{2}\frac{dt}{t^{\alpha/\beta}}\\
 &\lesssim&\int_{0}^{r^{2\beta}}\|\int_{0}^{s}e^{-(s-h)(-\triangle)^{\beta}}(-\triangle)^{\beta}
((-\triangle)^{1/2-\beta}(I-e^{-h(-\triangle)^{\beta}})(1_{r,x})u\otimes
v)dh\|_{L^{2}}^{2}\frac{dt}{t^{\alpha/\beta}}\\
&\lesssim&\int_{0}^{r^{2\beta}}\|
(-\triangle)^{1/2-\beta}(I-e^{-h(-\triangle)^{\beta}})(1_{r,x})u\otimes
vdh\|_{L^{2}}^{2}\frac{dt}{t^{\alpha/\beta}}.
 \end{eqnarray*}
 Since
 $\sup_{s\in(0,\infty)}s^{1-2\beta}(1-e^{-s^{2\beta}})<\infty$
 for $\frac{1}{2}<\beta<1,$
 we can obtain that
 $(-\triangle)^{1/2-\beta}(I-e^{-s(-\triangle)^{\beta}})$ is bounded
 on $L^{2}$ with operator norm $\lesssim s^{1-\frac{1}{2\beta}}.$
 Write $(1_{r,x})u(s,x)\otimes v(s,x)=M(s,x)$.
 Thus, using the Cahchy-Schwarz inequality, we have
 \begin{eqnarray*}
 I&\lesssim&\int^{r^{2\beta}}_{0}s^{2-\frac{1}{\beta}}\|M(s,\cdot)\|^{2}_{L^{2}}\frac{ds}{s^{\alpha/\beta}}\\
&\lesssim&\int^{r^{2\beta}}_{0}s^{2-\frac{1}{\beta}}\int_{|y-x|<r}|u(s,y)v(s,y)|^{2}dy\frac{ds}{s^{\alpha/\beta}}\\
&\lesssim&\left(\sup_{s\in(0,T)}s^{1-\frac{1}{2\beta}}\|u(s,y)\|_{\infty}\right)
\left(\sup_{s\in(0,T)}s^{1-\frac{1}{2\beta}}\|v(s,y)\|_{\infty}\right)\\
&&\times\left(\int^{r^{2\beta}}_{0}\int_{|y-x|<r}|u(s,y)|^{2}dy\frac{ds}{s^{\alpha/\beta}}\right)
\left(\int^{r^{2\beta}}_{0}\int_{|y-x|<r}|v(s,y)|^{2}dy\frac{ds}{s^{\alpha/\beta}}\right)\\
&\lesssim&
r^{n-2\alpha-2(\beta-1)}\|u\|^{2}_{(X_{\alpha;T}^{\beta})^{n}}\|v\|^{2}_{(X_{\alpha;T}^{\beta})^{n}}.
\end{eqnarray*}
Now by Lemma \ref{le5} with $k=0$, we estimate the term $B_{3}$ as
follows.
\begin{eqnarray*}
&&\int_{0}^{r^{2\beta}}\|B_{3}(u,v)\|_{L^{2}}^{2}\frac{dt}{t^{\alpha/\beta}}\\
&\lesssim&
\int_{0}^{r^{2\beta}}\left\|(-\triangle)^{1/2}e^{-t(-\triangle)^{\beta}}(\int_{0}^{t}M(s,\cdot)dh)\right\|_{L^{2}}^{2}\frac{dt}{t^{\alpha/\beta}}\\
&\lesssim&
r^{n-2\alpha+6\beta-2}\int_{0}^{1}\left\|(-\triangle)^{1/2}e^{-\tau(-\triangle)^{\beta}}
(\int_{0}^{\tau}M(r^{2\beta}\theta,r\cdot)d\theta)\right\|_{L^{2}}^{2}\frac{d\tau}{\tau^{\alpha/\beta}}\\
&\lesssim&
r^{n-2\alpha+6\beta-2}(\int_{0}^{1}\|M(r^{2\beta}s,r\cdot)\|_{L^{1}}\frac{ds}{s^{\alpha/\beta}})C(\alpha,\beta;f)\\
&=&r^{n-2\alpha+6\beta-2}\times II\times A(\alpha,\beta;
M(r^{2\beta} s,ry)).
\end{eqnarray*}
For $II$, we have
\begin{eqnarray*}
II&=&r^{2\alpha-n-2\beta}\int^{r^{2\beta}}_{0}\int_{|z-x|<r}|M(t,z)|\frac{dzdt}{t^{\alpha/\beta}}\\
&\lesssim&r^{2-4\beta}\|u\|_{(X_{\alpha;
T}^{\beta})^{n}}\|v\|_{(X_{\alpha; T}^{\beta})^{n}}.
\end{eqnarray*}
For $C(\alpha,\beta;M(r^{2\beta} s,ry)),$ we have
\begin{eqnarray*}
C(\alpha,\beta;M(r^{2\beta}
s,ry))&\lesssim&\rho^{2\alpha-n+2(\beta-1)}\int_{0}^{\rho^{2\beta}}\int_{|y-x|<\rho}|M(r^{2\beta
}s,ry)|\frac{dyds}{s^{\alpha/\beta}}\\
&\lesssim&\rho^{2\alpha-n+2(\beta-1)}r^{2\alpha-n-2\beta}\int_{0}^{(r\rho)^{2\beta}}\int_{|z-x|<r\rho}|M(t,z)|\frac{dzdt}{t^{\alpha/\beta}}
\\
&\lesssim&r^{2-4\beta}(r\rho)^{2\alpha-n+2(\beta-1)}\int_{0}^{(r\rho)^{2\beta}}\int_{|z-x|<r\rho}|M(t,z)|\frac{dzdt}{t^{\alpha/\beta}}\\
&\lesssim&r^{2-4\beta}\|u\|_{(X_{\alpha;
T}^{\beta})^{n}}\|v\|_{(X_{\alpha; T}^{\beta})^{n}}.
\end{eqnarray*}
 Therefore we get
 $$\int_{0}^{r^{2\beta}}\|B_{3}(u,v)\|_{L^{2}}^{2}\frac{dt}{t^{\alpha/\beta}}
 \lesssim r^{n-2\alpha+6\beta-2}r^{2-4\beta}r^{2-4\beta}\|u\|^{2}_{(X^{\beta}_{\alpha;T})^{n}}\|v\|^{2}_{(X^{\beta}_{\alpha;T})^{n}}
 =r^{n-2\alpha-2\beta+2}\|u\|^{2}_{(X^{\beta}_{\alpha,T})^{n}}\|v\|^{2}_{(X^{\beta}_{\alpha;T})^{n}},$$
that is,
$$r^{2\alpha-n+2(\beta-2)}\int^{r^{2\beta}}_{0}\|B_{3}(u,v)\|^{2}_{L^{2}}\frac{dt}{t^{\alpha/\beta}}
\lesssim\|u\|^{2}_{(X^{\beta}_{\alpha,T})^{n}}\|v\|^{2}_{(X^{\beta}_{\alpha;T})^{n}}.$$
For the estimate of $B_{1}$. According to Lemma \ref{le4}, we have
\begin{eqnarray*}
&&e^{-t(-\triangle)^{\beta}}P\nabla\cdot f(x)=\int \nabla
K_{j,k,t}^{\beta}(x-y)f(y)dy\\
\text{ and}&&\quad\nabla K_{j,k,t}^{\beta}(x-y)\lesssim
\frac{1}{t^{\frac{n}{2\beta}+\frac{1}{2\beta}}}\frac{1}{\left(1+t^{-1/2\beta}|x-y|\right)^{n+1}}
\lesssim\frac{1}{(t^{1/2\beta}+|x-y|)^{n+1}}.
\end{eqnarray*}
\text{Thus}\quad
 \begin{eqnarray*}
 |B_{1}(u,v)|&\leq&\left|\int_{0}^{s}e^{-(s-h)(-\triangle)^{\beta}}P\nabla\cdot ((1-1_{r,x})u\otimes
 v)dh\right|\\
 &\lesssim&\int_{0}^{s}\int_{|z-x|\geq 10
 r}\frac{|u(h,z)||v)(h,z)|}{((s-h)^{1/2\beta}+|z-y|)^{n+1}}dzdh.
\end{eqnarray*}
When $|z-x|\geq 10r,$ $0<s<r^{2\beta}$  and $|y-x|<r,$ we have
$|y-z|\geq |z-x|-|y-x|\geq 9r >9|y-x|. $ Thus $|x-z|\leq
|x-y|+|y-z|\leq \frac{1}{9}|y-z|+|y-z|=\frac{10}{9}|y-z|.$ This
gives us
\begin{eqnarray*}
|B_{1}(u,v)|&\lesssim&\int_{0}^{r^{2\beta}}\int_{|z-x|\geq
10r}\frac{|u(h,z)||v(h,z)|}{|x-z|^{n+1}}dzdh=I_{1}\times I_{2}.
\end{eqnarray*}
where
\begin{eqnarray*}
I_{1}&=&\left(\int_{0}^{r^{2\beta}}\int_{|z-x|\geq
10r}\frac{|u(h,z)|^{2}}{|x-z|^{n+1}}dzdh\right)^{1/2}\\
&\lesssim&\left(\sum_{j=3}^{\infty}\int_{0}^{r^{2\beta}}\int_{2^{j}r\leq
|z-x|\leq
2^{j+1}r}\frac{|u(h,z)|^{2}}{(2^{j}r)^{n+1}}dzdh\right)^{1/2}\\
&\lesssim&\left(\sum_{j=3}^{\infty}\frac{1}{(2^{j}r)^{n+1}}(r^{2\beta})^{\alpha/\beta}(2^{j}r)^{2\beta-2}(2^{j}r)^{2-2\beta}
\int_{0}^{r^{2\beta}}\int_{2^{j}r\leq |z-x|\leq
2^{j+1}r}|u(h,z)|^{2}\frac{dzdh}{h^{\alpha/\beta}}\right)^{1/2}\\
&\lesssim&\left(\sum_{j=3}^{\infty}{(2^{j}r)^{2\alpha-n}}(2^{j}r)^{-1}(2^{j}r)^{2\beta-2}(2^{j}r)^{2-2\beta}
\int_{0}^{r^{2\beta}}\int_{ |z-x|\leq
2^{j+1}r}|u(h,z)|^{2}\frac{dzdh}{h^{\alpha/\beta}}\right)^{1/2}\\
&\lesssim&\left(\frac{1}{r^{2\beta-1}}\right)^{1/2}\|u\|_{(X_{\alpha;
T}^{\beta})^{n}}.
\end{eqnarray*}
Similarly, we obtain $I_{2}\lesssim
\left(\frac{1}{r^{2\beta-1}}\right)^{1/2}\|v\|_{(X^{\beta}_{\alpha;T})^{n}}.$
Thus $|B_{1}(u,v)|\lesssim \frac{1}{r^{2\beta-1}}\|u\|_{(X_{\alpha;
T}^{\beta})^{n}}\|v\|_{(X_{\alpha; T}^{\beta})^{n}}.$ When
$0<\alpha<\beta,$ we have
\begin{eqnarray*}
\int_{0}^{r^{2\beta}}\int_{|y-x|<r}|B_{1}(u,v)|^{2}\frac{dydt}{t^{\alpha/\beta}}&\lesssim&
\frac{1}{r^{4\beta-2}}r^{n}\int^{r^{2\beta}}_{0}\frac{dt}{t^{\alpha/\beta}}\|u\|^{2}_{(X_{\alpha; T}^{\beta})^{n}}\|v\|^{2}_{(X_{\alpha; T}^{\beta})^{n}}\\
&\lesssim&r^{n-2\alpha-2\beta+2}\|u\|^{2}_{(X_{\alpha;
T}^{\beta})^{n}}\|v\|^{2}_{(X_{\alpha;
T}^{\beta})^{n}}.\end{eqnarray*} This implies that
$$r^{2\alpha-n+2(\beta-1)}\int_{0}^{r^{2\beta}}\int_{|y-x|<r}|B_{1}(u,v)|^{2}\frac{dydt}{t^{\alpha/\beta}}\lesssim\|u\|^{2}_{(X_{\alpha; T}^{\beta})^{n}}
\|v\|^{2}_{(X_{\alpha; T}^{\beta})^{n}}.$$
 {\it{Part 2. $L^{\infty}-$bound.}} The aim of this part is to prove
 $$\|B(u,v)\|_{L^{\infty}}\lesssim t^{\frac{1}{2\beta}-1}\|u\|_{(X_{\alpha; T}^{\beta})^{n}}\|v\|_{(X_{\alpha; T}^{\beta})^{n}}, \forall t\in (0,T).$$
  If $\frac{t}{2}\leq s<t$ then
 $$\|e^{-(t-s)(-\triangle)^{\beta}}P\nabla\cdot (u\otimes v)\|_{L^{\infty}}\lesssim\frac{\|u\|_{L^{\infty}}\|v\|_{L^{\infty}}}{(t-s)^{\frac{1}{2\beta}}}
 \lesssim(t-s)^{-\frac{1}{2\beta}}s^{\frac{1}{\beta}-2}\|u\|_{(X_{\alpha; T}^{\beta})^{n}}\|v\|_{(X_{\alpha; T}^{\beta})^{n}}.$$
 If $0<s<\frac{t}{2}$ then $t-s\approx t$ and so
 \begin{eqnarray*}
 |e^{-(t-s)(-\triangle)^{\beta}}P\nabla\cdot(u\otimes v)|
 &\lesssim&\int_{\mathbb{R}^{n}}\frac{|u(s,y)||v(s,y)|}{\left((t-s)^{\frac{1}{2\beta}}+|x-y|\right)^{n+1}}dy\\
 &\lesssim&\int_{\mathbb{R}^{n}}\frac{|u(s,y)||v(s,y)|}{\left(t^{\frac{1}{2\beta}}+|x-y|\right)^{n+1}}dy\\
&\lesssim&\sum_{k\in \mathbb{Z}^{n}}\int_{x-y\in
t^{\frac{1}{2\beta}}(k+[0,1]^{n})}\frac{|u(s,y)||v(s,y)|}{(t^{\frac{1}{2\beta}}(1+|k|))^{(n+1)}}dyds.
 \end{eqnarray*}
 This gives us
 \begin{eqnarray*}
 |B(u,v)|&\lesssim&\int_{0}^{t/2}|e^{-(t-s)(-\triangle)^{\beta}}P\nabla\cdot (u\otimes
 v)|ds+\int_{t/2}^{t}|e^{-(t-s)(-\triangle)^{\beta}}P\nabla\cdot (u\otimes
 v)|ds\\
 &\lesssim&\sum_{k\in
\mathbb{Z}^{n}}(t^{\frac{1}{2\beta}}(1+|k|))^{-(n+1)}\int_{0}^{t/2}\int_{x-y\in
t^{\frac{1}{2\beta}}(k+[0,1]^{n})}|u(s,y)||v(s,y)|dy\\
&&+\int^{t}_{t/2}(t-s)^{-\frac{1}{2\beta}}s^{\frac{1}{\beta}-2}ds\|u\|_{(X_{\alpha;
T}^{\beta})^{n}}\|v\|_{(X_{\alpha; T}^{\beta})^{n}}\\
&:=&I_{3}+I_{4}.
\end{eqnarray*} Here,
\begin{eqnarray*}
I_{4}&\lesssim&\int^{t}_{t/2}(t-s)^{-\frac{1}{2\beta}}s^{\frac{1}{\beta}-2}ds\|u\|_{(X_{\alpha;
T}^{\beta})^{n}}\|v\|_{(X_{\alpha; T}^{\beta})^{n}}\\
&\lesssim&t^{\frac{1}{\beta}-2}t^{1-\frac{1}{2\beta}}\|u\|_{(X_{\alpha;
T}^{\beta})^{n}}\|v\|_{(X_{\alpha; T}^{\beta})^{n}}\\
&\lesssim&t^{\frac{1}{2\beta}-1}\|u\|_{(X_{\alpha;
T}^{\beta})^{n}}\|v\|_{(X_{\alpha; T}^{\beta})^{n}}.
\end{eqnarray*}
On the other hand, we have
\begin{eqnarray*}
I_{3}&\lesssim&\sum_{k\in
\mathbb{Z}^{n}}(t^{\frac{1}{2\beta}}(1+|k|))^{-(n+1)}\left(\int_{0}^{t/2}\int_{|x-y|\lesssim
t^{\frac{1}{2\beta}}}|u(s,y)|^{2}dyds\right)^{1/2}\\
&&\times\left(\int_{0}^{t/2}\int_{|x-y|\lesssim
t^{\frac{1}{2\beta}}}|v(s,y)|^{2}dyds\right)^{1/2}\\
&:=&\sum_{k\in
\mathbb{Z}^{n}}(t^{\frac{1}{2\beta}}(1+|k|))^{-(n+1)}I_{3,1}\times
I_{3,2}.
\end{eqnarray*}
Here,
\begin{eqnarray*}
I_{3,1}&=&\left(\int_{0}^{t/2}\int_{|x-y|\lesssim
t^{\frac{1}{2\beta}}}|u(s,y)|^{2}dyds\right)^{1/2}\\
&=&\left(t^{\frac{1}{2\beta}(n-2\beta+2)}t^{\frac{1}{2\beta}(2\alpha-n+2\beta-2)}\int_{0}^{t/2}\int_{|x-y|\lesssim
t^{\frac{1}{2\beta}}}|u(s,y)|^{2}\frac{dyds}{s^{\alpha/\beta}}\right)^{1/2}\\
&\lesssim&t^{\frac{1}{4\beta}(n-2\beta+2)}\|u\|_{(X_{\alpha;
T}^{\beta})^{n}}.
\end{eqnarray*}
Similarly, we get $I_{3,2}\lesssim
t^{\frac{1}{4\beta}(n-2\beta+2)}\|v\|_{(X_{\alpha;
T}^{\beta})^{n}}.$
 These estimates about $I_{3,1}$ and $I_{3,2}$ imply that
 \begin{eqnarray*}
 I_{3}\lesssim t^{-\frac{1}{2\beta}(n+1)}t^{\frac{1}{2\beta}(n-2\beta+2)}\|u\|_{(X_{\alpha;
T}^{\beta})^{n}}\|v\|_{(X_{\alpha; T}^{\beta})^{n}} \lesssim
t^{\frac{1}{2\beta}-1}\|u\|_{(X_{\alpha;
T}^{\beta})^{n}}\|v\|_{(X_{\alpha; T}^{\beta})^{n}}.
 \end{eqnarray*}
 Thus $t^{1-\frac{1}{2\beta}}\|B(u,v)\|_{L^{\infty}}\lesssim\|u\|_{(X_{\alpha;
T}^{\beta})^{n}}\|v\|_{(X_{\alpha; T}^{\beta})^{n}}.$

Therefore,   we establish the boundedness of $B(u,v)$ and finish the
proof of (i) and (ii) by taking $T=\infty$ and $T\in (0,\infty),$
respectively.

\end{proof}
\section{Regularity of Generalized Navier-Stokes
equations} In this section, we study the regularity of the solutions
to the equations (\ref{eq1e}) with $\beta\in(1/2,1)$. For $\beta=1$,
that is, the classical Naiver-Stokes equations, the regularity has
been studied by several authors.
  In \cite{P. Germain N. Pavlovic
G.Staffilani}, Germain-Pavlovi$\acute{c}$-Staffilani analyzed the
regularity properties of the solutions constructed by Koch-Tataru.
More precisely, they showed that under certain smallness condition
of initial data in $BMO^{-1}$, the solution $u$ to the classical
Naiver-Stokes equations constructed in \cite{H. Koch D. Tataru}
satisfies the following regularity property:
$$t^{\frac{k}{2}}\nabla^{k}u\in X^{0},  \text{for all }  k\in\mathbb{N}_{0}:=\mathbb{N}\cup\{0\}$$
where $X^{0}$ denotes the space where the solution constructed by
Koch and Tataru belongs.

In this section, we establish a similar result for the solutions of
the equations (\ref{eq1e}) evolving  initial data in
$Q_{\alpha;\infty}^{\beta,-1}$ with $\beta\in(1/2,1]$. In fact we
get the solution $u$ to the equations (\ref{eq1e}) satisfies:
$$t^{\frac{k}{2\beta}}\nabla^{k}u\in X^{\beta,0}_{\alpha}\text{ for all } k$$
where $X^{\beta,0}_{\alpha}$ is the space
$X_{\alpha;\infty}^{\beta}$ constructed in (iii) of Definition
\ref{X space} for $\beta\in(1/2,1)$ and $X_{\alpha;\infty}^{1}$ in
Xiao \cite{J. Xiao 1} for $\beta=1.$ For convenience of the study,
we introduce a class of spaces $X^{\beta,k}_{\alpha}$ as follows.
\begin{definition}\label{X k space}
For a nonnegative integer $k$ and $\beta\in(1/2,1],$ we introduce
the space $X^{\beta,k}_{\alpha}$ which is equipped with the
following norm:
$$\|u\|_{X^{\beta,k}_{\alpha}}=\|u\|_{N^{\beta,k}_{\alpha,\infty}}+\|u\|_{N^{\beta,k}_{\alpha,C}}$$
where
\begin{eqnarray*}
&&\|u\|_{N^{\beta,k}_{\alpha,\infty}}=\sup_{\alpha_{1}+\cdots+\alpha_{n}=k}\sup_{t}t^{\frac{2\beta-1+k}{2\beta}}\|\partial_{x_{1}}^{\alpha_{1}}
\cdots\partial_{x_{n}}^{\alpha_{n}}u(\cdot,t)\|_{L^{\infty}},\\
&&\|u\|_{N^{\beta,k}_{\alpha,C}}=\sup_{\alpha_{1}+\cdots+\alpha_{n}=k}\sup_{x_{0},r}\left(r^{2\alpha-n+2\beta-2}\int_{0}^{r^{2\beta}}\int_{|y-x_{0}|<r}
|t^{\frac{k}{2\beta}}\partial_{x_{1}}^{\alpha_{1}}
\cdots\partial_{x_{n}}^{\alpha_{n}}u(t,y)|^{2}\frac{dy
dt}{t^{\alpha/\beta}}\right)^{1/2}.
\end{eqnarray*}
\end{definition}
In the following, we will denote
$\nabla^{k}u=\partial_{x_{1}}^{\alpha_{1}}
\cdots\partial_{x_{n}}^{\alpha_{n}}u$ with
$(\alpha_{1},\alpha_{2},\ldots,\alpha_{n})\in \mathbb{N}^{n}_{0}$
and $k=\alpha_{1}+\cdots+\alpha_{n}.$

\subsection{Several Technical Lemmas}
Before stating the main result of this section, we prove several
preliminary lemmas associated with the fractional heat semigroup
$e^{-t(-\triangle)^{\beta}}$. Recall that
$e^{-t(-\triangle)^{\beta}}f(x)=K^{\beta}_{t}\ast f(x)$ where
$K^{\beta}_{t}$
 is defined by $\widehat{(K^{\beta}_{t})}(\xi)=e^{-t|\xi|^{2\beta}}$ and  $P$ is the
Helmboltz-Weyl projection.
\begin{lemma}\label{g Ossen}
Let $\beta\in(1/2,1).$ There exists a constant $C>0$ depending only
on $n$ such that
$$|\partial^{k}_{x}P\nabla K^{\beta}_{t}(x)|\leq C^{k}k^{k/2\beta}t^{-k/2\beta}(k^{-\frac{1}{2\beta}}t^{\frac{1}{2\beta}}+|x|)^{-n-1}$$
for all $t>0,x\in \mathbb{R}^{n}$ and $k\in\mathbb{N}.$
\end{lemma}
\begin{proof}
By a dilation argument, we have
$$\partial^{k}_{x}P\nabla K^{\beta}_{t}(x)=t^{\frac{-k-1}{2\beta}}t^{-\frac{n}{2\beta}}\partial^{k}_{x}P\nabla K^{\beta}_{1}(x/t^{\frac{1}{2\beta}}).$$
If we could prove $|\partial_{x}^{k}P\nabla K_{1}^{\beta}(x)|\leq
C^{k}k^{k/2\beta}(k^{-\frac{1}{2\beta}}+|x|)^{-n-1},$ then we have
\begin{eqnarray*}
|\partial^{k}_{x}P\nabla
K^{\beta}_{t}(x)|&\leq&t^{-\frac{k+1}{2\beta}}t^{-\frac{n}{2\beta}}C^{k}k^{k/2\beta}\left(k^{-\frac{1}{2\beta}}+\left|\frac{x}{t^{1/2\beta}}\right|\right)^{-n-1}\\
&\leq&c^{k}k^{k/2\beta}t^{-\frac{k}{2\beta}}(t^{\frac{1}{2\beta}}k^{-\frac{1}{\beta}}+|x|)^{-n-1}.
\end{eqnarray*}
Hence we obtain the desired.

By the semigroup property, it is easy to see that
$\partial_{x}^{k}P\nabla K_{1}^{\beta}=P\nabla
K_{1/2}^{\beta}\ast\partial_{x}^{k}K^{\beta}_{1/2}.$ So we need to
prove the following two estimates:
\begin{equation}\label{section 5 ineq 1}
\quad |P\nabla K^{\beta}_{1/2}(x)|\leq C(1+|x|)^{-n-1}
\end{equation}
\begin{equation}\label{section 5 ineq 2}
\quad |\partial_{x}^{k}K^{\beta}_{1/2}(x)|\leq
C^{k-1}k^{\frac{k-1}{2\beta}}(k^{-\frac{1}{2\beta}}+|x|)^{-n-1}.
\end{equation}
For (\ref{section 5 ineq 1}). Taking $\alpha=1$ in the Lemma
\ref{le4}, we have $$(1+|x|)^{n+1}|P\nabla K^{\beta}_{1/2}(x)|\leq
C,$$ that is, (\ref{section 5 ineq 1}) is obvious.
For (\ref{section 5 ineq 2}), we claim that
$|\partial_{i}K^{\beta}_{1/2}(x)|\leq C(1+|x|)^{-n-1}.$ In fact when
$|x|<1$,
$$(1+|x|)^{n+1}|\partial_{i}K^{\beta}_{1/2}(x)|\leq 2^{n+1}\int_{\mathbb{R}^{n}}|i\xi_{i}|e^{-|\xi|^{2\beta}/2}d\xi\lesssim C.$$
When $|x|>1$, we define the operator
$$L(x,D)=\frac{x\cdot\nabla_{\xi}}{i|x|^{2}}, \text{ that is }L(x,D)e^{ix\cdot\xi}=e^{ix\cdot\xi}$$
and choose a $C^{\infty}_{c}(\mathbb{R}^{n})-$function $\rho(x)$
satisfying:
$$\rho(\xi)= \left\{\begin{array}{ll} 1, &|\xi|\leq 1,\\
 0, &  |\xi|>2,
\end{array}\right.$$
we have
\begin{eqnarray*}
|\partial_{i}K^{\beta}_{1/2}(x)|&\leq&\left|\int_{\mathbb{R}^{n}}\rho\left(\frac{\xi}{\delta}\right)i\xi_{i}e^{-|\xi|^{2\beta}/2}e^{ix\cdot\xi}d\xi\right|\\
&+&\left|\int_{\mathbb{R}^{n}}[1-\rho\left(\frac{\xi}{\delta}\right)]i\xi_{i}e^{-|\xi|^{2\beta}/2}e^{ix\cdot\xi}d\xi\right|\\
&:=&I_{3}+I_{4}.
\end{eqnarray*}
For $I_{3},$ we have
$$I_{3}\lesssim \int_{\mathbb{R}^{n}}\rho\left(\frac{\xi}{\delta}\right)|\xi|e^{-|\xi|^{2\beta}/2}d\xi
\lesssim\int_{|\xi|<2\delta}\delta d\xi\lesssim \delta^{n+1}.$$ For
$I_{4},$ using the integration by parts and
$L^{\ast}=-\frac{x\cdot\nabla_{\xi}}{i|x|^{2}},$ we have
\begin{eqnarray*}
I_{4}&=&\left|\int_{\mathbb{R}^{n}}(L^{\ast})^{N}\left([1-\rho\left(\frac{\xi}{\delta}\right)]i\xi_{i}e^{-|\xi|^{2\beta}/2}\right)
e^{ix\cdot\xi}d\xi\right|\\
&\lesssim&C_{N}|x|^{-N}\int_{\mathbb{R}^{n}}\left|\sum_{k=0}^{N}C^{k}_{N}\nabla_{\xi}^{k}\left[1-\rho\left(\frac{\xi}{\delta}\right)\right]
\nabla_{\xi}^{N-k}(i\xi e^{-|xi|^{2\beta}/2})\right|d\xi\\
&\lesssim&C_{N}|x|^{-N}\int_{|\xi|>\delta}
\sum_{k=1}^{N}|\xi|^{2\beta k-N+1}e^{-|\xi|^{2\beta}/2}d\xi\\
&+&C_{N}|x|^{-N}\int_{\delta\leq|\xi|\leq2\delta}\sum_{k=1}^{N}C_{N}^{k}\delta^{-k}
\sum_{l=0}^{N-k}C^{l}_{N-k}|\xi|^{2\beta l-N+k+1}
e^{-|\xi|^{2\beta}/2}d\xi\\
&\lesssim&C_{N}|x|^{-N}\int_{|\xi|>\delta}|\xi|^{1-N}d\xi+C_{N}|x|^{-N}\int_{\delta<|\xi|<2\delta}\delta^{-k}|\xi|^{1-N+k}d\xi\\
&\lesssim&C_{N}|x|^{-N}\delta^{n+1-N}.
\end{eqnarray*}
So we get, taking $\delta=|x|^{-1}$,
$$|\partial_{i}K^{\beta}_{1/2}(x)|\lesssim \delta^{n+1}+C_{N}|x|^{-N}\delta^{n+1-N}\lesssim C_{N}|x|^{-(n+1)}\lesssim C_{N}(1+|x|)^{-(n+1)}.$$

 Then we have
$$|\partial_{i}K_{\frac{1}{2k}}^{\beta}(x)|\lesssim k^{\frac{1}{2\beta}}k^{\frac{n}{2\beta}}
|\partial_{i}K_{1/2}^{\beta}(k^{1/2\beta}x)|\lesssim
C(k^{-\frac{1}{2\beta}}+|x|)^{-n-1}.$$ Because the following
integral inequality(see \cite{H. Miura O. Sawada}):
$$\int_{\mathbb{R}^{n}}(a+|x-y|)^{-n-1}(b+|y|)^{-n-1}dy\leq ca^{-1}(a+|x|)^{-n-1} \text{ for } 0<a\leq b,$$
we have
\begin{eqnarray*}
|\partial_{i,j}^{2}K_{\frac{1}{k}}^{\beta}(x)|&=&\left|\int_{\mathbb{R}^{n}}\partial_{i}K^{\beta}_{\frac{1}{2k}}(x-y)
\partial_{j}K^{\beta}_{\frac{1}{2k}}(y)dy\right|\\
&\lesssim&\int_{\mathbb{R}^{n}}(k^{-\frac{1}{2\beta}}+|x-y|)^{-n-1}(k^{-\frac{1}{2\beta}}+|y|)^{-n-1}dy\\
&\lesssim&k^{\frac{1}{2\beta}}(k^{-\frac{1}{2\beta}}+|x|)^{-n-1}.
\end{eqnarray*}
Operating the process above $k-1$ times, we get
$|\partial_{x}^{k}K^{\beta}_{1/2}(x)|\lesssim
k^{\frac{k-1}{2\beta}}(k^{-\frac{1}{2\beta}}+|x|)^{-n-1}$ and
\begin{eqnarray*}
|\partial_{x}^{k}P\nabla
K^{\beta}_{1}(x)|&=&\left|\int_{\mathbb{R}^{n}}P\nabla
K_{1/2}(x-y)\partial^{k}_{x}K^{\beta}_{1/2}(y)dy\right|\\
&\lesssim&\int_{\mathbb{R}^{n}}
c'(1+|x-y|)^{-n-1}c^{k-1}k^{\frac{k-1}{2\beta}}(k^{-\frac{1}{2\beta}}+|y|)^{-n-1}dy\\
&\lesssim&C^{k}k^{\frac{1}{2\beta}}(k^{-\frac{1}{2\beta}}+|x|)^{-n-1}k^{\frac{k-1}{2\beta}}\\
&\lesssim&C^{k}k^{\frac{k}{2\beta}}(k^{-\frac{1}{2\beta}}+|x|)^{-n-1}.
\end{eqnarray*}
This completes the proof of this lemma.
\end{proof}
The following lemma can be regarded as a generalization of
Proposition 3.2 of \cite{P. Germain N. Pavlovic G.Staffilani}.
\begin{lemma}\label{lemmma pq}
If $r$ is a natural number, $\alpha\in(0,1)$ and
$\max\{\alpha,\frac{1}{2}\}<\beta\leq 1$, the operator
$$P^{\beta}_{r}f(t,x)=\int_{0}^{t}e^{-(t-s)(-\triangle)^{\beta}}(t^{\frac{1}{2\beta}}-s^{\frac{1}{2\beta}})^{r}\nabla^{r+2\beta}f(s,x)ds$$
is bounded on $L^{2}([0,T],L^{2}(R^{n},
dx),\frac{dt}{t^{\alpha/\beta}})$ for any $T\in [0,\infty]$ with
constants $p(r)$ and $q(r)$.
\end{lemma}
\begin{proof}
By Planceherl's theorem and H\"{o}lder's inequality, we have
\begin{eqnarray*}
&&\int_{0}^{\infty}\|P_{r}^{\beta}f(t,\cdot)\|^{2}_{L^{2}}\frac{dt}{t^{\alpha/\beta}}\\
&&\lesssim\int_{0}^{\infty}\int_{\mathbb{R}^{n}}\left(\int_{0}^{t}e^{-(t-s)|\xi|^{2\beta}}(t^{\frac{1}{2\beta}}-s^{\frac{1}{2\beta}})^{r}|\xi|^{r+2\beta}\widehat{f}(s,\xi)ds\right)^{2}
\frac{dt}{t^{\alpha/\beta}}d\xi\\
&&\lesssim\int_{0}^{\infty}\int_{\mathbb{R}^{n}}\left(\int_{0}^{t}e^{-(t-s)|\xi|^{2\beta}}(t^{\frac{1}{2\beta}}-s^{\frac{1}{2\beta}})^{r}|\xi|^{r+2\beta}ds\right)\\
&&\times
\left(\int_{0}^{t}e^{-(t-s)|\xi|^{2\beta}}(t^{\frac{1}{2\beta}}-s^{\frac{1}{2\beta}})^{r}|\xi|^{r+2\beta}|\widehat{f}(s,\xi)|^{2}ds\right)\frac{dt}{t^{\alpha/\beta}}d\xi.
\end{eqnarray*}
Because $t^{1/2\beta}-s^{1/2\beta}\leq(t-s)^{1/2\beta}$ for
$2\beta>1$ and $0<s<t$, it is easy to see that
\begin{eqnarray*}
\int_{0}^{t}e^{-(t-s)|\xi|^{2\beta}}(t^{\frac{1}{2\beta}}-s^{\frac{1}{2\beta}})^{r}|\xi|^{r+2\beta}ds
&\leq&
\int_{0}^{t}e^{-(t-s)|\xi|^{2\beta}}(t-s)^{\frac{r}{2\beta}}|\xi|^{r+2\beta}ds\\
&\lesssim&\int_{0}^{\infty}e^{-v}v^{r/2\beta}dv\lesssim 1.
\end{eqnarray*}
Then we have, by $t^{1/2\beta}-s^{1/2\beta}\leq(t-s)^{1/2\beta}$ for
$2\beta>1$ and $0<s<t$ again,
\begin{eqnarray*}
&&\int_{0}^{\infty}\|P_{r}^{\beta}f(t,\cdot)\|^{2}_{L^{2}}\frac{dt}{t^{\alpha/\beta}}\\
&\lesssim&\int_{\mathbb{R}^{n}}\int_{0}^{\infty}|\widehat{f}(s,\xi)|^{2}
\left(\int_{s}^{\infty}e^{-(t-s)|\xi|^{2\beta}}(t-s)^{\frac{r}{2\beta}}|\xi|^{r+2\beta}dt\right)\frac{dsd\xi}{s^{\alpha/\beta}}\\
&\lesssim&\int_{\mathbb{R}^{n}}\int_{0}^{\infty}|\widehat{f}(s,\xi)|^{2}
\left(\int_{0}^{\infty}e^{-u|\xi|^{2\beta}}u^{\frac{r}{2\beta}}|\xi|^{r+2\beta}du\right)\frac{dsd\xi}{s^{\alpha/\beta}}\\
&\lesssim&\int_{0}^{\infty}\int_{\mathbb{R}^{n}}|\widehat{f}(s,\xi)|^{2}\frac{ds}{s^{\alpha/\beta}}d\xi.
\end{eqnarray*}
This completes the proof of this lemma.
\end{proof}
\begin{lemma}\label{linear part}
For any $k\geq 0$, $\alpha>0$ and $\max\{\alpha,1/2\}<\beta<1$ with
$\alpha+\beta-1\geq0$, there exists a constant $C(k)$ such that
$$\|e^{-t(-\triangle)^{\beta}}u\|_{X^{\beta,k}_{\alpha}}\leq C(k)\|u\|_{Q^{\beta,-1}_{\alpha;\infty}}.$$
\end{lemma}
\begin{proof}
Because
$\|u\|_{X_{\alpha}^{\beta,k}}=\|u\|_{N^{\beta,k}_{\alpha,\infty}}+\|u\|_{N^{\beta,k}_{\alpha,C}}$,
we split the proof into two parts.

For the $L^{\infty}$ part of the norm. Because
$Q^{\beta,-1}_{\alpha;\infty}\hookrightarrow
\dot{B}^{1-2\beta}_{\infty,\infty}$ and $\nabla^{k}:
\dot{B}^{1-2\beta}_{\infty,\infty}\longrightarrow\dot{B}^{1-2\beta-k}_{\infty,\infty},$
we have
\begin{eqnarray*}
&&\|\nabla^{k}e^{-t(-\triangle)^{\beta}}u\|_{\infty}\leq
t^{\frac{1-2\beta-k}{2\beta}}\|\nabla^{k}u\|_{\dot{B}^{1-2\beta-k}_{\infty,\infty}}
\leq
t^{\frac{1-2\beta-k}{2\beta}}\|u\|_{\dot{B}^{1-2\beta}_{\infty,\infty}}
\leq
t^{\frac{1-2\beta-k}{2\beta}}\|u\|_{Q^{\beta,-1}_{\alpha;\infty}}.
\end{eqnarray*}
Then we can get
$t^{\frac{2\beta-1+k}{2\beta}}\|\nabla^{k}e^{-t(-\triangle)^{\beta}}u\|_{\infty}\leq\|u\|_{Q^{\beta,-1}_{\alpha;\infty}}.$

For the Carleson part. Because $u\in
Q^{\beta,-1}_{\alpha;\infty}=\nabla\cdot(Q^{\beta}_{\alpha})^{n},$
there exists a sequence $\{f_{j}\}\subset Q^{\beta}_{\alpha}$ such
that $u=\sum_{j}\partial_{j}f_{j}.$ We only need to prove
\begin{equation}\label{reg char}
\sup_{x\in
\mathbb{R}^{n},r>0}r^{2\alpha-n+2\beta-2}\int_{0}^{r^{2\beta}}\int_{|y-x|<r}
\left|t^{\frac{k}{2\beta}}\nabla^{k}e^{-t(-\triangle)^{\beta}}\partial_{j}f_{j}(x)\right|^{2}
\frac{dydt}{t^{\alpha/\beta}}\lesssim
C(k)\|\partial_{j}f_{j}\|^{2}_{Q^{\beta,-1}_{\alpha;\infty}}
=C(k)\|f_{j}\|^{2}_{Q^{\beta}_{\alpha}}.
\end{equation}
Taking $\psi(x)=\nabla^{k}\partial_{j}e^{-(-\triangle)^{\beta}}(x)
=\int_{\mathbb{R}^{n}}(i\xi)^{k}i\xi_{j}e^{-|\xi|^{2\beta}}e^{2\pi
ix\cdot\xi}d\xi,$ we can justify the function $\psi(x)$
 with
 $\widehat{\psi_{t}}(\xi)=(it\xi)^{k}(it\xi_{j})e^{-t^{2\beta}|\xi|^{2\beta}}$
 satisfying the conditions in (\ref{prp psi}):
$$|\psi(x)|\lesssim (1+|x|)^{-n-1},\quad  \psi(x)\in L^{1} \quad \hbox{and}\  \int_{\mathbb{R}^{n}}\psi(x)dx=0.$$
 By the equivalent characterization of
$Q^{\beta}_{\alpha}$ (see(\ref{eq cha q a b})), we have
$$\sup_{x,r}r^{2\alpha-n+2\beta-2}\int_{0}^{r}\int_{|y-x|<r}\left|s^{k}\nabla^{k}e^{-s^{2\beta}(-\triangle)^{\beta}}(s\partial_{j})f_{j}(x)\right|^{2}
\frac{dyds}{s^{1+2\alpha-2(\beta-1)}}\lesssim
C(k)\|f_{j}\|^{2}_{Q^{\beta}_{\alpha}}.$$
By a change of variable:
$t=s^{2\beta},$ we get the desired result (\ref{reg char}).
\end{proof}

\subsection{Regularity} Now we state the main theorem of this
section.
\begin{theorem}\label{threg}
 Let $\alpha>0$ and $\max\{\alpha,1/2\}<\beta<1$ with $\alpha+\beta-1\geq0$. There exists an $\varepsilon=\varepsilon(n)$ such that if
$\|u_{0}\|_{Q^{\beta,-1}_{\alpha;\infty}}<\varepsilon,$ the
solution $u$ to equations (\ref{eq1e}) verifies:
$$t^{\frac{k}{2\beta}}\nabla^{k}u\in X^{\beta,0}_{\alpha}\text{ for any } k\geq 0.$$
\end{theorem}
\begin{proof}
 We can see that the solution to the equations
(\ref{eq1e}) can be represented as
$$u(t,x)=e^{-t(-\triangle)^{\beta}}u(0,x)-B(u,u)(t,x),$$ where
$$B(u,v)(t,x)=\int_{0}^{t}e^{-(t-s)(-\triangle)^{\beta}}P\nabla\cdot(u(s,x)\otimes v(s,x))ds.$$
Here $u\otimes v$ denotes the tensor product of $u$ and $v$. For the
linear term
$e^{-t(-\triangle)^{\beta}}u(0,x):=e^{-t(-\triangle)^{\beta}}u_{0},$
by Proposition \ref{linear part}, we have
$\|e^{-t(-\triangle)^{\beta}}u_{0}\|_{X^{\beta,k}_{\alpha}}\leq
C(k)\|u_{0}\|_{Q^{\beta,-1}_{\alpha}}.$ Now we estimate the
nonlinear term. We write
$\widetilde{X}^{\beta,k}_{\alpha}=\cap_{l=0}^{k}X^{\beta,k}_{\alpha}$
equipped with the norm
$\sum_{l=0}^{k}\|\cdot\|_{X^{\beta,l}_{\alpha}}.$ We shall prove
that the bilinear operator maps
$$B(u,v):\widetilde{X}^{\beta,k}_{\alpha}\times \widetilde{X}^{\beta,k}_{\alpha}\longrightarrow\widetilde{X}^{\beta,k}_{\alpha}.$$
{{\it{Part 1}} $N^{\beta,k}_{\alpha,\infty}$ norm.} Here we shall
prove that
\begin{eqnarray*}
\|B(u,v)\|_{N^{\beta,k}_{\alpha,\infty}}&\lesssim&C_{0}(k)\|u\|_{X^{\beta,0}_{\alpha}}\|v\|_{X^{\beta,0}_{\alpha}}
+C(k)\sum_{l=1}^{k-1}\|u\|_{N^{\beta,l}_{\alpha,\infty}}\|v\|_{N^{\beta,k-l}_{\alpha,\infty}}\\
&+&C_{1}\|u\|_{X^{\beta,0}_{\alpha}}\|v\|_{X^{\beta,k}_{\alpha}}+C_{1}\|u\|_{X^{\beta,0}_{\alpha}}\|v\|_{X^{\beta,k}_{\alpha}}.
\end{eqnarray*}
If $0<s<t(1-\frac{1}{m}),$ $\frac{t}{m}<t-s<t$. By
Lemma \ref{g Ossen}, we have
\begin{eqnarray*}
&&I=\int_{0}^{t(1-\frac{1}{m})}\left|\nabla^{k}e^{-(t-s)(-\triangle)^{\beta}}P\nabla\cdot(u(s,x)\otimes
v(s,x))\right|ds\\
&&=C^{k}k^{\frac{k}{2\beta}}\int_{0}^{t(1-\frac{1}{m})}\int_{\mathbb{R}^{n}}
\frac{|u(s,y)||v(s,y)|}{(t-s)^{\frac{k}{2\beta}}(t-s)^{\frac{n+1}{2\beta}}
[k^{-\frac{1}{2\beta}}+\frac{|x-y|}{(t-s)^{1/2\beta}}]^{n+1}}dyds\\
&&\leq
C^{k}k^{\frac{k}{2\beta}}\left(\frac{m}{t}\right)^{(n+k+1)/2\beta}\sum_{q\in\mathbb{Z}^{n}}\int_{0}^{t(1-\frac{1}{m})}
\int_{\frac{x-y}{t^{1/2\beta}}\in
q+[0,1]^{n}}\frac{|u(s,y)||v(s,y)|}{[k^{-\frac{1}{2\beta}}+|q|]^{n+1}}dyds.
\end{eqnarray*}
Because $\sum_{q\in
\mathbb{Z}^{n}}\frac{1}{[k^{-\frac{1}{2\beta}}+|q|]^{n+1}}\approx
k^{1/2\beta}$, we have

\begin{eqnarray*}
&&I\leq
C^{k}k^{k/2\beta}\left(\frac{m}{t}\right)^{(n+k+1)/2\beta}k^{\frac{1}{2\beta}}\int_{0}^{t(1-\frac{1}{m})}\int_{x-y\in
t^{\frac{1}{2\beta}}(q+[0,1]^{n})}|u(s,y)||v(s,y)|dyds\\
&&\leq
C^{k}k^{k/2\beta}\left(\frac{m}{t}\right)^{(n+k+1)/2\beta}k^{\frac{1}{2\beta}}t^{\alpha/\beta}\int_{0}^{t}
\int_{|x-y|<t^{1/2\beta}}|u(s,y)||v(s,y)|\frac{dyds}{s^{\alpha/\beta}}\\
&&\leq
C^{k}k^{\frac{k+1}{2\beta}}m^{\frac{n+k+1}{2\beta}}t^{\frac{-k-2\beta+1}{2\beta}}\|u\|_{X^{\beta,0}_{\alpha}}\|v\|_{X^{\beta,0}_{\alpha}}\\
&&:=C_{0}(k)t^{\frac{-k-2\beta+1}{2\beta}}\|u\|_{X^{\beta,0}_{\alpha}}\|v\|_{X^{\beta,0}_{\alpha}}
\end{eqnarray*}
where
$C_{0}(k)=C^{k}k^{\frac{k+1}{2\beta}}m^{\frac{n+k+1}{2\beta}}.$

If $t(1-\frac{1}{m})\leq s\leq t,$ by Young's inequality, we have
\begin{eqnarray*}
|\nabla^{k}e^{-(t-s)(-\triangle)^{\beta}}P\nabla\cdot(u(s,x)\otimes
v(s,x))|&=&|P\nabla
e^{-(t-s)(-\triangle)^{\beta}}\nabla^{k}(u(s,x)\otimes v(s,x))| \\
&\leq&\|P\nabla
e^{-(t-s)(-\triangle)^{\beta}}(x)\|_{L^{1}}\|\nabla^{k}(u(s,x)\otimes
v(s,x))\|_{L^{\infty}}.
\end{eqnarray*}
By the estimate for the generalized Oseen kernel:
$$|P\nabla^{l+1}e^{-(-\triangle)^{\beta}}(x)|\lesssim\frac{1}{(1+|x|)^{n+1+l}} \text{ and }
P\nabla^{l+1}e^{-u(-\triangle)^{\beta}}=u^{-\frac{n+1+l}{2\beta}}P\nabla^{l+1}e^{-(-\triangle)^{\beta}}(\frac{x}{u^{1/2\beta}}),$$
we have
$$|P\nabla^{l+1}e^{-u(-\triangle)^{\beta}}|\lesssim u^{-\frac{n+l+1}{2\beta}}\frac{1}{\left(1+\frac{|x|}{u^{1/2\beta}}\right)^{l+n+1}}
\lesssim\frac{1}{\left(u^{1/2\beta}+|x|\right)^{l+n+1}}.$$ Then we
take $l=0$ and have
\begin{eqnarray*}
\|P\nabla
e^{-u(-\triangle)^{\beta}}\|_{L^{1}}&\lesssim&\int_{0}^{\infty}\frac{|x|^{n-1}}{(u^{\frac{1}{2\beta}}+|x|)^{n+1}}d|x|
\lesssim\frac{1}{u^{1/2\beta}}.
\end{eqnarray*}
Hence we can get
\begin{eqnarray*}
|\nabla^{k}e^{-(t-s)(-\triangle)^{\beta}}P\nabla(u(s,x)\otimes
v(s,x))|&\lesssim&\frac{1}{(t-s)^{1/2\beta}}\sum_{l=0}^{k}\left(\begin{array}{ccc}k\\l\end{array}\right)\|\nabla^{l}u(s,\cdot)\|_{L^{\infty}}\|\nabla^{k-l}v(s,\cdot)\|_{L^{\infty}}\\
&\lesssim&\frac{1}{(t-s)^{1/2\beta}}\sum_{l=0}^{k}\left(\begin{array}{ccc}k\\l\end{array}\right)
\frac{\|u\|_{N^{\beta,l}_{\alpha,\infty}}\|v\|_{N^{\beta,k-l}_{\alpha,\infty}}}{s^{(2\beta-1+l)/2\beta}s^{(2\beta-1+k-l)/2\beta}}.
\end{eqnarray*}
So we have
\begin{eqnarray*}
&&\left|\int_{t(1-\frac{1}{m})}^{t}\nabla
e^{-(t-s)(-\triangle)^{\beta}}P\nabla^{k+1}(u(s,x)\otimes
v(s,x))ds\right|\\
&&\lesssim\sum_{l=0}^{k}\left(\begin{array}{ccc}k\\l\end{array}\right)\|u\|_{N^{\beta,l}_{\alpha,\infty}}\|v\|_{N^{\beta,k-l}_{\alpha,\infty}}
\int_{t(1-\frac{1}{m})}^{t}\frac{1}{(t-s)^{1/2\beta}}\frac{1}{s^{(4\beta-2+k)/2\beta}}ds.
\end{eqnarray*}
For the integral in the last inequality, we make the change of
variable: $s=zt$. Because $t\left(1-\frac{1}{m}\right)<s<t$ implies
$\left(1-\frac{1}{m}\right)<z<1,$ we have
\begin{eqnarray*}
II&=&\int_{t\left(1-\frac{1}{m}\right)}^{t}\frac{1}{(t-s)^{1/2\beta}}\frac{1}{s^{\frac{4\beta-2+k}{2\beta}}}ds\\
&\lesssim&
t^{\frac{1-2\beta-k}{2\beta}}(1-\frac{1}{m})^{-\frac{4\beta-2+k}{2\beta}}\int_{1-\frac{1}{m}}^{1}(1-z)^{-\frac{1}{2\beta}}dz\\
&=&t^{\frac{1-2\beta-k}{2\beta}}(1-\frac{1}{m})^{-\frac{4\beta-2+k}{2\beta}}(\frac{1}{m})^{1-\frac{1}{2\beta}}.
\end{eqnarray*}
Denote
$g(m)=\left(1-\frac{1}{m}\right)^{-\frac{4\beta-2+k}{2\beta}}\left(\frac{1}{m}\right)^{1-\frac{1}{2\beta}}$
and take $m=m(k)=k^{\frac{k-3}{n+k+1}}.$ We can prove that
$g(m)\longrightarrow 0$ as $k\longrightarrow\infty.$ Then we have
$II\lesssim Ct^{\frac{1-2\beta-k}{2\beta}}$ for $k\geq 1.$

Therefore we have
\begin{eqnarray*}
&&\left|\int_{t\left(1-\frac{1}{m}\right)}^{t}\nabla
e^{-(t-s)(-\triangle)^{\beta}}P\nabla^{k+1}(u(s,x)\times
v(s,x))ds\right|\\
&&\lesssim
Ct^{\frac{1-2\beta-k}{2\beta}}\left[\sum_{l=1}^{k}\left(\begin{array}{ccc}k\\l\end{array}\right)
\|u\|_{N^{\beta,l}_{\alpha,\infty}}\|v\|_{N^{\beta,k-l}_{\alpha,\infty}}
+\|u\|_{N^{\beta,0}_{\alpha,\infty}}\|v\|_{N^{\beta,k}_{\alpha,\infty}}+
\|u\|_{N^{\beta,k}_{\alpha,\infty}}\|v\|_{N^{\beta,0}_{\alpha,\infty}}\right].
\end{eqnarray*}
{{\it {Part 2}} $N^{\beta,k}_{\alpha,C}$ norm.} We split $B(u,v)$ as
follows: $B(u,v)=B_{1}(u,v)+B_{2}(u,v)$ with
\begin{eqnarray*}
&&B_{1}(u,v)(t,x)=\int_{0}^{t}e^{-(t-s)(-\triangle)^{\beta}}P\nabla\left[1-\phi\left(\frac{x-x_{0}}{R^{1/2\beta}}\right)\right]u(s,x)\otimes
v(s,x)ds\\
&&B_{2}(u,v)(t,x)=\int_{0}^{t}e^{-(t-s)(-\triangle)^{\beta}}P\nabla\phi\left(\frac{x-x_{0}}{R^{1/2\beta}}\right)u(s,x)\otimes
v(s,x)ds
\end{eqnarray*}
where
$\phi_{R^{\frac{1}{2\beta}},x_{0}}=\phi((x-x_{0})/R^{\frac{1}{2\beta}})$
for a smooth function $\phi$ supported in $B(0,15)$ and equals to 1
on $B(0,10).$

For the estimate for $B_{1}$. Because
$|P\nabla^{k+1}e^{-(t-s)(-\triangle)^{\beta}}(x)|\lesssim\frac{K(k)}{\left[(t-s)^{1/2\beta}+|x-y|\right]^{n+k+1}}$
and $0<t<R$, we have
\begin{eqnarray*}
&&|t^{\frac{k}{2\beta}}\nabla^{k}B_{1}(u,v)(t,x)|\\
&\lesssim&
t^{\frac{k}{2\beta}}\left|\nabla^{k}\int_{0}^{t}\int_{|y-x_{0}|\geq
10R^{1/2\beta}}e^{-(t-s)(-\triangle)^{\beta}}P\nabla(x-y)u(s,y)v(s,y)dyds\right|\\
&\lesssim&K(k)t^{\frac{k}{2\beta}}\int_{0}^{t}\int_{|y-x_{0}|\geq10R^{1/2\beta}}\frac{|u(s,y)||v(s,y)|dyds}{\left[(t-s)^{1/2\beta}+|x-y|\right]^{n+k+1}}\\
&\lesssim&K(k)R^{\frac{k}{2\beta}}\int_{0}^{R}\int_{|y-x_{0}|\geq10R^{1/2\beta}}\frac{|u(s,y)||v(s,y)|}{R^{\frac{n+k+1}{2\beta}}
\left[(\frac{t-s}{R})^{\frac{1}{2\beta}}+\frac{|x-y|}{R^{1/2\beta}}\right]^{n+k+1}}dyds\\
&\lesssim&K(k)R^{\frac{k}{2\beta}-\frac{n+k+1}{2\beta}}\sum_{q\in
\mathbb{Z}^{n}}\frac{1}{|q|^{n+k+1}}R^{\alpha/\beta}\int_{0}^{R}\int_{|y-x_{0}|<R^{1/2\beta}}|u(s,y)||v(s,y)|\frac{dyds}{s^{\alpha/\beta}}\\
&\lesssim&K(k)D(k)R^{\frac{1-2\beta}{2\beta}}\|u\|_{X^{\beta,0}_{\alpha}}\|v\|_{X^{\beta,0}_{\alpha}}.
\end{eqnarray*}
Then we have, taking $R=r^{2\beta}$,
\begin{eqnarray*}
&&r^{2\alpha-n+2\beta-2}\int_{0}^{r^{2\beta}}\int_{|y-x|<r}\left|t^{\frac{k}{2\beta}}\nabla^{k}B_{1}(u,v)(t,y)\right|^{2}\frac{dydt}{t^{\alpha/\beta}}\\
&&\lesssim(K(k)D(k))^{2}
r^{2\alpha-n+2\beta-2}\int_{0}^{r^{2\beta}}\int_{|y-x|<r}r^{2-4\beta}\|u\|_{X^{\beta,0}_{\alpha}}^{2}\|v\|^{2}_{X^{\beta,0}_{\alpha}}
\frac{dydt}{t^{\alpha/\beta}}\\
&&\lesssim(K(k)D(k))^{2}\|u\|_{X^{\beta,0}_{\alpha}}^{2}\|v\|^{2}_{X^{\beta,0}_{\alpha}}.
\end{eqnarray*}
For the estimate for $B_{2}.$ We further split $B_{2}$ as
$B_{2}=B^{1}_{2}+B^{2}_{2}$ with
\begin{eqnarray*}
&&B^{1}_{2}=\frac{1}{\sqrt{-\triangle}}P\nabla\int_{0}^{t}e^{-(t-s)(-\triangle)^{\beta}}\frac{\triangle}{\sqrt{-\triangle}}\left(I-e^{-s(-\triangle)}\right)
\left(\phi_{R^{\frac{1}{2\beta}},x_{0}}u(s,x)\otimes v(s,x)\right)ds,\\
&&B^{2}_{2}=\frac{1}{\sqrt{-\triangle}}P\nabla
e^{-t(-\triangle)^{\beta}}\int_{0}^{t}\phi_{R^{\frac{1}{2\beta}},x_{0}}u(s,x)\otimes
v(s,x)ds.
\end{eqnarray*}
At first we estimate the term
$t^{\frac{k}{2\beta}}\nabla^{k}B^{1}_{2}.$ Without loss of the
generalization, we assume $k$ is odd. The proof of the case that $k$
is even is similar. If $k$ is odd, we have $k=2K+1$ for $K\in
\mathbb{Z}_{+}.$ Because $\frac{1}{2}<\beta<1$, we have
\begin{eqnarray*}
t^{\frac{k}{2\beta}}&=&\left(t^{\frac{1}{2\beta}}-s^{\frac{1}{2\beta}}+s^{\frac{1}{2\beta}}\right)^{2K}
\left(t^{\frac{1}{2\beta}}-s^{\frac{1}{2\beta}}+s^{\frac{1}{2\beta}}\right)\\
&=&\sum_{l=0}^{2K}\left(\begin{array}{ccc}2K\\l\end{array}\right)\left(t^{\frac{1}{2\beta}}-s^{\frac{1}{2\beta}}\right)^{2K-l}s^{\frac{l}{2\beta}}
\left(t^{\frac{1}{2\beta}}-s^{\frac{1}{2\beta}}+s^{\frac{1}{2\beta}}\right)\\
&=&\sum_{l=0}^{2K}\left(\begin{array}{ccc}2K\\l\end{array}\right)
\left(t^{\frac{1}{2\beta}}-s^{\frac{1}{2\beta}}\right)^{2K-l+1}s^{\frac{l}{2\beta}}
+\sum_{l=0}^{2K}\left(\begin{array}{ccc}2K\\l\end{array}\right)(t^{\frac{1}{2\beta}}-s^{\frac{1}{2\beta}})^{2K-l}s^{\frac{l+1}{2\beta}}.
\end{eqnarray*}
Then we have, setting
$M(s,x)=\phi_{R^{\frac{1}{2\beta}},x_{0}}(x)u(s,x)\otimes v(s,x),$
\begin{eqnarray*}
t^{\frac{t}{2\beta}}\nabla^{k}B^{1}_{2}
&=&\sum_{l=0}^{2K-1}\left(\begin{array}{ccc}2K\\l\end{array}\right)
\frac{P\nabla}{\sqrt{-\triangle}}P^{\beta}_{2K-l+1}\left((-\triangle)^{\frac{1}{2}-\beta}
(I-e^{-s(-\triangle)^{\beta}})s^{\frac{l}{2\beta}}\nabla^{l}M(s,x)\right)\\
&+&\frac{P\nabla}{\sqrt{-\triangle}}P_{1}^{\beta}\left((-\triangle)^{\frac{1}{2}-\beta}
(I-e^{-s(-\triangle)^{\beta}})s^{\frac{2K}{2\beta}}\nabla^{2K}M(s,x)\right)\\
&+&\sum_{l=0}^{2K-1}\left(\begin{array}{ccc}2K\\l\end{array}\right)
\frac{P\nabla}{\sqrt{-\triangle}}P^{\beta}_{2K-l}\left((-\triangle)^{\frac{1}{2}-\beta}
(I-e^{-s(-\triangle)^{\beta}})s^{\frac{l+1}{2\beta}}\nabla^{l+1}M(s,x)\right)\\
&+&\frac{P\nabla}{\sqrt{-\triangle}}P_{0}^{\beta}\left((-\triangle)^{\frac{1}{2}-\beta}
(I-e^{-s(-\triangle)^{\beta}})s^{\frac{2K+1}{2\beta}}\nabla^{2K+1}M(s,x)\right).
\end{eqnarray*}
Since $\sup_{s\in (0,\infty)}s^{1-2\beta}(1-e^{-s^{2\beta}})<\infty$
for $\frac{1}{2}<\beta<1$, we can obtain that
$(-\triangle)^{1/2-\beta}(I-e^{-s(-\triangle)^{\beta}})$ is bounded
on $L^{2}$ with operator norm $\lesssim s^{1-\frac{1}{2\beta}}$. By
Lemma \ref{lemmma pq} and the $L^{2}-$boundedness of Riesz
transform, we have
\begin{eqnarray*}
&&r^{2\alpha-n+2\beta-2}\int_{0}^{r^{2\beta}}\int_{|x-x_{0}|<r}\left|\frac{P\nabla}{\sqrt{-\triangle}}P^{\beta}_{2K-l}(-\triangle)^{\frac{1}{2}-\beta}
(I-e^{-s(-\triangle)^{\beta}})s^{\frac{l+1}{2\beta}}\nabla^{l+1}M(s,x)\right|^{2}\frac{dxds}{s^{\alpha/\beta}}\\
&&\leq
p(2K-l)r^{2\alpha-n+2\beta-2}\int_{0}^{r^{2\beta}}\int_{\mathbb{R}^{n}}\left|(-\triangle)^{\frac{1}{2}-\beta}
(I-e^{-s(-\triangle)^{\beta}})s^{\frac{l+1}{2\beta}}\nabla^{l+1}M(s,x)\right|^{2}\frac{dxds}{s^{\alpha/\beta}}\\
&&\leq
p(2K-l)r^{2\alpha-n+2\beta-2}\int_{0}^{r^{2\beta}}\int_{|x-x_{0}|<r}
\left|s^{1-\frac{1}{2\beta}}s^{\frac{l+1}{2\beta}}\nabla^{l+1}M(s,x)\right|^{2}\frac{dxds}{s^{\alpha/\beta}}.
\end{eqnarray*}
Because $0<s<r^{2\beta}$ and
\begin{eqnarray*}
&&s^{\frac{l}{2\beta}+1}\nabla^{l+1}M(s,x)=s^{\frac{l}{2\beta}+1}\nabla^{l+1}(\phi_{R^{\frac{1}{2\beta}},x_{0}}u(s,x)\otimes
v(s,x))\\
&&=\sum_{m+\eta\leq
l+1}\left[s^{\frac{2\beta-1+m}{2\beta}}\nabla^{m}u(s,x)\right]\left[s^{\frac{\eta}{2\beta}}\nabla^{\eta}v(s,x)\right]
\left[s^{\frac{l}{2\beta}+1-\frac{2\beta-1+m}{2\beta}-\frac{\eta}{2\beta}}\nabla^{l+1-m-\eta}\phi_{R^{\frac{1}{2\beta}},x_{0}}\right],
\end{eqnarray*}
then we can get, taking $R=r^{2\beta}$,
$$r^{2\alpha-n+2\beta-2}\int_{0}^{r^{2\beta}}\int_{|x-x_{0}|<r}
\left|s^{1-\frac{1}{2\beta}}s^{\frac{l+1}{2\beta}}\nabla^{l+1}M(s,x)\right|^{2}\frac{dxds}{s^{\alpha/\beta}}\leq\sum_{m+\eta\leq
l+1}\|u\|_{N^{\beta,m}_{\alpha,\infty}}^{2}\|v\|^{2}_{N^{\beta,\eta}_{\alpha,
C}}.$$ In a similar way, we have
\begin{eqnarray*}
&&r^{2\alpha-n+2\beta-2}\int_{0}^{r^{2\beta}}\int_{|x-x_{0}|<r}\left|\frac{P\nabla}{\sqrt{-\triangle}}P_{0}^{\beta}
\left((-\triangle)^{\frac{1}{2}-\beta}
(I-e^{-s(-\triangle)^{\beta}})s^{\frac{2K+1}{2\beta}}\nabla^{2K+1}M(s,x)\right)\right|^{2}\frac{dxds}{s^{\alpha/\beta}}\\
&&\leq
p(0)r^{2\alpha-n+2\beta-2}\int_{0}^{r^{2\beta}}\int_{\mathbb{R}^{n}}\left|s^{1-\frac{1}{2\beta}}s^{\frac{2K+1}{2\beta}}\nabla^{2K+1}M(s,x)\right|^{2}
\frac{dxds}{s^{\alpha/\beta}}\\
&&\leq
p(0)r^{2\alpha-n+2\beta-2}\int_{0}^{r^{2\beta}}\int_{|x-x_{0}|<r}\left|s^{\frac{2K}{2\beta}+1}\sum_{m+\eta\leq
2K+1}\nabla^{m}u\nabla^{\eta}v\nabla^{2K+1-m-\eta}\phi_{R^{\frac{1}{2\beta}},x_{0}}\right|^{2}\frac{dxds}{s^{\alpha/\beta}}\\
&&\leq
p(0)\left(\|u\|^{2}_{N^{\beta,0}_{\alpha,\infty}}\|v\|^{2}_{N^{\beta,2K+1}_{\alpha,C}}+
\|v\|^{2}_{N^{\beta,0}_{\alpha,\infty}}\|u\|^{2}_{N^{\beta,2K+1}_{\alpha,C}}\right)
+r(K)\|u\|^{2}_{\widetilde{X}^{\beta,2K}_{\alpha}}\|v\|^{2}_{\widetilde{X}^{\beta,2K}_{\alpha}}.
\end{eqnarray*}
Similarly we can estimate the terms associated with $P^{\beta}_{1}$
and $P^{\beta}_{2K-l+1}$. Combining all the estimates together, we
can prove
\begin{eqnarray*}
&&\left(r^{2\alpha-n+2\beta-2}\int_{0}^{r^{2\beta}}\int_{|y-x_{0}|<r}\left|t^{\frac{k}{2\beta}}\nabla^{k}B^{1}_{2}(u,v)(t,x)\right|^{2}
\frac{dxdt}{t^{\alpha/\beta}}\right)^{1/2}\\
&&\leq
C_{1}\|u\|_{X^{\beta,0}_{\alpha}}\|v\|_{\widetilde{X}^{\beta,k}_{\alpha}}+C_{1}\|v\|_{X^{\beta,0}_{\alpha}}\|u\|_{\widetilde{X}^{\beta,k}_{\alpha}}
+C(k)\|u\|_{\widetilde{X}^{\beta,k-1}_{\alpha}}\|u\|_{\widetilde{X}^{\beta,k-1}_{\alpha}}.
\end{eqnarray*}
Now we estimate the term $B^{2}_{2}.$ Taking the change of
variables: $s=r^{2\beta}\theta$, $x=rz$ and $t=r^{2\beta}\tau,$ we
have
\begin{eqnarray*}
I&=&r^{2\alpha-n+2\beta-2}\int_{0}^{r^{2\beta}}\int_{|y-x_{0}|<r}\left|t^{\frac{k}{2\beta}}\nabla^{k}B^{2}_{2}(u,v)(t,x)\right|^{2}\frac{dxdt}{t^{\alpha/\beta}}\\
&=&r^{2\alpha-n+2\beta-2}\int_{0}^{r^{2\beta}}\int_{|y-x_{0}|<r}\left|t^{\frac{k}{2\beta}}\nabla^{k}
\frac{P\nabla}{\sqrt{-\triangle}}\sqrt{-\triangle}e^{-t(-\triangle)^{\beta}}\int_{0}^{t}M(s,x)ds\right|^{2}\frac{dxdt}{t^{\alpha/\beta}}\\
&\leq&
r^{2\alpha-n+2\beta-2}\int_{0}^{r^{2\beta}}\int_{\mathbb{R}^{n}}\left|t^{\frac{k}{2\beta}}\nabla^{k+1}
e^{-t(-\triangle)^{\beta}}\int_{0}^{t}M(s,x)ds\right|^{2}\frac{dxdt}{t^{\alpha/\beta}}\\
&=&r^{2\alpha-n+2\beta-2}\int_{0}^{1}\int_{\mathbb{R}^{n}}\left|(\tau
r^{2\beta})^{\frac{k}{2\beta}}\frac{1}{r^{k+1}}\nabla_{z}^{k+1}e^{-\tau(-\triangle_{z})^{\beta}}
\int^{\tau}_{0}M(r^{2\beta}\theta,rz)r^{2\beta}d\theta\right|^{2}
\frac{r^{n+2\beta}dzd\tau}{r^{2\alpha}\tau^{\alpha/\beta}}\\
&=&r^{8\beta-4}\int_{0}^{1}\int_{\mathbb{R}^{n}}\left|\tau^{k/2\beta}\nabla_{z}^{k+1}e^{-\tau(-\triangle_{z})^{\beta}}\int^{\tau}_{0}M(r^{2\beta}\theta,rz)
\right|^{2}\frac{dzd\tau}{\tau^{\alpha/\beta}}.
\end{eqnarray*}
Denote by $\nabla_{z}^{\nu}e^{-\tau(-\triangle_{z})^{\beta/2}}(x,y)$
the kernel of the operator
$\nabla_{z}^{\nu}e^{-\tau(-\triangle_{z})^{\beta}/2}, \nu>0.$
Because $\frac{1}{2}<\beta\leq 1$, we have
$$\left|\tau^{\frac{k(1-\beta)}{2\beta}}\nabla_{z}^{k(1-\beta)}e^{-\tau(-\triangle_{z})^{\beta}/2}(x,y)\right|\lesssim\tau^{\frac{k(1-\beta)}{2\beta}}
\frac{1}{(\tau/2)^{\frac{k(1-\beta)+n}{2\beta}}}\frac{1}{(1+\frac{|x-y|}{\tau^{1/2\beta}})^{n+k(1-\beta)}}\in
L^{1}(\mathbb{R}^{n})$$ uniformly in $\tau.$ By Young's inequality
and Lemma \ref{le5}, we have
\begin{eqnarray*}
I&=&r^{8\beta-4}\int_{0}^{1}\int_{\mathbb{R}^{n}}\left|\tau^{\frac{k(1-\beta)}{2\beta}}\nabla_{z}^{k(1-\beta)}
e^{-\tau(-\triangle_{z})^{\beta}/2}\tau^{\frac{k}{2}}\nabla^{k\beta+1}_{z}e^{-\tau(-\triangle_{z})^{\beta}/2}\int^{\tau}_{0}M(r^{2\beta}\theta,rz)
\right|^{2}\frac{dzd\tau}{\tau^{\alpha/\beta}}\\
&\leq&r^{8\beta-4}\int_{0}^{1}\int_{\mathbb{R}^{n}}\left|\tau^{\frac{k}{2}}\nabla^{k\beta+1}_{z}
e^{-\tau(-\triangle_{z})^{\beta}/2}\int^{\tau}_{0}M(r^{2\beta}\theta,rz)
\right|^{2}\frac{dzd\tau}{\tau^{\alpha/\beta}}\\
&\leq&r^{8\beta-4}b(k)A(\alpha,\beta,M)\int_{0}^{1}\int_{\mathbb{R}^{n}}|M(r^{2\beta}\theta,rz)|\frac{dzd\theta}{\theta^{\alpha/\beta}}\\
&:=&r^{8\beta-4}b(k)A(\alpha,\beta,M)I_{M}.
\end{eqnarray*}
For $A(\alpha,\beta,M),$ we have
\begin{eqnarray*}
A(\alpha,\beta,
M)&=&\rho^{2\alpha-n+2\beta-2}\int_{0}^{\rho^{2\beta}}\int_{|y-x|<\rho}\left|M(r^{2\beta}s,ry)\right|\frac{ds
dy}{s^{\alpha/\beta}}\\
&\leq&r^{2-4\beta}(r\rho)^{2\alpha-n+2\beta-2}\int_{0}^{(r\rho)^{2\beta}}\int_{|z-rx|<r\rho}|M(t,z)|\frac{dzdt}{t^{\alpha/\beta}}\\
&\leq&r^{2-4\beta}\|u\|_{X^{\beta,0}_{\alpha}}\|v\|_{X^{\beta,0}_{\alpha}}.
\end{eqnarray*}
For $I_{M}$, we have
\begin{eqnarray*}
\int_{0}^{1}\int_{\mathbb{R}^{n}}|M(r^{2\beta}\theta,rz)|\frac{dzd\theta}{\theta^{\alpha/\beta}}
&\leq&\int_{0}^{r^{2\beta}}\int_{\mathbb{R}^{n}}|M(t,z)|\frac{r^{-2\beta-n}dtdz}{r^{-2\alpha}t^{\alpha/\beta}}\leq
r^{2-4\beta}\|u\|_{X^{\beta,0}_{\alpha}}\|v\|_{X^{\beta,0}_{\alpha}}.
\end{eqnarray*}
Then we get
$$r^{2\alpha-n+2\beta-2}\int_{0}^{r^{2\beta}}\int_{|y-x_{0}|<r}\left|t^{\frac{k}{2\beta}}\nabla^{k}B^{2}_{2}(u,v)(t,x)\right|^{2}\frac{dxdt}{t^{\alpha/\beta}}
\leq
b(k)\|u\|^{2}_{X^{\beta,0}_{\alpha}}\|v\|^{2}_{X^{\beta,0}_{\alpha}}.$$
Now we have proved that
\begin{eqnarray*}
\|B(u,v)\|_{X^{\beta,k}_{\alpha}}&\leq&
C_{0}(k)\|u\|_{X^{\beta,0}_{\alpha}}\|v\|_{X^{\beta,0}_{\alpha}}+C(k)\|u\|_{\widetilde{X}^{\beta,k-1}_{\alpha}}
\|v\|_{\widetilde{X}^{\beta,k-1}_{\alpha}}\\
&+&C_{1}\|u\|_{X^{\beta,0}_{\alpha}}\|v\|_{X^{\beta,k}_{\alpha}}
+C_{1}\|u\|_{X^{\beta,k}_{\alpha}}\|v\|_{X^{\beta,0}_{\alpha}}.
\end{eqnarray*}
Similar to the method applied in Lemma 4.3 of \cite{P. Germain N.
Pavlovic G.Staffilani}, if we construct the approximating sequence
$u^{j}$ by
$$u^{-1}=0,\ u^{0}=e^{-t(-\triangle)^{\beta}}u_{0},\ u^{j+1}=u^{0}+B(u^{j},u^{j}),$$
 we can get the following lemma and hence complete the proof of
Theorem \ref{threg}.
\begin{lemma}\label{le6}
 Let $\alpha>0$ and $\max\{\frac{1}{2},\alpha\}<\beta<1$ with $\alpha+\beta-1\geq0$. Suppose $u_{0}$ be small enough in $Q^{\beta,-1}_{\alpha;\infty}.$
Then for any $k\geq 0$, there exist constants $D_{k}$ and $E_{k}$
such that
$$
\|u^{j}\|_{\widetilde{X}_{\alpha}^{\beta,k}}\lesssim D_{k}\quad
\text{and} \quad
\|u^{j+1}-u^{j}\|_{\widetilde{X}_{\alpha}^{\beta,k}}\lesssim
E_{k}(\frac{2}{3})^{j}.
$$
 In particular, for any $k\geq
0$, $u^{j}$ converges in $\widetilde{X}^{\beta,k}_{\alpha}.$
\end{lemma}
\end{proof}

\end{document}